\documentclass[11pt]{amsart}

\usepackage[margin= 60 pt]{geometry}
\usepackage{graphicx} 
\usepackage{amsmath,amssymb,amsthm,bm}
\usepackage{ytableau}
\usepackage{enumitem}
\usepackage{tikz}
\usetikzlibrary{arrows.meta}
\usepackage[colorlinks=true, citecolor=blue]{hyperref}
\usepackage{cleveref}
\usepackage{caption}
\usepackage{subcaption}
\usepackage[T1]{fontenc}
\usepackage{tgtermes}
\usetikzlibrary{positioning,arrows.meta}
\usepackage[all]{xy}
\usepackage{booktabs}
\usepackage{makecell}
\numberwithin{equation}{section}
\newtheorem{theorem}{Theorem}[section] 
\newtheorem{proposition}[theorem]{Proposition} 
 
\newtheorem{corollary}[theorem]{Corollary} 
\newtheorem{lemma}[theorem]{Lemma}

\theoremstyle{definition}
\newtheorem{definition}[theorem]{Definition}
\newtheorem{remark}[theorem]{Remark} 
\newtheorem{example}[theorem]{Example}

\newcommand\jj{\mathbf{j}}
\newcommand\mm{\mathbf{m}}
\newcommand\con{\operatorname{con}}
\newcommand\Con{\operatorname{Con}}
\newcommand\ba{{\bm{\alpha}}}
\newcommand\wt{\operatorname{wt}}

\newcommand{\col}{\operatorname{Col}}
\newcommand{\row}{\operatorname{Row}}

\newcommand{\sTam}{\operatorname{sTam}}
\newcommand{\esTam}{\operatorname{esTam}}
\newcommand{\Tam}{\operatorname{Tam}}

\newcommand{\posfb}{\operatorname{esTam}}

\definecolor{darkblue}{rgb}{0.0,0,0.7} 
\definecolor{darkred}{rgb}{0.7,0,0} 

\newcommand\red[1]{\textcolor{red}{\bf #1}}
\newcommand\qand{\quad\text{and}\quad}

\renewcommand\emph[1]{\textcolor{darkblue}{\it #1}}

\newcommand{\darkred}{\color{darkred}} 
\newcommand{\defn}[1]{\emph{\darkred #1}}

\def\noteson{\gdef\baptiste##1{\noindent{\color{orange}[Baptiste: ##1]}}\gdef\jihyeug##1{\noindent{\color{violet}[Jihyeug: ##1]}}\gdef\sylvie##1{\noindent{\color{purple}[Sylvie: ##1]}}}
\noteson

\title{The extra slow Tamari lattice}

\author{Sylvie Corteel}
\address{CNRS, Institut de Mathématiques de Jussieu-Paris Rive Gauche, Sorbonne Université, Paris, France. Paris, France and Department of Mathematics, University of California Berkeley, USA.}
\email{corteel@berkeley.edu}

\author{Jihyeug Jang}
\address{Université de Genève, 7--9, rue Conseil Général, 1205 Genève, Switzerland}
\email{jihyeug.jang@unige.ch}

\author{Baptiste Rognerud}
\address{Institut de Mathématiques de Jussieu-Paris Rive Gauche, Université Paris Cité, Paris, France.}
\email{baptiste.rognerud@imj-prg.fr}

\begin{document}

\begin{abstract}
  We introduce the extra slow Tamari lattices, a new family of
  lattices defined on faithfully balanced tableaux. These tableaux
  arise naturally from the representation theory of type \( A \)
  quivers, and our construction extends the classical Tamari lattice
  and the slow Tamari lattice.

  We explicitly describe meets and joins in the extra slow Tamari
  lattices, and then prove that they are lattices. We then show that
  they are semidistributive, trim, polygonal, and congruence uniform.
  Their join-irreducible elements are described in terms of a
  three-color analogue of the positive roots of type \( A \), which
  leads to descriptions of their spines and congruence lattices. We
  also obtain several enumerative results for the extra slow Tamari
  lattices and their associated structures. Finally, we derive new
  structural and enumerative results for the slow Tamari lattices.
\end{abstract}

\maketitle
\tableofcontents

\section{Introduction}

\defn{Tamari lattices} \( \Tam_n \) are partial orders on Catalan sets
introduced by Tamari in his thesis. They are now classical and
well-studied objects in combinatorics, and they appear in a remarkable
number of recent developments in mathematics such as algebra, computer
science, category theory, geometry, topology. There are already many
generalizations of Tamari lattices: Cambrian lattices
\cite{zbMATH05056850}, framing lattices \cite{arXiv:2512.20575},
$\nu$-Tamari lattices \cite{nutamari}, s-weak order
\cite{zbMATH07949995}, torsion classes \cite{tamari_torsion,DIRRT},
transfer systems \cite{zbMATH07480514,zbMATH08165348} and many others.
The aim of this article is to present another generalization of a
slightly different flavor. Instead of producing a large family of
lattices that contains the Tamari lattices, we introduce a family of
partial orders, indexed by the positive integers, which we call the
\defn{extra slow Tamari lattices}. These new partial orders arise from
the representation theory of finite-dimensional algebras and finite
quivers, but their definition is purely combinatorial. We view them as
a generalization of Tamari lattices for two reasons. On the one hand,
they are defined on a class of objects that contains a subclass in
bijection with binary trees, and the order relation on these objects
naturally corresponds to the rotation of the binary trees. On the
other hand, these new partial orders behave like a three-color version
of Tamari lattices, as will become clearer later in the introduction.

In the representation theory of quivers, the Tamari lattices have two
classical interpretations. The best known is the one related to the
theory of cluster algebras, in which the Tamari lattices appear as
particular examples of \defn{Cambrian lattices} of type $A$.
Alternatively, the Tamari lattices appear as lattices of \defn{torsion
  classes} \cite{tamari_torsion}, or equivalently as lattices of
support $\tau$-tilting modules in the sense of \cite{tau_tilting}. On
the other hand, the Tamari lattices are also the lattices of
\defn{tilting modules} over equioriented quivers of type \( A \); see
\cite{buan_krause} and also \cite{hille}. The first interpretation is
very close to the description of the Tamari lattices using bracket
vectors. The second is closer to the one using binary trees.

It is with this second interpretation that our article begins. The
tilting modules are known to be part of the much larger class of
\defn{faithfully balanced modules}, also called modules with the
\defn{double centralizer property} (see \cite{sauter} for a recent
summary on these modules). They were recently classified for the
so-called \defn{Nakayama algebras} in \cite{CMRS2021}. For the path
algebra of an equioriented quiver of type \(A\), a faithfully balanced
module corresponds to a partial filling of the cells of a Ferrers
diagram of shape \( (n,n,n-1, \dots, 3,2)\) with dots such that:
\begin{enumerate}
\item The top-left cell and the border cells contain dots.
\item If a cell, which is not the top-left cell, contains a dot, then
  there is a cell containing a dot on its left in the same row or
  above it in the same column.
\end{enumerate}
The border cells, highlighted in gray in our figures, are the
rightmost cells of each row (the top row has no border cell). In this
article, we call such a filling a \defn{fb-tableau} of size \( n \).
It is easy to see that a fb-tableau of size \( n \) contains at least
$n$ dots outside the border cells, and we say that a fb-tableau is
\defn{small} if it contains exactly $n$ such dots. The tilting modules
correspond to fb-tableaux without a free cell, where a cell is said to
be free if it lies strictly below a dot and strictly to the left of a
dot. Such a tableau without a free cell is called a \defn{binary}
fb-tableau. Binary fb-tableaux are small, and they are in immediate
bijection with binary trees. There are $n!$ small fb-tableaux of size
\( n \) and they are in simple bijection with \defn{tree-like
  tableaux} \cite{CMRS2021,tree_like} and alternative tableaux
\cite{Vi}. There are many more fb-tableaux as they are counted by
\([n]_2!:=\prod_{i=1}^{n}(2^i - 1)\).

To explain our construction and emphasize that it is a natural
generalization of the Tamari lattice, we begin by transporting the
usual order on the binary trees (or tilting modules) to an order on
binary fb-tableaux. For a fb-tableau \(T\), we define the \defn{row
  space} \(\row(T)\) of \( T \) to be the set of cells that lie weakly
to the right of the dots. Similarly, we define the \defn{column space}
\(\col(T)\) of \( T \) to be the set of cells that lie weakly below
the dots. Let \(\leq\) be the order on binary fb-tableaux defined by
\(T_1\leq T_2\) if and only if \(\col(T_1)\subseteq \col(T_2)\), or
equivalently \(\row(T_2)\subseteq\row (T_1)\). It is easy to see that
this order is isomorphic to the Tamari lattice; for the reader's
convenience, we give a direct proof in \Cref{pro:tamari}. A binary
fb-tableau is always small. However, a small tableau can have empty
free cells. As a consequence, the inclusion of the column spaces is no
longer a partial order on the set of small tableaux. This already
occurs for \(n=3\), as shown in the following example:
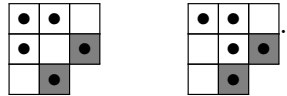
\[
  \ytableausetup{boxsize = 1 em}
\begin{ytableau}
    \bullet &\bullet & \\
    \bullet & &*(gray)\bullet\\
    &*(gray)\bullet\\
  \end{ytableau}
  \quad \quad \quad
\begin{ytableau}
    \bullet &\bullet & \\
    &\bullet &*(gray)\bullet\\
    &*(gray)\bullet\\
  \end{ytableau}.
\]
To get around this problem, a partial order on the small tableaux was
introduced in \cite[Section 9]{CMRS2021} by setting $T_1\leq T_2$ if
and only if $\col(T_1)\subseteq \col(T_2)$ and
$\row(T_2)\subseteq \row(T_1)$. It was also proved in the same article
that this is in fact a lattice. This poset on small tableaux is a
first generalization of the Tamari lattice (which behaves like a
two-color version of the Tamari lattice), and we propose to call it
the \defn{slow Tamari lattice}, denoted by \( \sTam_n \).

For general fb-tableaux, free cells may contain dots. Hence, inclusion
of column spaces and row spaces is not an antisymmetric relation,
because it does not provide any constraints on the free cells. For two
fb-tableaux $T_1$ and $T_2$, we set $T_1\lhd T_2$ if
\begin{enumerate}
\item $\col(T_1)\subseteq \col(T_2)$,
\item $\row(T_2) \subseteq \row(T_1)$,
\item Let $c$ be a cell that is free in both $T_1$ and $T_2$. If $c$ contains a dot in $T_1$, then it also contains a dot in $T_2$. 
\end{enumerate}

We now state our first main result.
\begin{theorem}[\Cref{thm:selfdual,thm:estam_lattice}]
  The set of fb-tableaux of size \(n\) endowed with the relation
  \(\lhd\) is a self-dual lattice.
\end{theorem}
We propose to call it the \defn{extra slow Tamari lattice}, denoted by
\( \esTam_n \). Almost by construction, the Tamari lattice on binary
fb-tableaux forms an induced subposet of the slow Tamari lattice,
which itself is an induced subposet of the extra-slow Tamari lattice.
In \Cref{sec:sublattices}, we prove a stronger result.

\begin{proposition}[\Cref{prop:sublattices,pro:tamari}]\label{pro:sblattice} ~
  \begin{enumerate}
  \item The slow Tamari lattice is a sublattice of the extra slow Tamari
    lattice.
  \item The poset of binary fb-tableaux is a sublattice of the slow
    Tamari lattice. Moreover, it is isomorphic to the Tamari lattice.
  \end{enumerate}
\end{proposition}

The smallest interesting Tamari lattice is the one with $5$ elements,
and the smallest interesting slow Tamari lattice is the one with $6$
elements (see \Cref{fig:slow_tam}). We see that going from the Tamari
lattice to the slow Tamari lattice amounts to replacing the cover
relation on the longer side of the pentagon by a chain of two cover
relations. In other words, when going from binary tableaux to small
tableaux, we `slow down' the usual rotation of the binary trees,
hence the name of slow Tamari lattice.
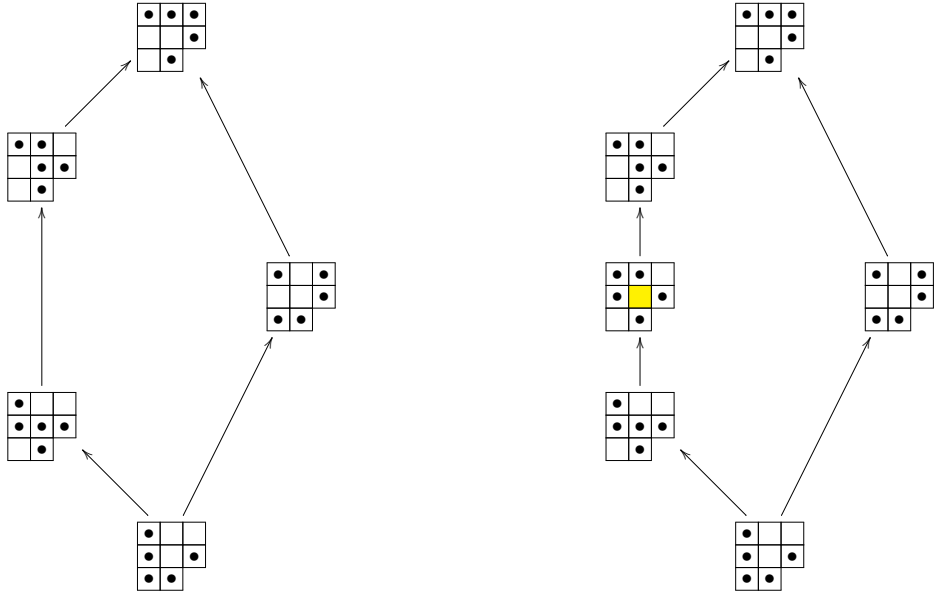
\begin{figure}
  \centering
  \ytableausetup{boxsize = 1 em}

  \begin{minipage}{0.45\textwidth}
    \centering
    \scalebox{0.75}{
      \xymatrix{
        &
        {\begin{ytableau}
            \bullet& \bullet & \bullet \\
            & & \bullet\\
            &\bullet
          \end{ytableau}}
        \ar@{<-}[ddr]\ar@{<-}[dl]
        &&&& \\
        {\begin{ytableau}
            \bullet& \bullet & \\
            & \bullet& \bullet\\
            &\bullet
          \end{ytableau}}
        \ar@{<-}[dd]
        &&&& \\
        && {\begin{ytableau}
            \bullet& &\bullet \\
            & & \bullet\\
            \bullet &\bullet
          \end{ytableau}}
        \ar@{<-}[ddl]
        && \\
        {\begin{ytableau}
            \bullet& & \\
            \bullet& \bullet& \bullet\\
            &\bullet
          \end{ytableau}}
        &&&& \\
        &{\begin{ytableau}
            \bullet& & \\
            \bullet&& \bullet\\
            \bullet&\bullet
          \end{ytableau}}
        \ar[ul]
        &&&&
      }
    }
  \end{minipage}
  \begin{minipage}{0.45\textwidth}
    \centering
    \scalebox{0.75}{
      \xymatrix{
        &
        {\begin{ytableau}
            \bullet& \bullet & \bullet \\
            & & \bullet\\
            &\bullet
          \end{ytableau}}
        \ar@{<-}[ddr]\ar@{<-}[dl]
        &&&& \\
        {\begin{ytableau}
            \bullet& \bullet & \\
            & \bullet& \bullet\\
            &\bullet
          \end{ytableau}}
        \ar@{<-}[d]
        &&&& \\
        {\begin{ytableau}
            \bullet&\bullet & \\
            \bullet&*(yellow) & \bullet\\
            &\bullet
          \end{ytableau}}
        \ar@{<-}[d]
        && {\begin{ytableau}
            \bullet& &\bullet \\
            & & \bullet\\
            \bullet &\bullet
          \end{ytableau}}
        \ar@{<-}[ddl]
        && \\
        {\begin{ytableau}
            \bullet& & \\
            \bullet& \bullet& \bullet\\
            &\bullet
          \end{ytableau}}
        &&&& \\
        &{\begin{ytableau}
            \bullet& & \\
            \bullet&& \bullet\\
            \bullet&\bullet
          \end{ytableau}}
        \ar[ul]
        &&&&
      }
    }
  \end{minipage}
  \caption{The Tamari lattice \( \Tam_3 \) and the slow Tamari lattice \(\sTam_3\) of size \( 3 \). The tableau with the free cell highlighted in yellow is the only small tableau that is not a binary fb-tableau. The sublattice consisting of the tableaux without free cells is isomorphic to the Tamari lattice \(\Tam_3\).}
  \label{fig:slow_tam}
\end{figure}

This idea continues to be the right one when moving from the slow
Tamari lattice to the extra slow Tamari lattice. Roughly speaking,
one local change is the replacement of the chain
\[
  \begin{ytableau}
    \bullet& \\
    \bullet&\bullet
  \end{ytableau}
  \rightarrow
  \begin{ytableau}
    \bullet&\bullet \\
    \bullet&
  \end{ytableau}
  \rightarrow
  \begin{ytableau}
    \bullet&\bullet \\
    &\bullet
  \end{ytableau}
  \quad \text{by} \quad
  \begin{ytableau}
    \bullet& \\
    \bullet&\bullet
  \end{ytableau}
  \rightarrow
  \begin{ytableau}
    \bullet&\bullet \\
    \bullet&
  \end{ytableau}
  \rightarrow
  \begin{ytableau}
    \bullet&\bullet \\
    \bullet&\bullet
  \end{ytableau}
  \rightarrow
  \begin{ytableau}
    \bullet&\bullet \\
    &\bullet
  \end{ytableau}.
\]
However, this local modification leads to a much larger and more
complex lattice, as shown in \Cref{fig:extra_slow}. In the rest of the
article, we study the lattice properties of the extra slow Tamari
lattice. Using \Cref{pro:sblattice}, we obtain similar results for the
slow Tamari lattice and the Tamari lattice. While these results are
well known for the Tamari lattice, they appear to be new even for the
slow Tamari lattice.

\ytableausetup{boxsize = 1 em}
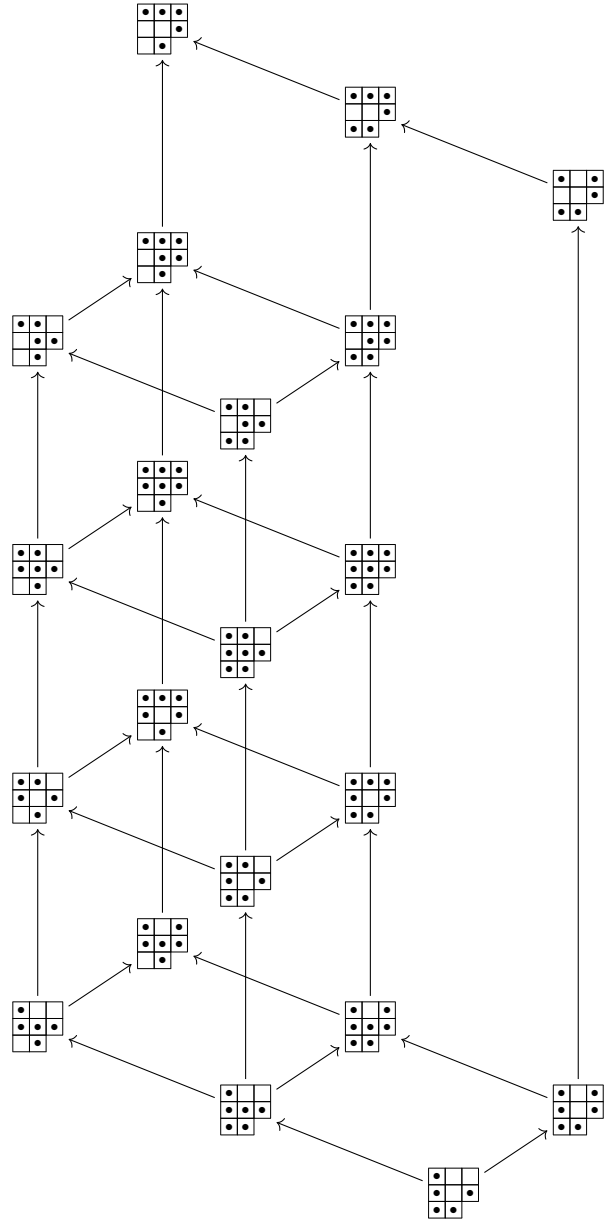
\begin{figure}
  \centering
  \begin{tikzpicture}[scale=.55, transform shape]
    \def\wx{5}
    \def\wy{2}
    \def\h{5.5}
    \def\dx{3}
    \def\dy{2}

    \node (v1) at (0,0) {
      $\begin{ytableau}
        \bullet&& \\
        \bullet&&\bullet \\
        \bullet&\bullet 
      \end{ytableau}$
    };
    \node (v2) at (-\wx,\wy) {
      $\begin{ytableau}
        \bullet&& \\
        \bullet&\bullet&\bullet \\
        \bullet&\bullet 
      \end{ytableau}$
    };
    \node (v3) at (-2*\wx,2*\wy) {
      $\begin{ytableau}
        \bullet&& \\
        \bullet&\bullet&\bullet \\
        &\bullet 
      \end{ytableau}$
    };
    \node (v4) at (\dx,\dy) {
      $\begin{ytableau}
        \bullet&&\bullet \\
        \bullet&&\bullet \\
        \bullet&\bullet 
      \end{ytableau}$
    };
    \node (v5) at (-\wx+\dx,\wy+\dy) {
      $\begin{ytableau}
        \bullet&&\bullet \\
        \bullet&\bullet&\bullet \\
        \bullet&\bullet 
      \end{ytableau}$
    };
    \node (v6) at (-2*\wx+\dx,2*\wy+\dy) {
      $\begin{ytableau}
        \bullet&&\bullet \\
        \bullet&\bullet&\bullet \\
        &\bullet 
      \end{ytableau}$
    };
    \node (v22) at (-\wx,\wy+\h) {
      $\begin{ytableau}
        \bullet&\bullet& \\
        \bullet&&\bullet \\
        \bullet&\bullet 
      \end{ytableau}$
    };
    \node (v32) at (-2*\wx,2*\wy+\h) {
      $\begin{ytableau}
        \bullet&\bullet& \\
        \bullet&&\bullet \\
        &\bullet 
      \end{ytableau}$
    };
    \node (v52) at (-\wx+\dx,\wy+\dy+\h) {
      $\begin{ytableau}
        \bullet&\bullet&\bullet \\
        \bullet&&\bullet \\
        \bullet&\bullet 
      \end{ytableau}$
    };
    \node (v62) at (-2*\wx+\dx,2*\wy+\dy+\h) {
      $\begin{ytableau}
        \bullet&\bullet&\bullet \\
        \bullet&&\bullet \\
        &\bullet 
      \end{ytableau}$
    };
    \node (v23) at (-\wx,\wy+2*\h) {
      $\begin{ytableau}
        \bullet&\bullet& \\
        \bullet&\bullet&\bullet \\
        \bullet&\bullet 
      \end{ytableau}$
    };
    \node (v33) at (-2*\wx,2*\wy+2*\h) {
      $\begin{ytableau}
        \bullet&\bullet& \\
        \bullet&\bullet&\bullet \\
        &\bullet 
      \end{ytableau}$
    };
    \node (v53) at (-\wx+\dx,\wy+\dy+2*\h) {
      $\begin{ytableau}
        \bullet&\bullet&\bullet \\
        \bullet&\bullet&\bullet \\
        \bullet&\bullet 
      \end{ytableau}$
    };
    \node (v63) at (-2*\wx+\dx,2*\wy+\dy+2*\h) {
      $\begin{ytableau}
        \bullet&\bullet&\bullet \\
        \bullet&\bullet&\bullet \\
        &\bullet 
      \end{ytableau}$
    };
    \node (v24) at (-\wx,\wy+3*\h) {
      $\begin{ytableau}
        \bullet&\bullet& \\
        &\bullet&\bullet \\
        \bullet&\bullet 
      \end{ytableau}$
    };
    \node (v34) at (-2*\wx,2*\wy+3*\h) {
      $\begin{ytableau}
        \bullet&\bullet& \\
        &\bullet&\bullet \\
        &\bullet 
      \end{ytableau}$
    };
    \node (v54) at (-\wx+\dx,\wy+\dy+3*\h) {
      $\begin{ytableau}
        \bullet&\bullet&\bullet \\
        &\bullet&\bullet \\
        \bullet&\bullet 
      \end{ytableau}$
    };
    \node (v64) at (-2*\wx+\dx,2*\wy+\dy+3*\h) {
      $\begin{ytableau}
        \bullet&\bullet&\bullet \\
        &\bullet&\bullet \\
        &\bullet 
      \end{ytableau}$
    };
    \node (v45) at (\dx,\dy+4*\h) {
      $\begin{ytableau}
        \bullet&&\bullet \\
        &&\bullet \\
        \bullet&\bullet 
      \end{ytableau}$
    };
    \node (v55) at (-\wx+\dx,\wy+\dy+4*\h) {
      $\begin{ytableau}
        \bullet&\bullet&\bullet \\
        &&\bullet \\
        \bullet&\bullet 
      \end{ytableau}$
    };
    \node (v65) at (-2*\wx+\dx,2*\wy+\dy+4*\h) {
      $\begin{ytableau}
        \bullet&\bullet&\bullet \\
        &&\bullet \\
        &\bullet 
      \end{ytableau}$
    };

    \draw[->] (v1)  -- node[below]{} (v2);
    \draw[->] (v2)  -- node[below]{} (v3);
    \draw[->] (v4)  -- node[below]{} (v5);
    \draw[->] (v5)  -- node[below]{} (v6);
    \draw[->] (v1)  -- node[below]{} (v4);
    \draw[->] (v2)  -- node[below]{} (v5);
    \draw[->] (v3)  -- node[below]{} (v6);

    \draw[->] (v22)  -- node[below]{} (v32);
    \draw[->] (v22)  -- node[below]{} (v52);
    \draw[->] (v52)  -- node[below]{} (v62);
    \draw[->] (v32)  -- node[below]{} (v62);
    \draw[->] (v23)  -- node[below]{} (v33);
    \draw[->] (v23)  -- node[below]{} (v53);
    \draw[->] (v53)  -- node[above]{} (v63);
    \draw[->] (v33)  -- node[below]{} (v63);
    \draw[->] (v24)  -- node[below]{} (v34);
    \draw[->] (v24)  -- node[below]{} (v54);
    \draw[->] (v54)  -- node[below]{} (v64);
    \draw[->] (v34)  -- node[below]{} (v64);

    \draw[->] (v45)  -- node[below]{} (v55);
    \draw[->] (v55)  -- node[below]{} (v65);
    \draw[->] (v2)  -- node[left]{} (v22);
    \draw[->] (v22)  -- node[left]{} (v23);
    \draw[->] (v23)  -- node[left]{} (v24);
    \draw[->] (v3)  -- node[left]{} (v32);
    \draw[->] (v32)  -- node[left]{} (v33);
    \draw[->] (v33)  -- node[left]{} (v34);
    \draw[->] (v5)  -- node[left]{} (v52);
    \draw[->] (v52)  -- node[left]{} (v53);
    \draw[->] (v53)  -- node[left]{} (v54);
    \draw[->] (v54)  -- node[left]{} (v55);
    \draw[->] (v6)  -- node[right]{} (v62);
    \draw[->] (v62)  -- node[left]{} (v63);
    \draw[->] (v63)  -- node[right]{} (v64);
    \draw[->] (v64)  -- node[left]{} (v65);
    \draw[->] (v4)  -- node[left]{} (v45);
  \end{tikzpicture}

  \caption{The extra slow Tamari lattice $\esTam_3$ of size \( 3 \).}
  \label{fig:extra_slow}
\end{figure}

In \Cref{sec:covers}, we give a description of the cover relations in
the extra slow Tamari lattice. The cover relations are of three types:
those given by filling an empty free cell, those in which a dot moves
upward, and those in which a dot moves to another dot on its right,
resulting in a fusion. For a precise statement we refer to
\Cref{def:cover_relation}. This result is used in \Cref{sec:irred} to
describe the irreducible elements. We prove in \Cref{cor:meet_irr}
that there are $2{n\choose 2}+{n-1\choose 2}$ meet-irreducible
elements, and since the extra slow Tamari lattice is self-dual by
\Cref{thm:selfdual}, this is also the number of join-irreducible
elements. For comparison, the number of join-irreducible elements of
the Tamari lattice $\Tam_n$ is ${n \choose 2}$ and they are in
bijection with the positive roots of the root system of type $A_n$. In
other words, they are in bijection with intervals $[i,j]$ for
$1\leq i\leq j\leq n-1$. For the extra slow Tamari lattice we have the
following result:
\begin{proposition}
  The join irreducible elements of the extra slow Tamari lattice on
  fb-tableaux of size $n$ are in bijection with the set consisting of
  \begin{enumerate}
  \item \([i,j]_{\mathtt{red}}\) for \(1\leq i\leq j \leq n-1\);
  \item \([i,j]_{\mathtt{green}} \) for \(1\leq i <j \leq n-1\);
  \item \( [i,j]_{\mathtt{blue}}\) for \(1\leq i <j \leq n-1\).
  \end{enumerate}  
\end{proposition}
We refer to \Cref{cor:join-irr-description} and \Cref{def:join-irr}
for a more precise statement and to \Cref{fig:reflect_J_irr_poset} for
an enlightening picture.

In
\Cref{sec:semidistributivity,sec:trim,sec:polygonality,sec:lattice-congruences},
we investigate the lattice-theoretic properties of the extra slow
Tamari lattice. We summarize the results in the following theorem.
\begin{theorem}\label{thm:intro}
  The extra slow Tamari lattice is
  \begin{enumerate}
  \item semidistributive (\Cref{thm:semidistributive});
  \item extremal (\Cref{prop:extremal});
  \item trim (\Cref{thm:trim});
  \item polygonal with quadrilaterals and heptagons (\Cref{prop:polygonal});
  \item congruence uniform (\Cref{thm:HH_uniform}).
  \end{enumerate}
\end{theorem}

To prove congruence uniformity, we show that the extra slow Tamari
lattice is an HH-lattice in the sense of \cite{CLM2004}. This requires
the existence of an edge labeling satisfying suitable conditions. We
refer to \Cref{sec:poly-label} for a precise statement. This labeling
can be thought of as a three-color version of the usual \defn{brick
  labeling} of \cite{DIRRT}. We then give a complete description of
its lattice of congruences in~\Cref{sec:lattice-congruences}. To
explain our result, we continue the analogy with the Tamari lattice.
By a result of Geyer \cite{geyer} (see also \cite{rognerud_stors}),
the lattice of congruences of the Tamari lattice is isomorphic to the
\defn{Stanley lattice} of Dyck paths, or equivalently to the lattice
of order ideals of the poset of positive roots of type \(A\). For the
extra slow Tamari lattice, our result is very similar once we take the
colors into account.
\begin{theorem}[\Cref{thm:con_jirr}]
The congruence lattice of the extra slow Tamari lattice is isomorphic to the lattice
  of order ideals on (the opposite of) the inclusion poset of colored
  intervals.
\end{theorem}

Trim lattices have a remarkable sublattice called the \defn{spine}.
This is an important sublattice: it is distributive and contains the
modular elements of the lattice. Since it is a distributive lattice,
the spine is in general much easier to understand than the rest of the
trim lattice. In~\Cref{sec:trim}, we describe the fb-tableaux in the
spine of the extra slow Tamari lattice in terms of tableaux avoiding a
specific configuration that we call a \defn{forbidden cohook}. Since
the spine is a distributive lattice, it is isomorphic to the lattice
of order ideals of its poset of join-irreducible elements. For the
Tamari lattice, this poset is isomorphic to the poset of intervals of
a total order, endowed with the product order. For the extra slow
Tamari lattice, we obtain the analogous statement once the colors are
taken into account.

\begin{theorem}[\Cref{thm:spine_colored_intervals}]
  The spine of the extra slow Tamari lattice is isomorphic to the
  lattice of order ideals on (the opposite of) the product poset of
  colored intervals.
\end{theorem}

In \Cref{sec:slow_Tamari}, we revisit the \defn{slow Tamari lattice} and we
prove the following results.

\begin{theorem}
The slow Tamari lattice is
\begin{enumerate}
\item semidistributive (\Cref{lem:sTam_trim});
\item extremal (\Cref{lem:sTam_trim});
\item trim (\Cref{lem:sTam_trim});
\item polygonal with quadrilaterals and hexagons (\Cref{lem:1});
\item congruence uniform (\Cref{lem:2}).
\end{enumerate}
\end{theorem}

Since the slow Tamari lattice is a sublattice of the extra slow Tamari
lattice, the first and fifth results are direct corollaries of
\Cref{thm:intro}. Extremality, trimness and polygonality are not
preserved by taking sublattices. They still hold for the slow Tamari
lattice, but they require slightly different proofs.

\Cref{tab:comparison} summarizes several meaningful enumerative
results for the classical Tamari lattice, together with the
corresponding enumerations for the slow Tamari and extra slow Tamari
lattices that we compute in this paper. We next briefly discuss the
enumerations of spine elements and of congruence lattices.

\begin{table}[ht]
  \centering
  \renewcommand{\arraystretch}{1.35}
  \setlength{\tabcolsep}{10pt}
  \begin{tabular}{@{} p{4.2cm} c c c @{}}
    \toprule
     & Tamari & slow Tamari & extra slow Tamari \\
    \midrule
    elements
      & \(C_n\)
      & \(n!\)
      & \( [n]_2 !\) \\
    \addlinespace[0.4em]

    \makecell[l]{join-irreducible elements\\
    (\(=\) length of a longest chain)}
      & \(\binom{n}{2}\)
      & \makecell[l]{\Cref{lem:stamjiir}\\ \(\binom{n}{2} + \binom{n-1}{2} = (n-1)^2\)}
      & \makecell[l]{\Cref{cor:join_irr}\\ \(2\binom{n}{2} + \binom{n-1}{2}\)}\\
    \addlinespace[0.4em]

    spine elements
      & \(2^{n-1}\)
      & \makecell[l]{\Cref{pro:1}\\ \(1,2,5,16,62,280,\dots\)}
      & \makecell[l]{\Cref{thm:counting_spine}\\ \(1,3,20,288,8672,\dots\)} \\
    \addlinespace[0.4em]

    elements in the congruence lattice
      & \(C_n\)
      & \makecell[l]{\Cref{pro:2,the:1}\\ \(1,2,7,41,328,3077,\dots\)}
      & \makecell[l]{\Cref{the:congruence_enumeration,prop:esTam_cong_cf}\\ \(1,4,44,932,35788,\dots\)} \\
    \bottomrule
  \end{tabular}
  \caption{Comparison of Tamari, slow Tamari, and extra slow Tamari.}
  \label{tab:comparison}
\end{table}

In~\Cref{pro:1} and \Cref{thm:counting_spine}, we derive formulas for
the numbers of spine elements in the slow and extra slow Tamari
lattices, respectively. More precisely, we introduce a recursively
defined two-parameter sequence \( (f_{i,j})_{i,j \geq 1} \) and a
three-parameter sequence \( (f(i,j,k))_{i,j,k \geq 0} \), and show
that the numbers of spine elements in \( \sTam_n \) and \( \esTam_n \)
are given by, respectively,
\[
  \sum_{ i + j = n+1 } f_{i,j}, \qand \sum_{i+j+k=n-1} f(i,j,k).
\]
Since the number of spine elements of
\( \Tam_n \) is known to be \( 2^{n-1} \), these formulas may be
regarded as slow and extra slow analogues of the sequence
\( 2^{n-1} \).

The number of elements in the congruence lattice of \( \Tam_n \) is
the Catalan number \( C_n := \frac{1}{n+1}\binom{2n}{n} \), and hence
is counted by Dyck paths of length \( 2n \). In~\Cref{pro:2} and
\Cref{the:congruence_enumeration}, we show that the numbers of
elements in the congruence lattices of the slow and the extra slow
Tamari lattices can likewise be expressed as certain weighted sums
over Dyck paths. Moreover, we obtain their generating functions
in~\Cref{the:1} and \Cref{prop:esTam_cong_cf}.
\begin{theorem}\label{the:2}
  The generating functions for the numbers of elements in the
  congruence lattices of \( \sTam_n \) and \( \esTam_n \) can be
  written in the continued fraction form
\[
F(a_0,a_1,\ldots;\lambda_1,\lambda_2,\ldots)
:=
\cfrac{1}{1+a_0x-\cfrac{\lambda_1 x}{1+a_1x-\cfrac{\lambda_2 x}{1+a_2x-\ddots}}},
\]
where
\begin{enumerate}
\item for \( \sTam_n \),
  \[
    (a_i)_{i \geq 0} = (0,0,1,1,1, \cdots), \qand (\lambda_i)_{i \geq 1} = (1,1,4,4,4, \dots );
  \]
\item for \( \esTam_n \),
  \[
    (a_i)_{i \geq 0} = (0,1,1,2,4,8,16, \cdots), \qand (\lambda_i)_{i \geq 1} = (1,4,8,16,32,64, \dots ).
  \]
\end{enumerate}
\end{theorem}
It is well known that the generating function for Catalan numbers is
\[
  \sum_{ n \geq 0 } C_n x^n = F(0,0,0, \dots ; 1,1,1, \dots ).
\]
Thus, the generating functions counting elements in the congruence
lattices of the Tamari, slow Tamari, and extra slow Tamari lattices
all belong to the same class
\( F(a_0,a_1,\ldots;\lambda_1,\lambda_2,\ldots) \) of continued
fractions.

In this paper, we focus on the poset structure of fb-tableaux. In a
subsequent work, we will study a more general family of tableaux and
give a bijective proof of a weighted enumeration formula. We will also
investigate connections with orthogonal polynomials and Schubert cells
of the flag variety in type \( A \).

The remainder of this paper is organized as follows. In
\Cref{sec:lattice}, we introduce the poset of fb-tableaux and prove
that it is a lattice. In \Cref{sec:sublattices}, we show that our
lattice contains several remarkable sublattices, including the Tamari
lattice and the slow Tamari lattice. In \Cref{sec:covers}, we describe
the cover relations of the lattice, and in \Cref{sec:irred}, we use
this description to characterize its join-irreducible and
meet-irreducible elements. In \Cref{sec:semidistributivity}, we prove
that the lattice is semidistributive. In \Cref{sec:trim}, we show that
it is trim, describe its spine, and enumerate the elements on the
spine. In \Cref{sec:polygonality}, we study the polygonal structure of
the lattice and show that it admits a polygonal labeling, and in
\Cref{sec:lattice-congruences}, we use this to prove that the lattice
is congruence uniform and to describe its lattice of congruences.
Finally, in \Cref{sec:slow_Tamari}, we revisit the slow Tamari lattice
and show that several properties established for the extra slow Tamari
lattice also hold for the slow Tamari lattice.

\section{A lattice of fb-tableaux}\label{sec:lattice}

Throughout this paper, we assume that \( n \) is a positive integer.

\subsection{A poset of fb-tableaux}\label{sec:def}

For representation-theoretic reasons, we denote the position of the
\( j \)th cell from the left in the \( i \)th row of the shape
\( (n,n,n-1, \dots, 3,2) \) by \( (i,n+1-j) \). That is,
\[
  \begin{tikzpicture}[x=6em, y=2em, every node/.style={font=\small}]
    \foreach \i in {1,...,5} {
      \pgfmathtruncatemacro{\len}{6-\i}
      \pgfmathtruncatemacro{\imax}{\i-1}
      \pgfmathtruncatemacro{\endj}{\len}
      \foreach \j in {1,...,\endj} {
        \draw[thick] (\j-1, -\imax) rectangle ++(1,-1);
      }
    }

    \filldraw[fill=gray, thick] (4,-1) rectangle ++(1,-1); 
    \filldraw[fill=gray, thick] (3,-2) rectangle ++(1,-1); 
    \filldraw[fill=gray, thick] (2,-3) rectangle ++(1,-1);
    \filldraw[fill=gray, thick] (1,-4) rectangle ++(1,-1);    
    
    \newcommand{\putcell}[3]{
      \node at ({(#2-1)+0.5},{-(#1-1)-0.5}) {$#3$};
    }
    \putcell{1}{1}{(1,n)}
    \putcell{1}{2}{(1,n-1)}
    \putcell{1}{3}{\cdots}
    \putcell{1}{4}{(1,2)}
    \putcell{1}{5}{(1,1)}
    \putcell{2}{1}{(2,n)}
    \putcell{2}{2}{(2,n-1)}
    \putcell{2}{3}{\cdots}
    \putcell{2}{4}{(2,2)}
    \putcell{2}{5}{(2,1)}
    \putcell{3}{1}{\vdots}
    \putcell{3}{2}{\vdots}
    \putcell{3}{3}{\reflectbox{\( \ddots \)}}
    \putcell{3}{4}{(3,2)}
    \putcell{4}{1}{(n-1,n)}
    \putcell{4}{2}{(n-1,n-1)}
    \putcell{4}{3}{\reflectbox{\( \ddots \)}}
    \putcell{5}{1}{(n,n)}
    \putcell{5}{2}{(n,n-1)}
  \end{tikzpicture}
\]
where the pair written in each cell indicates its position.
\begin{definition}
  A \defn{fb-tableau} of size \( n \) is a filling of the cells of
  shape \( (n,n,n-1,n-2,\dots,3,2) \) with dots such that
  \begin{itemize}
  \item each cell is either empty or contains a dot,
  \item the root \( (1,n) \) and border cells \( (2,1), (3,2),\dots, (n,n-1) \) contain dots, and
  \item every dot, except the one in the root \((1,n)\), has a dot either to its left or above it.
  \end{itemize}
\end{definition}
By convention, the fb-tableau of size \( 1 \) consists of a single
cell \( (1,1) \) containing a dot. We first parse any fb-tableau and
replace some of the dots and border cells with arrows or dots. Every
cell or border cell, except $(1,n)$, that has no dot to its left is
replaced by a left-arrow $\leftarrow$. Every cell or border cell,
except $(1,n)$, that has no dot above itself is replaced by an
up-arrow~$\uparrow$. The other dots are denoted by $\bullet$. For the
rest of the article, we use the following convention: the
\defn{elements} of a fb-tableaux are the entries of the nonempty
cells, the \defn{arrows} are the elements equal to $\leftarrow$ or
$\uparrow$, and the \defn{dots} are the elements equal to $\bullet$.
An element $\bullet$ has a left-arrow to its left and an up-arrow
above itself, except if it is the root. However, in the proofs, the
root will still behave like the other $\bullet$. For this reason, we
give it the same notation. 

\noindent Here we introduce the notion of the cohook of a cell. This
will be useful for verifying that a tableau is a fb-tableau.
\begin{definition}\label{def: cohook}
  We define the \defn{cohook} \( H(i,j) \) at the position \( (i,j) \)
  to be the set consisting of the cell \( (i,j) \) together with all
  cells to its left and above it. 
\end{definition}
\noindent Then a fb-tableau is a filling of a tableau with dots such that:
\begin{enumerate}
\item The root and the border cells are nonempty.
\item The cohook of each nonempty cell except the root contains at least two elements.  
\end{enumerate}

Here is an example of a fb-tableau of size $5$:
\[
\ytableausetup{boxsize = 1.2em}
  \begin{ytableau}
    \bullet & & \uparrow & \uparrow & \\
    & & \leftarrow & &*(gray)\uparrow\\
    \leftarrow &&\bullet &*(gray)\bullet \\
    \leftarrow &\uparrow &*(gray)\bullet\\
    &*(gray)\leftarrow\\
  \end{ytableau}
\]
We highlight the border cells in gray. By definition, they cannot be
empty.

\begin{remark}
  Since every border cell is nonempty, the cohook of each border cell
  contains at least two distinct elements: the topmost nonempty cell
  above it, which is an up-arrow, and the leftmost nonempty cell to
  its left, which is a left-arrow. Hence, as an fb-tableau of size
  \( n \) has \( n-1 \) border cells, it contains exactly \( n-1 \)
  up-arrows and \( n-1 \) left-arrows.
\end{remark}

For a fb-tableau \(T\), we define \(\row(T)\) to be the set of cells
lying weakly to the right of the left-arrows and the root, and call it
the \defn{row space} of \( T \). Similarly, we define \( \col(T) \) to
be the set of cells lying weakly below the up-arrows and the root, and
call it the \defn{column space} of \( T \). A cell is said to be
\defn{free} if it lies strictly below an up-arrow and strictly right
to a left-arrow. We define \( F(T) \) to be the set of free cells in
\( T \). By definition, every dot of \(T\) belongs to \(F(T)\);
however, \(F(T)\) may also contain some empty cells. For example, if
\(T\) is
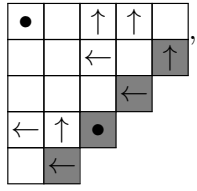
\[
  \begin{ytableau}
    \bullet & & \uparrow & \uparrow & \\
    & & \leftarrow & &*(gray)\uparrow\\
    &&&*(gray)\leftarrow \\
    \leftarrow &\uparrow &*(gray)\bullet\\
    &*(gray)\leftarrow\\
  \end{ytableau},
\]
then the cells in \(\row(T)\), \(\col(T)\), and \(F(T) \) are
highlighted from left to right, in red, green, and blue respectively:
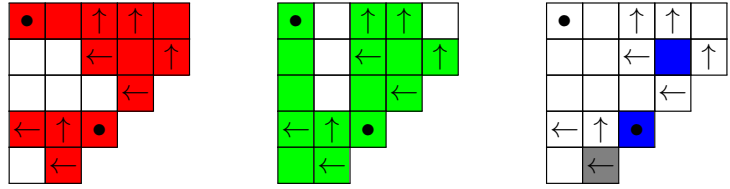
\[\begin{ytableau}
    *(red)\bullet &*(red)  &*(red) \uparrow &*(red) \uparrow &*(red) \\
    & &*(red) \leftarrow  &*(red) &*(red) \uparrow\\
    &&& *(red) \leftarrow\\
    *(red) \leftarrow &*(red) \uparrow &*(red)\bullet\\
    &*(red)\leftarrow\\
  \end{ytableau}
\quad\quad\quad\begin{ytableau}
    *(green)  \bullet & &*(green) \uparrow & *(green) \uparrow & \\
    *(green) & &*(green) \leftarrow  &*(green) &*(green)\uparrow\\
    *(green) &&*(green)&*(green) \leftarrow \\
    *(green) \leftarrow & *(green) \uparrow  &*(green)\bullet\\
    *(green) &*(green)\leftarrow\\
  \end{ytableau}
  \quad \quad \quad 
  \begin{ytableau}
    \bullet & & \uparrow & \uparrow & \\
    & & \leftarrow &*(blue) &\uparrow\\
    &&&\leftarrow \\
    \leftarrow &\uparrow &*(blue)\bullet\\
    &*(gray)\leftarrow\\
  \end{ytableau}
\]
For two fb-tableaux \(S\) and \(T\), we let \(F(S,T)\) denote the
set of cells that are free in both \(S\) and \(T\). These sets of
cells allow us to define a binary relation $S\lhd T$ on the set of
fb-tableaux. We write $S\lhd T$ if
\begin{enumerate}
    \item $\col(S)\subseteq \col(T)$;
    \item $\row(S)\supseteq \row(T)$;
    \item Let $c$ be a cell in $F(S,T)$. If $c$ contains a dot in $S$, then it contains a dot in $T$.
\end{enumerate}
Here is the poset for \(n=2\):
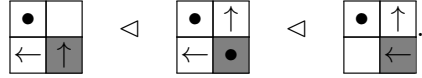
\[
\begin{ytableau}
    \bullet & \\
    \leftarrow&*(gray)\uparrow \\
  \end{ytableau}
\quad \lhd \quad 
\begin{ytableau}
    \bullet &\uparrow \\
    \leftarrow&*(gray)\bullet \\
  \end{ytableau}
\quad \lhd \quad 
\begin{ytableau}
    \bullet &\uparrow \\
    &*(gray)\leftarrow \\
  \end{ytableau}.
\]
Here is an example of two tableaux with the same row spaces and same
columns spaces, illustrating the third condition:
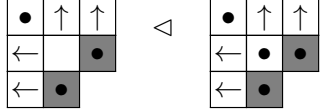
\[
\begin{ytableau}
    \bullet &\uparrow & \uparrow \\
    \leftarrow&& *(gray)\bullet \\
    \leftarrow & *(gray)\bullet \\ 
  \end{ytableau}
  \quad \lhd \quad 
  \begin{ytableau}
    \bullet &\uparrow & \uparrow \\
    \leftarrow& \bullet& *(gray)\bullet \\
    \leftarrow & *(gray)\bullet \\ 
  \end{ytableau}. 
\]
We also denote by \(\col^{s}(T)\) the set of cells that lie
strictly below the up-arrows of \(T\). Similarly, we denote by
\(\row^{s}(T)\) the set of cells that lie strictly to the right of the
left-arrows of \(T\).

\begin{lemma}\label{lem:free}
  Let \(S\) and \(T\) be two fb-tableaux such that
  \(\col(S)\subseteq \col(T) \) and \( \row(S) \supseteq \row(T)\).
  Then
  \[
    F(S,T)= \col^s(S)\cap \row^s(T).
  \]
\end{lemma}
\begin{proof}
  Let \(c\) be free in both \(S\) and \(T\). Then \( c \) lies below
  the up-arrow of its column in \(S\) and \(T\). The lowest of the
  two arrows is the one of \(S\), hence \(c\in \col^{s}(S)\).
  Similarly, we get \(c\in \row^s(T)\). Conversely, let
  \(c\in \col^s(S)\cap \row^s(T)\). Since
  \( \row(S) \supseteq \row(T)\), we get
  \(c\in \col^s(S)\cap \row^s(S)=F(S)\). Similarly, since
  \(\col(S)\subseteq \col(T)\), we get \(c\in F(T)\).
\end{proof}
\begin{lemma}\label{lem:poset}
  The relation \(\lhd\) is a partial order on the set of fb-tableaux
  of size \(n\).
\end{lemma}
\begin{proof}
  By definition, the relation is reflexive. To check transitivity, let
  \(T_1,T_2\) and \(T_3\) be fb-tableaux such that \(T_1\lhd T_2 \)
  and \( T_2 \lhd T_3 \). The only nontrivial part is to verify the
  condition involving free cells. By~\Cref{lem:free}, we have
\[
  F(T_1,T_3) = \col^{s}(T_1) \cap \row^{s}(T_3) \subseteq \col^s(T_2)
  \cap \row^{s}(T_3) =F(T_2,T_3).
\]
 Similarly, we have 
 \[
   F(T_1,T_3)= \col^{s}(T_1) \cap \row^{s}(T_3) \subseteq \col^s(T_1)
   \cap \row^{s}(T_2) = F(T_1,T_2).
 \]
 Let \(c\) be a cell in \(F(T_1,T_3)\) such that \(c\) contains a dot
 in \(T_1\). Since \(c\in F(T_1,T_2)\), the cell \( c \) also contains
 a dot in \(T_2\). Since \(c\in F(T_2,T_3)\), it also contains a dot
 in \(T_3\). Hence, the relation \(\lhd\) is transitive.

 To check the anti-symmetry, let \( S \) and \( T \) be fb-tableaux
 such that \(S\lhd T\) and \(T\lhd S\). Then \(\col(S)=\col(T)\),
 so in each column, the up-arrow of \(S\) is in the same position as
 the up-arrow of \(T\). Similarly, in each row, the left-arrow of
 \( S \) is in the same position as the left-arrow of \( T \). By
 \Cref{lem:free}, we have \(F(S,T)=F(T,S)\). The third condition in
 the definition of \(\lhd\) then implies that, a cell \( c \) contains
 a dot in \(S\) if and only if \( c \) contains a dot in \(T\), and
 therefore \( S = T \).
\end{proof}

\subsection{Self-duality of the poset}

\begin{definition}\label{def:conjugate}
  Given a fb-tableau \(T\), we define its \defn{conjugate} \(T'\) to
  be the unique tableau obtained as follows.
  \begin{itemize}
  \item First, reflect \( T \) about the main diagonal, thereby
    interchanging all left-arrows and up-arrows.
  \item Next, consider a free cell \( (i,j) \) in \( T \) with
    \( 1 \leq i \leq j \leq n \) that is neither \( (1,n) \) nor a
    border cell. If \( (i,j) \) contains a dot in \( T \), then the
    corresponding cell \((n+1-j,n+1-i)\) is empty in \(T'\); if
    \( (i,j) \) is empty in \(T\), then \((n+1-j,n+1-i)\) contains a
    dot in \(T'\).
  \end{itemize}
\end{definition} 

\begin{lemma}
  The map \(T \mapsto T'\) defines an involution on the set of
  \(fb\)-tableaux.
\end{lemma}
\begin{proof}
  We clearly have \((T')'=T\), so we only need to check that \( T' \)
  is a fb-tableau whenever \(T\) is. Reflecting \(T\) about the main
  diagonal sends each cohook of \(T\) to a cohook of \(T'\) and each
  free cell of \(T\) to a free cell of \( T' \). The result follows.
\end{proof}
We give an example of the conjugate of a tableau $T$:
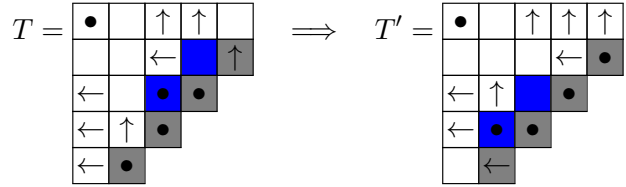
\[
  T = 
 \begin{ytableau}
    \bullet & & \uparrow & \uparrow & \\
    & & \leftarrow &*(blue) &*(gray)\uparrow\\
    \leftarrow &&*(blue)\bullet &*(gray)\bullet \\
    \leftarrow &{\uparrow} &*(gray)\bullet\\
    \leftarrow&*(gray)\bullet
  \end{ytableau}
  \quad \Longrightarrow \quad T' = 
   \begin{ytableau}
    \bullet & & \uparrow & \uparrow & \uparrow\\
    & & &\leftarrow  &*(gray)\bullet\\
    \leftarrow &{\uparrow}&*(blue) &*(gray)\bullet \\
    \leftarrow & *(blue)\bullet &*(gray)\bullet\\
    &*(gray)\leftarrow
  \end{ytableau}
\]

\begin{theorem}\label{thm:selfdual}
  The poset \(\posfb\) is self-dual, that is, \(S\lhd T\) if and only
  if \(T'\lhd S'\).
\end{theorem}
\begin{proof}
  Let \( S \) and \( T \) be two fb-tableaux such that \(S\lhd T\).
  For \(t=S,T\), a cell \((i,j)\) is in \(\col(t)\) (resp.
  \(\row(t)\)) if and only if the cell \((n+1-j,n+1-i)\) is in
  \(\row(t')\) (resp. \(\col(t')\)). From
  \(\col(S)\subseteq \col(T)\), we obtain
  \(\row(S')\subseteq \row(S)\), and from
  \(\row(T)\subseteq \row(S)\), we obtain
  \(\col(T')\subseteq \col(S')\). Now let \(c=(n+1-j,n+1-i)\) be a
  free cell (not on the border) of \(S'\) that contains a dot in
  \(T'\). Then \((i,j)\) is a free cell of both \(S\) and \(T\).
  Moreover, \( (i,j) \) is empty in \(T\). By the third condition in
  the definition of the order \(\lhd\), it is also empty in \(S\).
  Hence the cell \(c\) contains a dot in \(S'\) and \(T'\lhd S'\).
\end{proof}

\subsection{Meets and joins}

\begin{lemma}\label{lem:meet}
Let $M$ and $N$ be two fb-tableaux. Let $L$ be the tableau defined as follows:
\begin{enumerate}
\item[(a)] The root of $L$ is a dot.
\item[(b)] In each column, the up-arrow of $L$ is the lowest of the up-arrows of $M$ and $N$. 
\label{step_c}\item[(c)] In each row, the left-arrow of $L$ is the leftmost of the left-arrows of $M$ and $N$. If it is pointed by an up-arrow, we move it to the left at the first position which is not pointed by an up-arrow. 
\item[(d)] Let $c$ be free cell of $L$. It contains a dot in $L$ if
\begin{enumerate}
    \item[(i)] $c$ contains a dot in $M$ or is not free in $M$ and
    \item[(ii)] $c$ contains a dot in $N$ or is not free in $N$.
\end{enumerate}
\end{enumerate}
Then 
\begin{enumerate}
\item If a cell \(c\) is free in \(L\), then it is free in \(M\) or in
  \(N\). If a cell is free in \(M\) and in \(N\), then it is free in
  \(L\).
\item Let \(c\) be a border cell, and \(e_M\), \(e_N\), and \(e_L\) be
  the elements of \(M\), \(N\) and \(L\), respectively, contained in
  \( c \). Then \(e_L = \min(e_N,e_M)\) where the minimum is taken
  with respect to the total order
  \(\uparrow \,\leq \bullet \leq\, \leftarrow\).
\item The tableau \(L\) is a fb-tableau.
\end{enumerate}
\end{lemma}
\begin{proof}
  Because of the root, there is always a position for the left-arrows
  in the step (c) of the construction. We then need to check that the
  labeling of the cells that we have chosen is coherent with their
  meaning: all cells to the left of a left-arrow are empty, and
  similarly for up-arrows. If there is a nonempty cell to the left of
  a left-arrow in \(L\), then it must be an up-arrow. Indeed, it is
  not a free cell, and there is only one left-arrow in each row. This
  up-arrow is an up-arrow in \(M\) or in \(N\), and it lies to the
  left of the left-arrows of \(M\) and \(N\). This is a contradiction.
  Similarly, if there is a nonempty cell above an up-arrow, then it
  cannot be a free cell, so it must be an arrow. Since a left-arrow
  lies below an up-arrow, this implies that \(L\) has a column with
  two up-arrows, which is again a contradiction.

  A cell \(c\) is free in \(L\) when it lies strictly to the right of
  a left-arrow of \(L\) and strictly below an up-arrow of \(L\). The
  up-arrows of \(L\) are below the up-arrows of \(M\) and \(N\), so
  \(c\in \col^{s}(M)\cap \col^s(N)\). The corresponding left-arrow of
  \(L\) is the left-arrow of \(M\) or \(N\), say of \(N\), or else it
  is at the first allowed position to its left. By construction, the
  cells between the left-arrows of \(L\) and \(N\) are not below an
  up-arrow, so they are not free. Hence \(c\in \row^s(N)\), and
  therefore \(c\) is free in \(N\). If \(c\) is free in both \(M\) and
  \(N\), then \(c\) is clearly free in \(L\), and this proves the
  first statement.

  If a border cell contains an up-arrow in one of the tableaux \(M\)
  and \(N\), then it contains an up-arrow in \(L\). If it contains a
  left-arrow in both \(M\) and \(N\), then it contains a left-arrow in
  \(L\). Otherwise, suppose that it contains a dot in \(M\). There are
  two cases: it contains a left-arrow in \(N\), or it contains a dot
  in \(N\). In the first case, it is not free in \(N\), and hence the
  cell contains a dot in \(L\). In the second case, it contains a dot
  in both \(N\) and \(M\), and so also in \(L\). This proves the second statement

  The left-arrows of \(L\) are strictly below the up-arrows, so their
  cohooks contain at least two elements. An up-arrow \(u\) of \(L\) is
  an up-arrow of \(M\) or of \(N\). Hence, in \(M\) or in \(N\), there
  is a left-arrow to its left. This left-arrow also appears in \(L\)
  (maybe not in the same position, but further to the left), so \(u\)
  has a left-arrow to its left in \(L\), and its cohook contains at
  least two elements. The cells containing a dot are in \(F(L)\), so
  their cohooks contain at least three elements. The root and the
  border cells of \(L\) are nonempty, so \(L\) is a fb-tableau.
\end{proof}

For example, consider the fb-tableaux $M$ and $N$:
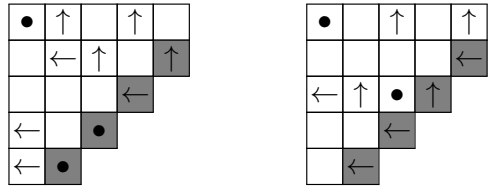
\[
\begin{ytableau}
    \bullet &\uparrow &  &\uparrow & \\
     &\leftarrow&\uparrow& &*(gray)\uparrow\\
     && &*(gray)\leftarrow \\
    \leftarrow & &*(gray)\bullet\\
    \leftarrow&*(gray)\bullet\\
  \end{ytableau}\quad \quad\quad\quad
\begin{ytableau}
    \bullet & &\uparrow  & &\uparrow \\
    &   & & &*(gray)\leftarrow\\
    \leftarrow &\uparrow &\bullet &*(gray)\uparrow \\
    & &*(gray)\leftarrow\\
    &*(gray)\leftarrow\\
  \end{ytableau}
\]
To construct $L$, we first place the up-arrows, then the left-arrows
in allowed positions, and finally fill the free cells:
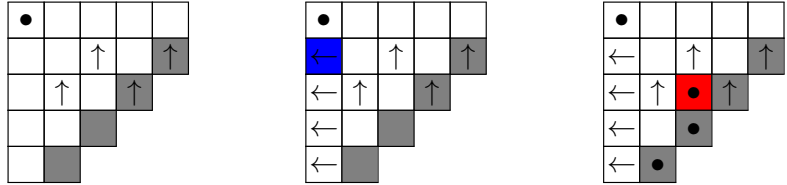
\[
\begin{ytableau}
    \bullet & &  &  & \\
     &&\uparrow & &*(gray)\uparrow\\
     &\uparrow & &*(gray)\uparrow \\
      & &*(gray)\\
    &*(gray)\\
  \end{ytableau}\quad \quad\quad\quad
\begin{ytableau}
    \bullet & &  &  & \\
     *(blue)\leftarrow&&\uparrow & &*(gray)\uparrow\\
     \leftarrow &\uparrow & &*(gray)\uparrow \\
      \leftarrow & &*(gray)\\
    \leftarrow &*(gray)\\
  \end{ytableau}\quad \quad\quad\quad
\begin{ytableau}
    \bullet & &  &  & \\
     \leftarrow&&\uparrow & &*(gray)\uparrow\\
     \leftarrow &\uparrow &*(red)\bullet &*(gray)\uparrow \\
      \leftarrow & &*(gray)\bullet\\
    \leftarrow & *(gray)\bullet \\
  \end{ytableau}
\]
Note that the left-arrow highlighted in blue in the second picture had
to be moved since its natural position is pointed by an up-arrow. The
free cell highlighted in red in the third picture is an example of a
cell that is empty in $M$, but is not free.

\begin{theorem}\label{thm:estam_lattice}
  The poset \((\esTam_n,\lhd)\) is a lattice.
\end{theorem}
\begin{proof}
  Using the symmetry in \Cref{thm:selfdual}, it is enough to prove
  that \((\esTam_n,\lhd)\) is a meet-semilattice; that is, every pair
  of elements of \( \esTam_n \) has a meet. For two fb-tableaux
  \( M \) and \( N \), let \(L\) be the fb-tableau defined in
  \Cref{lem:meet}. We have:
\begin{enumerate}
    \item \(\col(L)=\col(M)\cap \col(N)\).
    \item \( \row(L) \supseteq \row(M)\cup \row(N)\).
    \item Let \(c \in F(L,M)\). If \(c\) contains a dot in \(L\), then
      \(c\) contains a dot in \(M\), since \(c\) is free in \(M\).
\end{enumerate} 
It follows that \(L \lhd M\), and similarly \(L\lhd N\). If \(T\) is a
tableau such that \(T \lhd M\) and \(T\lhd N\), then
\begin{enumerate}
    \item \(\col(T) \subseteq \col(M)\cap \col(N) = \col(L)\). 
    \item \(\row(T) \supseteq \row(M) \cup \row(N)\). In a given row,
      the left-arrow \(\ell\) of \(T\) lies to the left of the
      left-arrows of \(M\) and \(N\). Since \(T\) is a fb-tableau,
      there is an up-arrow in \(T\) above \(\ell\). Since
      \(\col(T) \subseteq \col(L)\), there is also an up-arrow in
      \(L\) above \(\ell\). Thus, the position of \(\ell\) is a
      candidate for step (c) of the construction of \(L\). By
      the minimality of this step, we obtain that the left-arrow
      of \(L\) is weakly to the right of \(\ell\). Hence
      \(\row(T) \supseteq \row(L)\).
    \item Let \(c\in F(T,L)\) be a dot in \(T\). If \(c\) is free in
      \(M\) (resp. \(N\)), then it contains a dot in \(M\) (resp.
      \(N\)). Hence \(c\) contains a dot in \(M\) or is not free in
      \(M\), and similarly for \(N\). Therefore, \(c\) contains a dot
      in \(L\).
\end{enumerate}
This proves that \(T\lhd L\), and hence \(L\) is the meet of \(M\) and \(N\).
\end{proof}

It will be useful to give an explicit description of the join of two
fb-tableaux. The following lemma is obtained by applying the
anti-isomorphism of \Cref{thm:selfdual} and remarking that condition
(d) for the meet in \Cref{lem:meet} is equivalent to the following
statement: Let \(c\) be a free cell of \(L\). Then \( c \) is empty in
\(L\) if and only if it is free and empty in \(M\) or free and empty
in \(N\).

\begin{lemma}\label{lem:join}
  The join \(L = M\lor N\) of two fb-tableaux \(M\) and \(N\) is the
  unique tableau defined by
\begin{enumerate}
\item[(a)] The root of \(L\) is a dot.
\item[(b)] In each row, the left-arrow of \(L\) is the rightmost of
  the left-arrows of \(M\) and \(N\).
\item[(c)] In each column, the up-arrow of \(L\) is the highest of the
  up-arrows of \(M\) and \(N\). If it is pointed by a left-arrow, we
  move it upward to the first position that is not pointed by an
  up-arrow.
\item[(d)] Let \(c\) be a free cell of \(L\). Then \( c \) contains a
  dot in \(L\) if it contains a dot in \(M\) or in \(N\).
\end{enumerate}
\end{lemma}
For example, consider the following two tableaux $T$ and $U$:
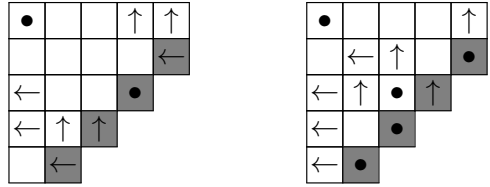
\[
\begin{ytableau}
    \bullet & &  & \uparrow & \uparrow\\
     &&& &*(gray)\leftarrow\\
    \leftarrow && &*(gray)\bullet \\
    \leftarrow &\uparrow &*(gray)\uparrow\\
    &*(gray)\leftarrow\\
  \end{ytableau}\quad \quad\quad\quad
\begin{ytableau}
    \bullet & &  & &\uparrow \\
    &  \leftarrow &\uparrow & &*(gray)\bullet\\
    \leftarrow &\uparrow &\bullet &*(gray)\uparrow \\
    \leftarrow & &*(gray)\bullet\\
    \leftarrow&*(gray)\bullet\\
  \end{ytableau}
\]
To construct $T\vee U$, we first place the left-arrows, then the
up-arrows, and finally fill the free cells:
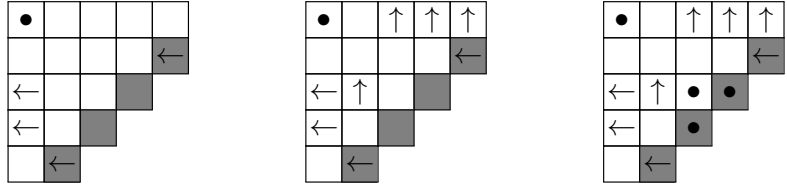
\[
\begin{ytableau}
    \bullet & &  &  & \\
     &&& &*(gray)\leftarrow\\
    \leftarrow && &*(gray) \\
    \leftarrow & &*(gray)\\
    &*(gray)\leftarrow\\
  \end{ytableau}\quad \quad\quad\quad
\begin{ytableau}
    \bullet & & \uparrow & \uparrow & \uparrow\\
     &&& &*(gray)\leftarrow\\
    \leftarrow & \uparrow &  &*(gray) \\
    \leftarrow & &*(gray)\\
    &*(gray)\leftarrow\\
  \end{ytableau}\quad \quad\quad\quad
\begin{ytableau}
    \bullet & & \uparrow & \uparrow & \uparrow\\
     &&& &*(gray)\leftarrow\\
    \leftarrow & \uparrow & \bullet &*(gray)\bullet \\
    \leftarrow & &*(gray)\bullet\\
    &*(gray)\leftarrow\\
  \end{ytableau}\]

\section{Remarkable sublattices}\label{sec:sublattices}

\begin{definition}
  Let \(T\) be a fb-tableau of size \(n\). Then
\begin{enumerate}
    \item the tableau \(T\) is \defn{small} if it contains exactly one dot, the root;
    \item the tableau \(T\) is \defn{binary} if the set of free cells of \(T\) is empty. 
\end{enumerate}
\end{definition}

We denote by \(\sTam_n\) the set of small fb-tableaux, and we use the
relation \(\lhd\) to regard \(\sTam_n\) as an induced subposet of
\(\esTam_n\). We denote by \(\Tam_n\) the set of binary fb-tableaux,
which we also regard as a subposet of \(\sTam_n\) and of \(\esTam_n\).

\begin{proposition}\label{prop:sublattices}
  The posets \(\Tam_n\) and \(\sTam_n\) are lattices.
\end{proposition}
\begin{proof}
  By \Cref{lem:meet} and its dual statement for joins, we see that the
  meet and join of two tableaux without free cells are again tableaux
  without free cells. Similarly, the meet and join of two tableaux
  without dots are again tableaux without dots, except the root.
  Hence, \(\Tam_n\) and \(\sTam_n\) are sublattices of \(\esTam_n\).
\end{proof}
As an immediate corollary, we recover \cite[Proposition
9.4]{CMRS2021}, where \(\sTam_n\) was denoted by \(fb(n)\).

\begin{lemma}\label{lem:binary}
  The following statements hold for binary fb-tableaux.
\begin{enumerate}
    \item A binary fb-tableau is determined by its up-arrows.
    \item For binary fb-tableaux \( S \) and \( T \), we have
      \(S\lhd T\) if and only if \(\col(S)\subseteq \col(T)\).
\end{enumerate}
\end{lemma}
\begin{proof}
  For~(1), it is enough to show that the position of the left-arrow in
  each row is uniquely determined by the positions of the up-arrows.
  Fix a row, and let \( c \) be the rightmost cell in that row lying
  below an up-arrow. We claim that the left-arrow must be placed in
  \( c \). Indeed, if it were placed strictly to the left of \( c \),
  the \( c \) would become a free cell. On the other hand, it cannot
  be placed strictly to the right of \( c \), since the left-arrow
  must lie below an up-arrow or below the root. Hence the left-arrow
  is uniquely determined in each row, proving~(1).

  For~(2), the implication
  \(S\lhd T \Rightarrow \col(S)\subseteq \col(T) \) is
  immediate from the definition. Thus, it remains to prove the
  converse. Assume \(\col(S)\subseteq \col(T)\). Fix a row, and
  let \( c \) be the rightmost cell in that row lying below an
  up-arrow of \( S \). By~(1), this is the position of the
  left-arrow of \( S \). Since \(\col(S)\subseteq \col(T)\), the
  up-arrow above \( c \) in \( S \) lies weakly below an up-arrow of
  \( T \) in the same column. Hence the rightmost cell in that row
  lying below an up-arrow of \( T \) is weakly to the right of
  \( c \). Again by~(1), this is the position of the left-arrow of
  \( T \). Therefore, \(\row(T)\subseteq \row(S)\), and so
  \(S\lhd T\).
\end{proof}

A \defn{bracket vector} of length \( n-1 \) is a tuple
\( (a_1, \dots, a_{n-1}) \) of \( n-1 \) integers satisfying
\( 0 \leq a_i \leq n-i \), and \( j + a_{i+j} \leq a_i \) for all
\( j \leq a_i \). The set of all bracket vectors of length \( n-1 \)
is partially ordered by
\( (a_1, \dots, a_{n-1}) \leq (b_1, \dots, b_{n-1}) \) if
\( a_i \leq b_i \) for \( i=1, \dots, n-1 \). Huang and Tamari
\cite{zbMATH03392515} proved that this poset is a lattice and is
isomorphic to the Tamari lattice.

\begin{proposition}\label{pro:tamari}
  The poset \((\Tam_n,\lhd)\) is isomorphic to the Tamari lattice on
  the bracket vectors of length \(n-1\).
\end{proposition}
\begin{proof}
  Let \( T \) be a binary fb-tableau, and consider its column space
  \(\col(T)\). We associate with \( T \) an \(n\)-tuple whose entries
  are, from left to right, the numbers of non-empty non-border cells in the
  columns of \(\col(T)\). Removing the entry corresponding to the
  first column, which is always equal to \(n\), gives an
  \((n-1)\)-tuple \((a_1,\dots,a_{n-1})\) satisfying
  \(0\leq a_i \leq n-i\). We denote it by \(b(T)\).

  Consider \( i \) and \( j \) with \(0\leq j\leq a_i\). In column
  \(i\), there are \(a_i\) non-border cells weakly below the up-arrow.
  We label the cells in each column from bottom to top by
  \(1,2,\dots\). Then, in column \(i\), the cell labeled by \(j\) lies
  in the same row as the border cell of column \(i+j\). Hence, in
  column \(i+j\), the cell labeled by \(a_{i+j}\) lies in the same row
  as the cell labeled by \(j+a_{i+j}\) in column \(i\). In the figure
  below, this cell is marked by a star. If \(a_{i}<j+a_{i+j}\), then
  the up-arrow in column \(i+j\) lies strictly above the up-arrow in
  column \(i\). Since every up-arrow has a left-arrow on its left in
  the same row, this creates a free cell in column \(i+j\), marked by
  \(F\) in the figure:
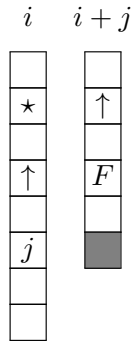
\[
\begin{array}{c c}
i & i+j \\[0.5em]

\begin{ytableau}
~ \\ \star \\ ~ \\ \uparrow \\ ~ \\ j \\ ~\\ ~ \\ 
\end{ytableau}
&
\begin{ytableau}
~ \\ \uparrow \\ ~ \\ F \\ ~ \\ *(gray)~
\end{ytableau}
\end{array}
\]

\noindent Since a binary fb-tableau has no free cells, we have
\( j + a_{i+j} \leq a_i \) for all \( j \leq a_i \). Thus \(b(T)\) is
a bracket vector. Moreover, if \(T_1\) and \(T_2\) are two binary
tableaux, then \(b(T_1)\leq b(T_2)\) if and only if
\(\col(T_1)\subseteq \col(T_2)\). By \Cref{lem:binary}~(2), we have
\(b(T_1)\leq b(T_2)\) if and only if \(T_1\lhd T_2\) in \(\Tam_n\). It
remains to show that the map \(T\mapsto b(T)\) is surjective.

Let \(v=(a_1,\dots,a_{n-1})\) be a bracket vector. We use \( v \) to
construct a tableau containing the root and \(n-1\) up-arrows. We then
place one left-arrow in each row except the first, under the rightmost
up-arrow or under the root, exactly as we did in the proof of
\Cref{lem:binary}~(1). By construction, every left-arrows have at
least two elements in its cohook. Now suppose there is an up-arrow in
column \(i\) to the left of a left-arrow in the column \(i+j\),
illustrated as follows:
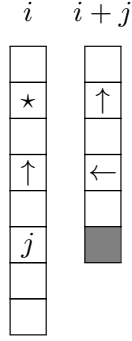
\[
\begin{array}{c c}
i & i+j \\[0.5em]

\begin{ytableau}
~ \\ \star \\ ~ \\ \uparrow \\ ~ \\ j \\ ~\\ ~ \\ 
\end{ytableau}
&
\begin{ytableau}
~ \\ \uparrow \\ ~ \\ \leftarrow \\ ~ \\ *(gray)~
\end{ytableau}
\end{array}
\]
Then \(j+a_{i+j}> a_{i}\), contradicting the assumption that \(v\) is
a bracket vector. Therefore, there are no elements to the left of a
left-arrow. Moreover, this also shows that every up-arrow has a
left-arrow to its left. Since there is no element above an up-arrow,
all labels are placed correctly, and the resulting tableau is a
fb-tableau.

Finally, suppose there is a free cell \( c \). Then \(c\) lies to the
right of a left-arrow and strictly below an up-arrow. But then \( c \)
would be a possible position for the left-arrow, contrary to the
construction. Hence there are no free cells, and the tableau is
binary.
\end{proof}

\begin{definition}\label{def:sTam}
  The lattice \((\sTam_n,\lhd)\) is called the \defn{slow Tamari lattice}.
\end{definition}


In \cite{CMRS2021}, the notion of \defn{minimal faithfully balanced
  modules} is discussed. In our terminology, a fb-tableau \(T\) is
\defn{minimal} if removing any element from a cell other than the root
and the border cells yields a tableau that no longer satisfies the
definition of a fb-tableaux. The first example of minimal tableau that
is not small occurs for \(n=4\), see below:
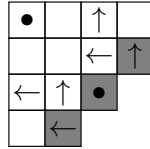
\[
\begin{ytableau}
    \bullet & & \uparrow  &  \\
     & &\leftarrow &*(gray)\uparrow\\
    \leftarrow &\uparrow&*(gray)\bullet \\
    &*(gray) \leftarrow \\
  \end{ytableau}
\]
If \(n\geq 5\), the minimal tableaux do not form a sublattice of
\(\esTam_n\), as can be seen in the following example:
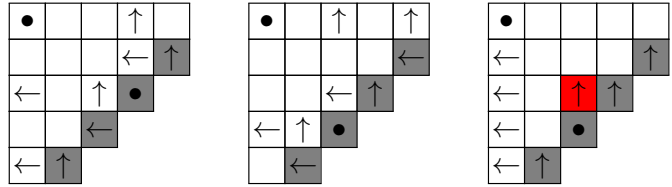
\[
\begin{ytableau}
    \bullet & & & \uparrow  &  \\
     & & &\leftarrow &*(gray)\uparrow\\
    \leftarrow && \uparrow&*(gray)\bullet \\
    & &*(gray) \leftarrow \\
    \leftarrow&*(gray)\uparrow\\
  \end{ytableau}
  \qquad 
  \begin{ytableau}
    \bullet & & \uparrow  & & \uparrow  \\
     & &&  &*(gray)\leftarrow\\
     & & \leftarrow & *(gray)\uparrow \\
    \leftarrow & \uparrow & *(gray)\bullet\\
    & *(gray)\leftarrow
  \end{ytableau}
  \qquad 
  \begin{ytableau}
\bullet & &&& \\
\leftarrow & &&  & *(gray)\uparrow \\
\leftarrow & & *(red)\uparrow & *(gray)\uparrow\\
\leftarrow & & *(gray)\bullet\\
\leftarrow & *(gray)\uparrow\\
  \end{ytableau}
\]
on the left we have two minimal fb-tableaux and on the right their
meet. The red arrow in the meet can be removed without breaking the
fb-condition. Hence the meet of two minimal tableaux needs not to be a
minimal tableau. Computer evidences suggest that the poset of minimal
tableaux is not a lattice when \(n\geq 5\).

Let \(T\) be a fb-tableau of size \(n\), we denote by \(\nu_T\) the
word in \(\{ \uparrow, \leftarrow,\bullet \}^{n-1}\) obtained by
reading the contents of the border cells of \(T\) from top to bottom.
The word \(\nu_T\) is simply called the \defn{border} of \(T\). If
\(T'\) denotes the \defn{conjugate} of \(T\) in \Cref{def:conjugate},
then the border of \(T'\) is obtained by reversing the border of \(T\)
and exchanging the two types of arrows. A \defn{congruence} on a
lattice \( L \) is an equivalence relation \( \ba \) on \( L \) such
that if \( x_1 \equiv_\ba x_2 \) and \( y_1 \equiv_\ba y_2 \), then
\( x_1 \wedge y_1 \equiv_\ba x_2 \wedge y_2 \) and
\( x_1 \vee y_1 \equiv_\ba x_2 \vee y_2 \).

\begin{lemma}\label{lem:cong}
  The relation defined by \(T_1\sim T_2\) if the tableaux \(T_1\) and
  \(T_2\) have the same border is a congruence relation on
  \(\esTam_n\).
\end{lemma}
\begin{proof}
  Let \(T_1,T_2\) and \(T\) be three fb-tableaux such that
  \(T_1\sim T_2\). By the second item of \Cref{lem:meet}, the border
  of \(T_1\land T\) depends only on the borders of \(T_1\) and \(T\).
  Hence \((T_1\land T) \sim (T_2 \land T)\). Moreover, if
  \(T_1 \sim T_2\), then their conjugates \(T_1'\) and \(T_2'\) have
  the same border. Therefore, \(T_1 \lor T = (T_1'\land T')'\) and
  \(T_2 \lor T = (T_2'\land T')'\) also have the same border.
\end{proof}

Given a word \(\nu\) in \(\{\uparrow,\leftarrow,\bullet\}^{n-1}\), we
denote by \((\esTam_\nu,\lhd)\) the subposet of \((\esTam_n,\lhd)\)
consisting of all tableaux with border \(\nu\). Since \(\esTam_\nu\)
extends the notion of \(\nu\)-Tamari introduced by Preville-Ratelle
and Viennot \cite{nutamari}, we call it the \defn{extra slow
  \(\nu\)-Tamari} lattice.

\begin{corollary}
\begin{enumerate}
\item \(\esTam_\nu\) is an interval in \(\esTam_n\).
\item \(\esTam_\nu\) is a sublattice of \(\esTam_n\). 
\item \(\esTam_n = \sqcup_\nu \esTam_\nu\).
\end{enumerate}
\end{corollary}
\begin{proof}
This follows immediately from \Cref{lem:cong}. 
\end{proof}
%


We denote by \([3]^{n-1}\) the lattice obtained as the product of
\(n-1\) copies of the total order
\(\uparrow \leq \bullet \leq \leftarrow\).

\begin{lemma}
  The quotient of \(\esTam_n\) by the congruence relation \(\sim\) is
  isomorphic to \([3]^{n-1}\).
\end{lemma}
\begin{proof}
  The map \(T\mapsto \nu_T\) from \(\esTam_n\) to \([3]^{n-1}\) has
  fibers given by the equivalence classes of tableaux with respect to
  the relation \(\sim\). Moreover, the second item of \Cref{lem:meet}
  shows that this map is a homomorphism of meet-semilattices. If
  \(T_1\) and \(T_2\) are two fb-tableaux, then
  \(T_1 \lor T_2 = (T_1'\land T_2')'\). Therefore,
  \(\nu_{T_1\lor T_2}\) is obtained from \(\nu_{T_1'\land T_2'}\) by
  reversing the word and exchanging the two types of arrows.
  Exchanging the two types of arrows is a lattice anti-isomorphism on
  \([3]\). Hence, \(\nu_{T_1 \lor T_2}= \nu_{T_1}\lor \nu_{T_2}\), and
  therefore \(T\mapsto \nu_T\) is a lattice homomorphism. It remains
  to show that this map is surjective. Let \( \nu \) be a word in
  \([3]^{n-1}\). We construct a tableau as follows. Place the letters
  of \(\nu\) in the border cells of a tableau \(T\), from top to
  bottom, and place a dot in the root of \(T\). If a letter is an
  up-arrow or a dot, place a left-arrow in the same row and in the
  first column below the root of \(T\). If a letter is a dot or a
  left-arrow, place an up-arrow in the same column and in the first
  row, to the right of the root. The resulting tableau is a fb-tableau
  with a border equal to \(\nu\).
\end{proof}

\section{Cover relations}\label{sec:covers}

An up-arrow is called \defn{changeable of type I} if there is a cell
above it that is not pointed by a left-arrow. A left-arrow is called
\defn{changeable of type II} if the first free cell to its right
contains a \(\bullet\). A free cell is called \defn{changeable of type
  III} if it is empty.

Changeable cells of type I can occur in cells \((i,j)\) with
\(2\le i\le n\), \(i-1\le j\le n-1\). Changeable cells of type II can
occur in cells \((i,j)\) with \(2\le i\le n\), \(i\le j\le n\).
Changeable cells of type III can occur in cells \((i,j)\) with
\(2\le i\le n\), \(i\le j\le n-1\).

\begin{definition}\label{def:cover_relation}
  Let \(S\) be a fb-tableau, and let \(c\) be a changeable cell. The
  tableau \( T \) is obtained from \( S \) by applying one of the
  following three types of moves:
  \begin{itemize}
  \item A \defn{move of type I} moves a changeable \(\uparrow\) to the
    closest cell above it that is not pointed by a left-arrow. It
    creates a new free cell. If this cell is a border cell, then we
    fill it with a \(\bullet\); otherwise we leave it empty. For
    example,
    
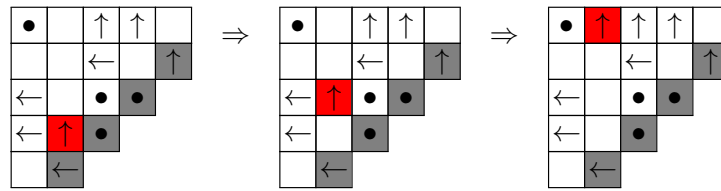
\[
      \begin{ytableau}
        \bullet & & \uparrow & \uparrow & \\
        & & \leftarrow & &*(gray)\uparrow\\
        \leftarrow &&\bullet &*(gray)\bullet \\
        \leftarrow &*(red){\uparrow} &*(gray)\bullet\\
        &*(gray)\leftarrow
      \end{ytableau}
      \quad{\rm \Rightarrow}\quad
      \begin{ytableau}
        \bullet & & \uparrow & \uparrow & \\
        & & \leftarrow & &*(gray)\uparrow\\
        \leftarrow &*(red){\uparrow}&\bullet &*(gray)\bullet \\
        \leftarrow & &*(gray)\bullet\\
        &*(gray)\leftarrow
      \end{ytableau}\quad{\rm \Rightarrow}\quad
      \begin{ytableau}
        \bullet &*(red){\uparrow} & \uparrow & \uparrow & \\
        & & \leftarrow & &*(gray)\uparrow\\
        \leftarrow &&\bullet &*(gray)\bullet \\
        \leftarrow & &*(gray)\bullet\\
        &*(gray)\leftarrow
      \end{ytableau}
    \]
  \item A \defn{move of type II} moves a changeable \(\leftarrow\) to
    the closest free cell to its right. For example,
    
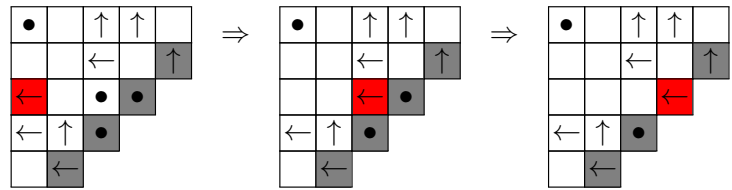
\[
      \begin{ytableau}
        \bullet & & \uparrow & \uparrow & \\
        & & \leftarrow & &*(gray)\uparrow\\
        *(red)\leftarrow &&\bullet &*(gray)\bullet \\
        \leftarrow &{\uparrow} &*(gray)\bullet\\
        &*(gray)\leftarrow
      \end{ytableau}
      \quad{\rm \Rightarrow}\quad
      \begin{ytableau}
        \bullet & & \uparrow & \uparrow & \\
        & & \leftarrow & &*(gray)\uparrow\\
        & &*(red)\leftarrow &*(gray)\bullet \\
        \leftarrow &{\uparrow}  &*(gray)\bullet\\
        &*(gray)\leftarrow
      \end{ytableau}\quad{\rm \Rightarrow}\quad
      \begin{ytableau}
        \bullet & & \uparrow & \uparrow & \\
        & & \leftarrow & &*(gray)\uparrow\\
        && &*(red)\leftarrow \\
        \leftarrow &\uparrow &*(gray)\bullet\\
        &*(gray)\leftarrow
      \end{ytableau}
    \]
  \item A \defn{move of type III} modifies an empty free cell by
    inserting a \( \bullet \) into it. For example,
    
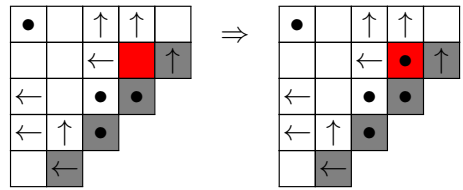
\[
      \begin{ytableau}
        \bullet & & \uparrow & \uparrow & \\
        & & \leftarrow &*(red){} &*(gray)\uparrow\\
        \leftarrow& &   \bullet &*(gray)\bullet\\ \leftarrow &{\uparrow} &*(gray)\bullet\\
        &*(gray)\leftarrow
      \end{ytableau}\quad{\rm \Rightarrow}\quad
      \begin{ytableau}
        \bullet & & \uparrow & \uparrow & \\
        & & \leftarrow & *(red)\bullet &*(gray)\uparrow\\
        \leftarrow &&\bullet &*(gray)\bullet \\
        \leftarrow &{\uparrow} &*(gray)\bullet\\
        &*(gray)\leftarrow
      \end{ytableau}\]
  \end{itemize}
\end{definition}

We say that two tableaux have the \defn{same up-arrows} if the
up-arrows in both tableaux are in the same position. We define
\defn{same left-arrows} and \defn{same arrows} similarly.

\begin{lemma}\label{lem:cover1}
  Let \(S\) be a fb-tableau and \(T\) the tableau obtained by
  applying one of the moves in \Cref{def:cover_relation}. Then,
\begin{enumerate}
\item The tableau \(T\) is a fb-tableau.
\item We have \(S\lhd T\).
\item The relation \(S\lhd T\) is a cover relation; that is, if
  \(U\) is a fb-tableau such that \(S\lhd U \lhd T\), then \(U=S\) or
  \(U=T\).
\end{enumerate}
\end{lemma}
\begin{proof}
  The moves do not empty the cohooks of the border cells, and the new
  position of the arrow is always compatible with the definition of fb-tableaux. Hence \(T\) is a fb-tableau.

  A move of type I preserves the row space and increases the column
  space. Moreover, it creates a new free cell, and this cell does not
  belong to \(F(S,T)\). A move of type II preserves the column space
  and decreases the row space. It remove one dot from the tableau
  \(S\), but this dot does not belong to \(F(S,T)\). A move of type
  III preserves both the row and column spaces, and adds a dot in an
  empty free cell of \(F(S,T)\). Therefore, in all three cases, we
  have \(S\lhd T\).

  It remains to prove that \(S\lhd T\) is a cover relation.
  Let \(S\lhd U \lhd T\).

  First, assume that \(T\) is obtained from \(S\) by a move of type
  III. Then \(S\) and \(T\) have the same column and row spaces, so
  the same is true for \(U\). In other words, the three tableaux have
  the same arrows, and hence the same free cells:
  \[
    F(S,U)=F(U,T)=F(S)=F(U)=F(T).
  \]

  The tableaux \(S\) and \(T\) differ at exactly one free cell
  \( c \), which is empty in \(S\) and contains a dot in \(T\). Let
  \( d \neq c \) be another free cell. If \( d \) contains a dot in
  \(S\), then it also contains a dot in \(U\). If \( d \) is empty in
  \(T\), it is also empty in \(U\). Thus, all free cells of \(U\)
  except \(c\) coincide with those of \(S\) and \(T\). Therefore,
  \( U = S \) or \( U = T \).

  Next, assume that \(T\) is obtained by a move of type II, and the
  left-arrow in cell \( c \) of \( S \) is moved to cell \( d \) in
  \( T \). Then the three tableaux have the same up-arrows, and
  \( S \) and \( T \) differ only in the row containing \( c \) and
  \( d \). Hence \( U \) coincides with \( S \) and \( T \) in every
  other row. Since \( S \) and \( T \) are identical on the part of
  this row strictly to the right of \( d \), the tableau \( U \) must
  also be identical to them there. Thus any difference must occur
  between \( c \) and \( d \). However, by definition, all cells
  between \( c \) and \( d \), except \( d \), are not free. Therefore
  \( U = S \) or \( U = T \).

  Finally, assume that \(T\) is obtained from \( S \) by a move of
  type I. Applying conjugation, we get \(T'\lhd U'\lhd S'\) and \(S'\)
  is obtained from \( T' \) by a move of type II. By the previous
  case, we obtain \(U'=S'\) or \(U'=T'\). Therefore, \( U = S \) or
  \( U = T \).
\end{proof}

\begin{lemma}\label{pro:cover}
  The cover relations of $\esTam_n$ are those described
  in~\Cref{def:cover_relation}.
\end{lemma}
\begin{proof}
  Let \(T_1\lhd T_2\) be a cover relation. Since \(T_1\neq T_2\),
  there are several cases to consider. In each case, we construct a
  tableau \(T\), obtained from \( T_1 \) by applying one of the moves
  in \Cref{def:cover_relation}, and show that \(T_1\lhd T \lhd T_2\).
  Since \(T_1 \lhd T_2\) is a cover relation, we obtain \(T=T_2\),
  which proves the lemma.

  We first assume that \(\col(T_1) =\col(T_2)\) and
  \(\row(T_1)=\row(T_2)\). In other words, the two tableaux have the
  same arrows. Then, we have \(F(T_1)= F(T_2) = F(T_1,T_2)\). The
  third condition in the definition of \( \lhd \) implies that \(T_1\)
  has strictly less dots than \(T_2\). Choose a free cell \(c\) that
  is empty in \(T_1\) and contains a \( \bullet \) in \(T_2\). Let
  \(T\) the tableau obtained from \(T_1\) by filling \(c\) with a
  \(\bullet\). Then \(T_1 \lhd T\lhd T_2\), so \(T = T_2\). Thus, in
  this case, \( T_2 \) is obtained from \( T_1 \) by a move of
    type III.

  Next, assume that \(\col(T_1) \neq \col(T_2)\). Then in some column
  the up-arrow of \( T_1 \) lies below that of \( T_2 \). Let \( u \)
  and \( v \) be the corresponding cells in \( T_1 \) and \( T_2 \),
  respectively. We claim that \(v\) is not pointed by a left-arrow in
  \(T_1\). Since \(T_2\) is a fb-tableau, there is a left-arrow (or
  the root) to the left of \(v\) in \(T_2\). Because
  \(\row(T_2)\subseteq \row(T_1)\), there is also a left-arrow (or the
  root) to the left of \(v\) in \(T_1\). Moreover, all cells strictly
  between \(v\) and \(u\) in that column are empty in \(T_1\). Hence
  we can move the up-arrow from \(u\) upward to \(v\), obtaining a
  tableau \(T\). By construction, \(T\) is obtained from \(T_1\) by
  applying a move of type I, so \(T_1\lhd T\). We also have
  \[
    \row(T_1)=\row(T)\supseteq \row(T_2).
  \]
  The up-arrows of \(T_1\) and \(T\) coincide except in the chosen
  column, where the up-arrow of \(T\) is at \(v\), as in \(T_2\). Thus
  \[
    \col(T_1)\subseteq \col(T)\subseteq \col(T_2).
  \]
  Finally, \(T_1\) and \(T\) also have the same dots. Let
  \(c\in F(T,T_2)\) be a cell containing a dot in \(T\). Then \(c\)
  does not lie between \(v\) and \(u\), so \(c\in F(T_1,T_2)\). Since
  \(T_1\lhd T_2\), the cell \(c\) must also contain a dot in \(T_2\).
  Therefore \(T_1\lhd T\lhd T_2\), so \(T=T_2\). Thus \(T_2\) is
  obtained from \(T_1\) by a move of type I.

  Finally, assume that \(\row(T_1)\neq \row(T_2)\). Let \(\sigma\) be
  the symmetry from \Cref{thm:selfdual}. Then
  \(\sigma(T_2)\lhd \sigma(T_1)\) is a cover relation, and
  \(\col(\sigma(T_2))\neq \col(\sigma(T_1))\). By the previous case,
  \(\sigma(T_1)\) is obtained from \(\sigma(T_2)\) by a move of
    type I. Applying \(\sigma\) again, we conclude that \(T_2\) is
  obtained from \(T_1\) by a move of type II.
\end{proof}

By the previous two lemmas, we obtain a characterization of the cover
relation.

\begin{theorem}\label{thm:cover_relation_iff}
  For two fb-tableaux \( S \) and \( T \), the relation \( S \lhd T \)
  is a cover relation if and only if \( T \) is obtained from \( S \)
  by one of the moves in~\Cref{def:cover_relation}.
\end{theorem}

Given a tableau \(T\), the number of tableaux that cover \(T\) is
exactly the number of its changeable cells. In the example below,
changeable cells are highlighted in green, and the free cells are
highlighted in blue:
 
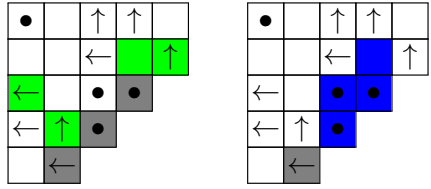
\[
 \begin{ytableau}
    \bullet & & \uparrow & \uparrow & \\
    & & \leftarrow & *(green){}&*(green)\uparrow\\
    *(green)\leftarrow &&\bullet &*(gray)\bullet \\
    \leftarrow &*(green)\uparrow &*(gray)\bullet \\
    &*(gray)\leftarrow
  \end{ytableau}\quad \quad
   \begin{ytableau}
    \bullet & & \uparrow & \uparrow & \\
    & & \leftarrow & *(blue){}&\uparrow\\
    \leftarrow &&*(blue)\bullet &*(blue)\bullet \\
    \leftarrow &\uparrow &*(blue)\bullet \\
    &*(gray)\leftarrow
  \end{ytableau}
\]

The following lemmas follow directly from the definitions.

\begin{lemma} \label{lem:changeable_up_iff}
An up-arrow is changeable if and only if it is not in the first row.
\end{lemma}

\begin{lemma} \label{lem:changeable_left_iff}
A left-arrow is changeable if and only if the first free cell to its right contains a dot. 
\end{lemma}

\begin{lemma}
  If, in the same row, there is a changeable cell of type II and
  another of type III or I, then there must be a dot between
  them.
\end{lemma}

\begin{theorem}\label{thm:determined}
  A fb-tableau \(T\) is uniquely determined by its changeable cells
  and their type.
\end{theorem}

\begin{proof}
  Recall that a fb-tableau contains \(n-1\) up-arrows and \(n-1\)
  left-arrows, together with a dot in cell \( (1,n) \).

  We first place the up-arrows. We put an up-arrow in each changeable
  cell of type I. By~\Cref{lem:changeable_up_iff}, every other
  up-arrow must be placed in the first row.

  Next we place the left-arrows. We put a left-arrow in each
  changeable cell of type II. In each of the remaining row, the
  left-arrow is then placed in the rightmost cell that is not pointed
  by an up-arrow and such that every cell to its left is empty and not
  changeable.

  Finally, we determine the free cells. Each changeable cell of type
  III is left empty, and every other free cell is filled with a dot.
\end{proof}

We present three examples illustrating the construction of the tableau once the positions and types of the changeable cells are given.
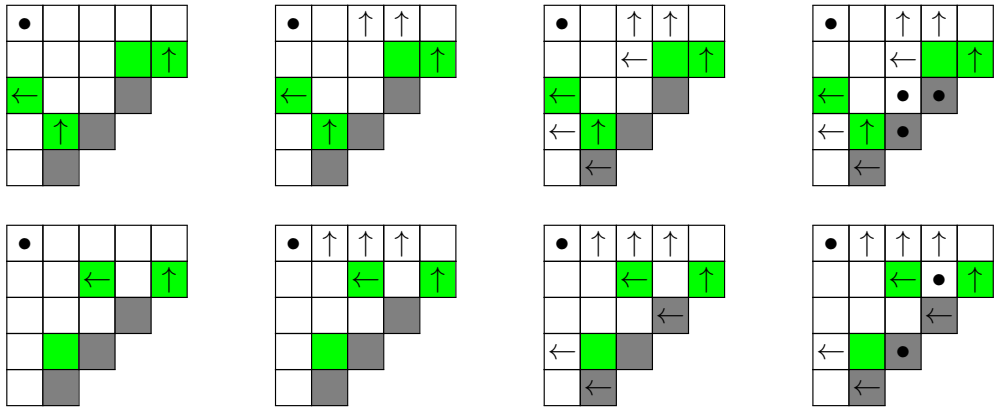
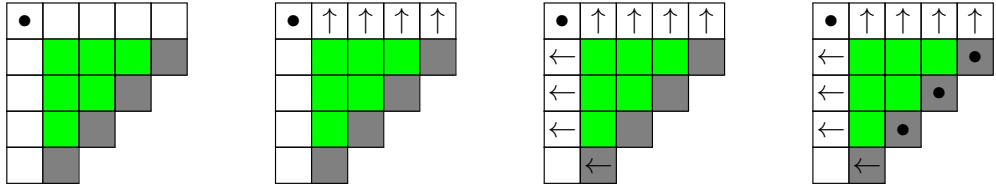
\[
 \begin{ytableau}
    \bullet & &  &  & \\
    & &  & *(green){}&*(green)\uparrow\\
    *(green)\leftarrow && &*(gray) \\
     &*(green)\uparrow &*(gray) \\
    &*(gray)
  \end{ytableau}
\quad\quad\quad
\begin{ytableau}
    \bullet & & \uparrow & \uparrow & \\
    & &  & *(green){}&*(green)\uparrow\\
    *(green)\leftarrow && &*(gray) \\
     &*(green)\uparrow &*(gray) \\
    &*(gray)
  \end{ytableau}\quad\quad\quad
\begin{ytableau}
    \bullet & & \uparrow & \uparrow & \\
    & &  \leftarrow & *(green){}&*(green)\uparrow\\
    *(green)\leftarrow && &*(gray) \\
     \leftarrow &*(green)\uparrow &*(gray) \\
    &*(gray)\leftarrow
  \end{ytableau}\quad\quad\quad
\begin{ytableau}
    \bullet & & \uparrow & \uparrow & \\
    & &  \leftarrow & *(green){}&*(green)\uparrow\\
    *(green)\leftarrow &&\bullet &*(gray) \bullet\\
     \leftarrow &*(green)\uparrow &*(gray)\bullet \\
    &*(gray)\leftarrow
  \end{ytableau}\]

\[
 \begin{ytableau}
    \bullet & &  &  & \\
    & &   *(green){\leftarrow}&&*(green)\uparrow\\
     && &*(gray) \\
     &*(green) &*(gray) \\
    &*(gray)
  \end{ytableau}
\quad\quad\quad \begin{ytableau}
    \bullet &\uparrow & \uparrow & \uparrow & \\
    & &   *(green){\leftarrow}&&*(green)\uparrow\\
     && &*(gray) \\
     &*(green) &*(gray) \\
    &*(gray)
  \end{ytableau}\quad\quad\quad \begin{ytableau}
    \bullet &\uparrow & \uparrow & \uparrow & \\
    & &   *(green){\leftarrow}&&*(green)\uparrow\\
     && &*(gray)\leftarrow \\
     \leftarrow &*(green) &*(gray) \\
    &*(gray)\leftarrow
  \end{ytableau}
\quad\quad\quad \begin{ytableau}
    \bullet &\uparrow & \uparrow & \uparrow & \\
    & &   *(green){\leftarrow}&\bullet&*(green)\uparrow\\
     && &*(gray)\leftarrow \\
     \leftarrow &*(green) &*(gray)\bullet \\
    &*(gray)\leftarrow
  \end{ytableau}
\]
\[
 \begin{ytableau}
    \bullet & &  &  & \\
    & *(green){}& *(green){} &*(green){} &*(gray)\\
    &*(green){}&*(green){} &*(gray) \\
     &*(green){} &*(gray) \\
    &*(gray)
  \end{ytableau}
\quad\quad\quad  
\begin{ytableau}
    \bullet &\uparrow & \uparrow & \uparrow & \uparrow\\
    & *(green){}& *(green){} &*(green){} &*(gray)\\
    &*(green){}&*(green){} &*(gray) \\
     &*(green){} &*(gray) \\
    &*(gray)
  \end{ytableau}\quad\quad\quad  
\begin{ytableau}
    \bullet &\uparrow & \uparrow & \uparrow & \uparrow\\
     \leftarrow& *(green){}& *(green){} &*(green){} &*(gray)\\
     \leftarrow &*(green){}&*(green){} &*(gray) \\
      \leftarrow &*(green){} &*(gray) \\
      &*(gray)\leftarrow
  \end{ytableau}\quad\quad\quad  
\begin{ytableau}
    \bullet &\uparrow & \uparrow & \uparrow & \uparrow\\
     \leftarrow& *(green){}& *(green){} &*(green){} &*(gray)\bullet\\
     \leftarrow &*(green){}&*(green){} &*(gray)\bullet \\
      \leftarrow &*(green){} &*(gray)\bullet \\
    &*(gray) \leftarrow 
  \end{ytableau}\]

\section{Irreducible elements}\label{sec:irred}

A tableau \( T \) is \defn{meet-irreducible} if it is covered by a
unique tableau \( S \). Dually, a tableau \( T \) is
\defn{join-irreducible} if it covers a unique tableau \( S \).

\begin{lemma}
  There is a bijection between the set of changeable cells together
  with their types and the meet-irreducible elements.
\end{lemma}
\begin{proof}
  By \Cref{thm:cover_relation_iff}, a fb-tableau is meet-irreducible
  if and only if it has exactly one changeable cell. By
  \Cref{thm:determined}, the tableau is uniquely determined by this
  changeable cell and its type.
\end{proof}

As a consequence, we obtain the following description of the
meet-irreducible tableaux.

\begin{proposition}
  The meet-irreducible tableaux of size \(n\) are exactly the following:
\begin{itemize}
\item Tableaux with a unique changeable up-arrow. Such an up-arrow can
  be placed in any cell except those in the first row or first column.
  Hence there are \( \binom{n}{2} \) such tableaux. For example,
  
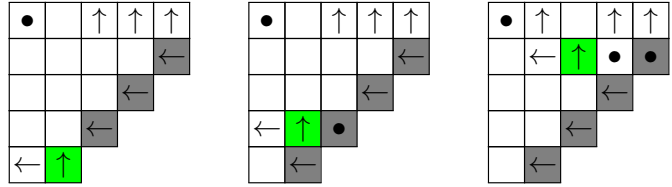
\[
    \begin{ytableau}
    \bullet & & \uparrow & \uparrow & \uparrow\\
     &&& &*(gray)\leftarrow\\
     & &&*(gray)\leftarrow \\
     & &*(gray)\leftarrow\\
    \leftarrow&*(green)\uparrow\\
  \end{ytableau}\quad \quad \begin{ytableau}
    \bullet &  &\uparrow & \uparrow & \uparrow\\
    & && &*(gray)\leftarrow\\
     && &*(gray)\leftarrow \\
     \leftarrow& *(green)\uparrow& *(gray)\bullet\\
    &*(gray)\leftarrow\\
  \end{ytableau}
  \quad \quad \begin{ytableau}
    \bullet &  \uparrow && \uparrow & \uparrow\\
    & \leftarrow&*(green)\uparrow& \bullet &*(gray)\bullet\\
     & &&*(gray)\leftarrow \\
     & &*(gray)\leftarrow\\
    &*(gray)\leftarrow\\
  \end{ytableau}
  \]
\item Tableaux with a unique changeable left-arrow. Such a left-arrow
  can be placed in any non-border cell outside the first row. Hence
  there are again \( \binom{n}{2} \) such tableaux. For example,
  
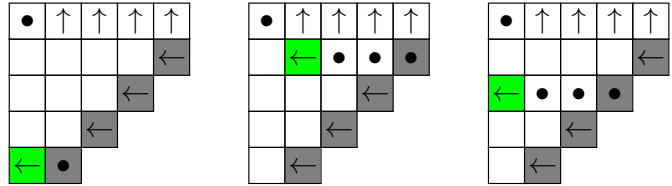
\[
    \begin{ytableau}
    \bullet &\uparrow & \uparrow & \uparrow & \uparrow\\
     &&& &*(gray)\leftarrow\\
     & &&*(gray)\leftarrow \\
     & &*(gray)\leftarrow\\
    *(green)\leftarrow&*(gray)\bullet\\
  \end{ytableau}\quad \quad \begin{ytableau}
    \bullet &  \uparrow &\uparrow& \uparrow & \uparrow\\
    & *(green)\leftarrow&\bullet& \bullet&*(gray)\bullet\\
     & &&*(gray)\leftarrow \\
     & &*(gray)\leftarrow\\
     &*(gray)\leftarrow\\
  \end{ytableau}
  \quad \quad \begin{ytableau}
    \bullet &  \uparrow &\uparrow& \uparrow & \uparrow\\   
    & & &&*(gray)\leftarrow \\
     *(green)\leftarrow&\bullet& \bullet&*(gray)\bullet\\
     & &*(gray)\leftarrow\\
    &*(gray)\leftarrow\\
  \end{ytableau}
  \]
\item Tableaux with a unique changeable free cell. This cell can be
  any empty free cell that is neither in the first row nor in the
  first column, and is not a border cell. Hence there are
  \( \binom{n-1}{2} \) such tableaux. For example,
  
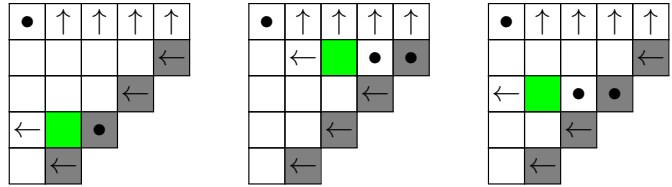
\[
    \begin{ytableau}
    \bullet &\uparrow & \uparrow & \uparrow & \uparrow\\
     &&& &*(gray)\leftarrow\\
     & &&*(gray)\leftarrow \\
    \leftarrow&*(green)&*(gray)\bullet\\
      &*(gray)\leftarrow\\  
  \end{ytableau}\quad \quad \begin{ytableau}
    \bullet &  \uparrow &\uparrow& \uparrow & \uparrow\\
    & \leftarrow&*(green) & \bullet&*(gray)\bullet\\
     & &&*(gray)\leftarrow \\
     & &*(gray)\leftarrow\\
    &*(gray)\leftarrow\\
  \end{ytableau}
  \quad \quad \begin{ytableau}
    \bullet &  \uparrow &\uparrow& \uparrow & \uparrow\\   
    & & &&*(gray)\leftarrow \\
     \leftarrow&*(green) & \bullet&*(gray)\bullet\\
     & &*(gray)\leftarrow\\
    &*(gray)\leftarrow\\
  \end{ytableau}
  \]
  \end{itemize}
\end{proposition}

\begin{corollary}\label{cor:meet_irr}
  There are \( 2 \binom{n}{2} + \binom{n-1}{2} \) meet-irreducible
  elements.
\end{corollary}

By \Cref{thm:selfdual}, the self-duality of the lattice gives an
immediate characterization of the join-irreducible elements.

\begin{corollary}\label{cor:join_irr}
  There are \( 2 \binom{n}{2} + \binom{n-1}{2} \) join-irreducible
  elements.
\end{corollary}

\begin{corollary}\label{cor:join-irr-description}
The join-irreducible elements are the tableaux such that:
\begin{itemize}
\item All left-arrows are in the first column. All up-arrows but one are in the border cells. The cell right below the up-arrow that is not in a border cell can contain a $\bullet$ or be empty and the others under are empty except the border cell that must contain a $\bullet$,
\item All left-arrows but one are in the first column. All up-arrows but one are in the border cells. The cell right above the left-arrow that is not the first column can contain an up-arrow.
The border cell under the left-arrow  must contain a $\bullet$.
\end{itemize}
\end{corollary}
See the tableaux in \Cref{fig:J_irr_poset} for examples of the
join-irreducible elements of \( \esTam_n \).
By~\Cref{cor:join-irr-description}, each join-irreducible has exactly
one up-arrow not on the border. We now introduce a notation for
join-irreducible elements.

\begin{definition}\label{def:join-irr}
  Let \( 1 \leq i \leq j < n \). We consider the following two cases.
  If \( i < j \), then
  \( \sigma \in \{ \emptyset, \leftarrow, \bullet \} \); if \( i=j \),
  then \( \sigma \in \{ \leftarrow, \bullet \} \). In either case, we
  define the tableau \( (i,j,\sigma) \) as follows. All up-arrows are
  placed on the border, except one at position \( (i,j) \). The cell
  immediately below \( (i,j) \) contains the entry \( \sigma \). All
  remaining left-arrows are placed in the leftmost column.
\end{definition}

Accordingly, we denote these join-irreducible elements by
\( (i,j,\emptyset) \), \( (i,j,\leftarrow) \), and
\( (i,j,\bullet) \). This notation will be used frequently in
Sections~\ref{sec:polygonality} and
\ref{sec:lattice-congruences}. For a join-irreducible element of the
form \( (i,j,\sigma) \) with \( i<j \), the third entry
\( \sigma \in \{\emptyset, \leftarrow, \bullet\} \) is called its
\defn{type}. When \( i=j \), the only possible join-irreducible
elements are \( (i,i,\bullet) \) and \( (i,i,\leftarrow) \). We regard
\( (i,i,\bullet) \) as a degenerate case of type \( \emptyset \), and
\( (i,i,\leftarrow) \) as being of type \( \leftarrow \). This will
occur in \Cref{prop:join-irred_poset}, and also in
\Cref{sec:lattice-congruences}. In order to visualize them easily, we
say that a join-irreducible of type~$\leftarrow$ is a \defn{red} join-irreducible.
The \defn{green} join-irreducibles are the ones of type $\emptyset$ and the
join-irreducibles $(i,i,\bullet)$. The other join-irreducibles of type $\bullet$ are
the \defn{blue} join-irreducible.

We can now describe the poset of join-irreducible elements.
\begin{proposition}\label{prop:join-irred_poset}
  The poset of join-irreducible elements of \(\esTam_n\) has \(n-1\)
  connected components, in bijection with the columns \(1,\dots,n-1\).
  The connected component corresponding to the column \(j\) contains
  all the join-irreducible elements \((i,j,\sigma)\) for
  \(i=1,\dots,j\). Moreover,
  \begin{enumerate}
  \item \((i,j,\emptyset)\lhd (i,j,\bullet) \lhd (i,j,\leftarrow)\);
  \item \((j,j,\bullet) \lhd (j-1,j,\emptyset)\lhd \dots \lhd (1,j,\emptyset)\).
  \end{enumerate}
\end{proposition}
\begin{proof}
  Let \((i,j,\sigma)\) and \((k,\ell,\sigma')\) be two
  join-irreducible elements. If \( \ell \neq j \), then by comparing
  their column spaces, we see that they are incomparable. Hence the
  join-irreducible elements can only be comparable when their unique
  non-border up-arrows lie in the same column.

  Now consider the case \( \ell = j \). Suppose that
  \((i,j,\sigma)\lhd (k,j,\sigma')\). Then the up-arrows of
  \( (i,j,\sigma) \) must lie weakly below those of
  \( (k,j,\sigma') \), and hence \( i \geq k \).

  If \( i=k \), then we immediately obtain
  \[
    (i,j,\emptyset)\lhd (i,j,\bullet) \lhd (i,j,\leftarrow).
  \]
  Now assume that \( i > k \).
  Then \( (i,j,\bullet) \) is not smaller than
  \( (k,j,\sigma') \), except in the case \( i=j \). Indeed,
  the cell containing the dot in \( (i,j,\bullet) \) is free in
  \( (k,j,\sigma') \), and therefore it must contain a dot there as
  well. Similarly, by comparing row spaces, we see that
  \( (i,j,\leftarrow) \) is not smaller than \( (k,j,\sigma') \)
  whenever \( i > k \).

  Thus the only remaining possible relations with \( i > k \) are
  among the tableaux of type \( \emptyset \) and \( (j,j,\bullet) \).
  It remains to prove that
  \[
    (j,j,\bullet) \lhd (j-1,j,\emptyset)\lhd \dots \lhd (1,j,\emptyset).
  \]
  This is immediate, since these tableaux have the same row spaces and
  no non-border dots.
\end{proof}

See~\Cref{fig:poset of join-irre} for an illustration of one connected
component.
\begin{figure}
   \[
    \xymatrix{
 & (i-1,i,\leftarrow) & (i-2,i,\leftarrow) &  & (1,i,\leftarrow) \\
(i,i,\leftarrow) & (i-1,i,\bullet) \ar@{->}[u] & (i-2,i,\bullet) \ar@{->}[u] &  & (1,i,\bullet) \ar@{->}[u] \\
(i,i,\bullet) \ar@{->}[u] \ar@{->}[r] & (i-1,i,\emptyset) \ar@{->}[u] \ar@{->}[r] & (i-2,i,\emptyset) \ar@{->}[r] \ar@{->}[u] & \dots \ar@{->}[r] & (1,i,\emptyset) \ar@{->}[u]
}
    \]
    \caption{A connected component in the poset of join-irreducible elements of $\esTam_n$.}
\label{fig:poset of join-irre}
\end{figure}
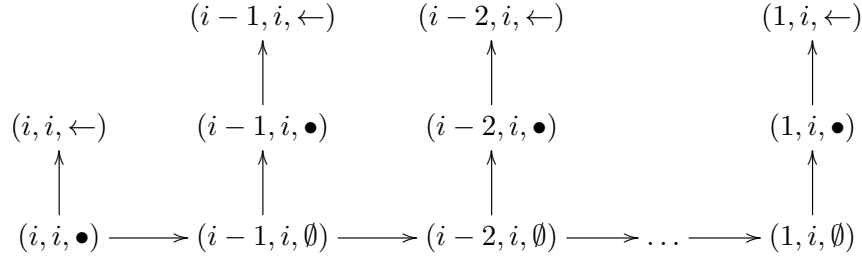

\section{Semidistributivity}\label{sec:semidistributivity}

A lattice \( L \) is said to be \defn{semidistributive} if, for all
\( x,y,z \in L \), the following hold:
\begin{itemize}
\item (join-semidistributivity) if \( x \vee z = y \vee z \), then
  \( (x \wedge y) \vee z = x \vee z \), and
\item (meet-semidistributivity) if \( x \wedge z = y \wedge z \), then
  \( (x \vee y) \wedge z = x \wedge z \).
\end{itemize}

In this section, we prove that \(\esTam_n\) is semidistributive. We
will use \Cref{lem:meet} and \Cref{lem:join}, which give explicit
description of the meets and joins in our lattice. 

To prove that our lattice is semidistributive, we first analyze
carefully what the condition \(x\land z = y\land z\) means. The result
will then follow from the fact that total orders are \defn{distributive
  lattices}.

\begin{lemma}\label{lem:sd1}
  Let \(x,y,z\) be fb-tableaux such that
  \[
    \col(z)\cap \col(x) = \col(z)\cap \col(y).
  \]
  Then the tableaux \(z\land x\), \(z\land y\) and
  \(z \land (x\lor y)\) have the same up-arrows.
\end{lemma}
\begin{proof}
  Fix an arbitrary column \( C \). For a tableau \( t \), let
  \( u_t \) denote the cell of \( C \) containing the up-arrow of
  \( t \). For two cells \( u \) and \( v \) in the same column, we
  write \( u < v \) if \( u \) lies below \( v \), and \( u \leq v \)
  if \( u \) lies weakly below \( v \).

  We first use the assumption of the lemma. Suppose that
  \( u_z \leq u_x \). Then \( u_z \) is the highest nonempty cell of
  \( C \) in \( \col(z) \cap \col(x) \). Since
  \( \col(z)\cap \col(x) = \col(z)\cap \col(y) \), the cell \( u_z \)
  is also the highest nonempty cell of \( C \) in
  \( \col(z) \cap \col(y) \). Hence \( u_z \leq u_y \).

  Next, suppose that \( u_x < u_z \). Then \( u_x \) is the highest
  nonempty cell of \( C \) in
  \( \col(z)\cap \col(x) = \col(z)\cap \col(y) \). Therefore
  \( u_x = u_y \).
  
  Thus exactly one of the following two cases occurs:
  \[
    u_z \leq u_x, \ u_z \leq u_y, \qquad \text{or} \qquad u_x = u_y < u_z.
  \]
  \begin{enumerate}
  \item Suppose that \( u_z \leq u_x \) and \( u_z \leq u_y \). Then
    \(u_{z\land x} = u_{z\land y}=u_z\). By construction,
    \(u_{x\lor y}\) lies weakly above both \( u_x \) and \( u_y \).
    Hence \(u_{z\land (x\lor y)} = u_z\).
  \item Suppose that \(u_x = u_y < u_z\). To determine
    \(u_{x\lor y}\), we look at the higher of the two up-arrows in
    \( x \) and \( y \), which is \( u_x = u_y \). In both \( x \) and
    \( y \), there is a left-arrow to the left of \( u_x \).
    Therefore, in \(x\lor y\), there is also a left-arrow to the left
    of \( u_x \). Hence the up-arrow of \( x \lor y \) is also in the
    cell \( u_x \). Thus \(u_{x\lor y} = u_{x} = u_y\). It follows
    that
    \(u_{z\land(x\lor y)} = u_x = u_y = u_{z\land x}=u_{z\land y}\).
  \end{enumerate}
  In either case, the tableaux \(z\land x\), \(z\land y\), and
  \(z \land (x\lor y)\) have their up-arrows in the same position in
  the column \( C \). This proves the lemma.
\end{proof}

\begin{lemma}\label{lem:sd2}
  Let \(x,y,z\) be fb-tableaux such that \(x\land z = y\land z\). Then
  the tableaux \(z\land (x\lor y)\) and \(z\land x\) have the same
  arrows.
\end{lemma}
\begin{proof}
  By~\Cref{lem:sd1}, the tableaux \(z\land (x\lor y)\) and
  \(z\land x\) have the same up-arrows. Thus it remains to show that
  they have the same left-arrows.

  Fix an arbitrary row \( R \), viewed as a total order from right to
  left. For a tableau \( t \), let \( \ell_t \) denote the cell of
  \( R \) containing the left-arrow of \( t \). With this order, the
  join is given by the minimum, so \( \ell_{x\lor y} = \min(\ell_x, \ell_y) \).

  For a tableau \( t \), define \( m_{z\land t} = \max(\ell_z,\ell_t) \),
  that is, the leftmost of \( \ell_z \) and \( \ell_t \). Then we have
  \[
    m_{z\land (x\lor y)} = \max(\ell_z,\ell_{x\lor y})= \max(\ell_z, \min(\ell_x,\ell_y)).
  \]
  The operations \( \min \) and \( \max \) satisfy
  \( \max(a, \min(b,c)) = \min(\max(a,b), \max(a,c)) \).
  Hence
  \[
    m_{z\land (x\lor y)} = \min(\max(\ell_z,\ell_x), \max(\ell_z,\ell_y)) = \min(m_{z\land x},m_{z\land y}).
  \]
  Since \( z\land x = z \land y \), we have
  \( m_{z\land (x\lor y)} = m_{z\land x} \).

  Let \( u \) be the first cell weakly to the left of
  \(m_{z\land(x\lor y)}\) that is not pointed by an up-arrow. Then, by
  the description of the meet in~\Cref{lem:meet},
  \( u = \ell_{z\land (x\lor y)} \). Since, by~\Cref{lem:sd1},
  \(z\land(x\lor y)\), \(z\land x\) have the same up-arrows, the cells
  of \( R \) not pointed by up-arrows are the same in both tableaux.
  Therefore, in the tableau $z\land x$, the first cell weakly to the left of \( m_{z\land x} \)
  that is not pointed by an up-arrow is also \( u \). Hence
  \[
    \ell_{z\land(x\lor y)} = \ell_{z\land x}.
  \]
  Since the row \(R\) was arbitrary, the tableaux \(z\land(x\lor y)\)
  and \(z\land x\) have the same left-arrows. This completes the
  proof.
 \end{proof}

\begin{lemma}\label{lem:sd3}
  Let \(x,y,z\) be three fb-tableaux be such that
  \(x\land z = y\land z\). Then the tableaux \(z\land (x\lor y)\) and
  \(z\land x\) have the same dots.
\end{lemma}
\begin{proof}
Since the tableaux $x\land (y\lor z)$ and $z\land x=z\land y$ have the same arrows, they also have the same free cells. Let $c$ be such a free cell. We show that $c$ is empty in $x\land (y\lor z)$ if and only if it is empty in $z\land x=z\land y$. Recall that a free cell in the meet of two fb-tableaux is empty if and only if it is free and empty in at least one of the tableaux.  

We first assume that $c$ is empty in $z\land (x\lor y)$. If $c$ is empty and free in $z$, then it is also empty in $z\land x=z\land y$. If $c$ is empty and free in $x\lor y$, then it is empty in both $x$ and $y$. Since $c$ is free in $x$ or in $y$, then $c$ is empty in $z\land x$ or in $z\land y$.  

Conversely, the cell $c$ is empty in $z\land x$ if and only if it is free and empty in $z$ or in $y$. In the first case, it is empty in $z\land (x\lor y)$. If it is free and empty in $x$, we consider $z\land y$ and we can assume that it is free and empty in $y$. Then, it is empty in both $x$ and $y$, so it is empty in $x\lor y$. Since it is free in $x$ and in $y$, it is also free in $x\lor y$. Hence, it is empty in $z\land (x\lor y)$. 
\end{proof}

\begin{theorem}\label{thm:semidistributive}
The lattice $\esTam_n$ is semidistributive.
\end{theorem}
\begin{proof}
Using \Cref{lem:sd1,lem:sd2,lem:sd3} we see that $\esTam_n$ is meet-semidistributive. We can now use the self-duality of \Cref{thm:selfdual} of our lattice to conclude. 
\end{proof}

\section{Trimness and spine}\label{sec:trim}

The notion of \defn{trim lattice} was introduced by Thomas in
\cite{Thomas2006}, as a natural non-graded analogue of distributive
lattices, preserving key structural properties. It has since been
shown that various generalizations of the Tamari lattice are trim. For
instance, all \defn{Cambrian lattices} of Dynkin types are trim, and
more generally, the lattice of \defn{torsion classes} of a
representation-directed algebra is trim \cite{thomas_williams}. In
this section, we show that the extra slow Tamari lattice is trim. We
also describe its spine and give a formula for the number of elements
on the spine.

\subsection{Trimness of the extra slow Tamari lattice.}

Let \( L \) be a lattice. An element \( x \in L \) is said to be
\defn{left modular} if
\begin{equation}\label{eq:condition_leftmodular}
  (y \wedge x) \vee z = y \vee (x \wedge z) \quad\text{for all \( y \leq z \)}.
\end{equation}
A lattice is called \defn{left modular} if it has a maximal chain of
elements that are all left modular.

In general, if a poset has a longest chain of \( m+1 \) elements, then
it has at least \( m \) join-irreducible elements and at least \( m \)
meet-irreducible elements. A lattice is said to be \defn{extremal} if the
length of its longest chain is equal to both the number of its
join-irreducible elements and the number of its meet-irreducible elements. A lattice
is called \defn{trim} if it is both left modular and extremal.

\begin{example}\label{exa:trim}
  Consider the poset given by
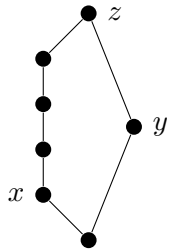
\[
 \begin{tikzpicture}[scale=.6]
   \tikzset{
     dot/.style={circle, fill=black, inner sep=0pt, minimum size=6pt},
     elabel/.style={midway, font=\scriptsize, inner sep=3pt} 
   }
  \def\h{5}
  \def\w{1}
  \node[dot] (x) at (0,0) {}; 
  \node[dot, label=left:\( x \)] (u1) at (-\w,\h/5) {};
  \node[dot] (u2) at (-\w,2*\h/5) {};
  \node[dot] (u3) at (-\w,3*\h/5) {};
  \node[dot] (u4) at (-\w,4*\h/5) {};
  \node[dot, label=right:\( z \)] (y) at (0,\h) {};
  \node[dot, label=right:\( y \)] (v) at (\w,\h/2) {};
  \draw (x) -- (u1) -- (u2) -- (u3) -- (u4) -- (y);
  \draw (x) -- (v) -- (y);
\end{tikzpicture}.
\]
The element \( x \) is left modular; this can be verified by checking
the defining condition~\eqref{eq:condition_leftmodular} for all pairs
\( y \leq z \). For example, for one particular choice of \( y \) and
\( z \) in the figure, we have
\[
  (y \wedge x) \vee z = y \vee (x \wedge z) = z.
\]
Checking this condition for all elements in a maximal chain shows that
the poset is left modular. Moreover, the length of a longest chain and
the number of join-irreducible elements are both equal to \( 5 \).
Hence, this poset is left modular and extremal, and therefore trim.
\end{example}

In \cite{thomas_williams}, Thomas and Williams proved the following
result, showing that to verify that a lattice is trim, it suffices to
check semidistributivity instead of left modularity.
\begin{proposition}[{\cite[Theorem~1.4]{thomas_williams}}]\label{pro:trim_condition}
  A lattice is trim if it is both semidistributive and extremal.
\end{proposition}

We prove that our poset is trim. Since the semidistributivity is
proved in \Cref{thm:semidistributive}, it remain only to show that it
is extremal.

\begin{proposition}\label{prop:extremal}
  The lattice $\esTam_n$ is extremal.
\end{proposition}
\begin{proof}
  By Corollaries~\ref{cor:meet_irr} and \ref{cor:join_irr}, the
  numbers of join-irreducible and meet-irreducible elements are both equal to
  \( 2 \binom{n}{2} + \binom{n-1}{2} \). Since the length of a longest
  chain is, in general, at most the number of join- or
  meet-irreducible elements, it suffices to show the existence of a 
  chain with exactly \( 2 \binom{n}{2} + \binom{n-1}{2} \) cover
  relations.

  Starting from \( \hat{0} \), that is, the tableau where all the left-arrows are in column $n$ and all the up-arrows are in border cells:
 
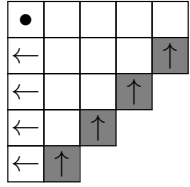
\[
   \begin{ytableau}
    \bullet & & &  & \\
     \leftarrow&&& &*(gray)\uparrow\\
     \leftarrow& &&*(gray)\uparrow \\
     \leftarrow& &*(gray)\uparrow\\
    \leftarrow&*(gray)\uparrow\\
  \end{ytableau}.
 \] 
 We move the leftmost up-arrow to the first row and all the left-arrows to the $(n-1)st$ column using cover relations.
 If the up-arrow is at a border cell,
 then moving it up simultaneously creates a bullet, so we can move right after the left-arrow in row $n$:
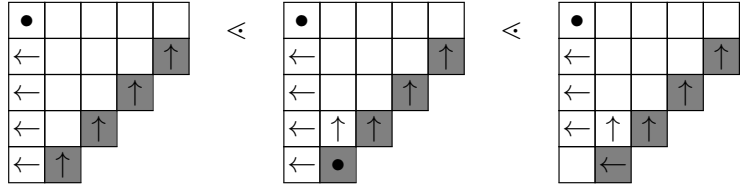
 \[
   \begin{ytableau}
    \bullet & & &  & \\
     \leftarrow&&& &*(gray)\uparrow\\
     \leftarrow& &&*(gray)\uparrow \\
     \leftarrow& &*(gray)\uparrow\\
    \leftarrow&*(gray)\uparrow\\
  \end{ytableau}
  \quad\lessdot\quad
   \begin{ytableau}
    \bullet & & &  & \\
     \leftarrow&&& &*(gray)\uparrow\\
     \leftarrow& &&*(gray)\uparrow \\
     \leftarrow& \uparrow &*(gray)\uparrow\\
    \leftarrow&*(gray)\bullet\\
  \end{ytableau}
  \quad\lessdot\quad
   \begin{ytableau}
    \bullet & & &  & \\
     \leftarrow&&& &*(gray)\uparrow\\
     \leftarrow& &&*(gray)\uparrow \\
     \leftarrow& \uparrow &*(gray)\uparrow\\
    &*(gray)\leftarrow\\
  \end{ytableau}
 \] 
 Otherwise, at each step we perform three actions in order: move the
 up-arrow up by one, place a bullet in its previous position, and move
 the left-arrow into the position of the bullet.
 
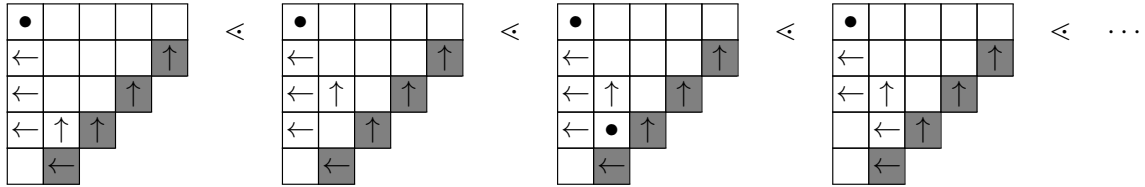
\[
   \begin{ytableau}
    \bullet & & &  & \\
     \leftarrow&&& &*(gray)\uparrow\\
     \leftarrow& &&*(gray)\uparrow \\
     \leftarrow& \uparrow &*(gray)\uparrow\\
    &*(gray)\leftarrow\\
  \end{ytableau}
  \quad\lessdot\quad
   \begin{ytableau}
    \bullet & & &  & \\
     \leftarrow&&& &*(gray)\uparrow\\
     \leftarrow& \uparrow &&*(gray)\uparrow \\
     \leftarrow&  &*(gray)\uparrow\\
    &*(gray)\leftarrow\\
  \end{ytableau}
  \quad\lessdot\quad
   \begin{ytableau}
    \bullet & & &  & \\
     \leftarrow&&& &*(gray)\uparrow\\
     \leftarrow& \uparrow &&*(gray)\uparrow \\
     \leftarrow& \bullet  &*(gray)\uparrow\\
    &*(gray)\leftarrow\\
  \end{ytableau}
  \quad\lessdot\quad
   \begin{ytableau}
    \bullet & & &  & \\
     \leftarrow&&& &*(gray)\uparrow\\
     \leftarrow& \uparrow &&*(gray)\uparrow \\
     & \leftarrow  &*(gray)\uparrow\\
    &*(gray)\leftarrow\\
  \end{ytableau}
  \quad\lessdot \quad\cdots 
 \]
We repeat this process for all the up-arrows.

 For a column of length \( m \geq 2 \), a total of
 \( 2(m-1) + (m-2) \) cover relations are applied to move the lowest
 up-arrow to the first row. Hence, we obtain a chain from
 \( \hat{0} \) to \( \hat{1} \) by performing
 \[
   \sum_{ m = 2 }^{ n } (2(m-1) + (m-2)) = 2\binom{n}{2} + \binom{n-1}{2}
 \]
 cover relations, which completes the proof.
\end{proof}

Using \Cref{thm:semidistributive} and \Cref{prop:extremal}, we obtain
the following result immediately from \Cref{pro:trim_condition}.

\begin{theorem}\label{thm:trim}
  The lattice $\esTam_n$ is trim.
\end{theorem}


\subsection{Elements on the spine}
\label{sec:spines}

In the previous subsection, we proved extremality by considering a
specific longest chain and its elements. We now consider the elements
in all longest chains. Let \( L \) be a trim lattice. The \defn{spine}
of \( L \) consists of the elements that belong to some longest chain
in \( L \). In this subsection, we give the description of the spine
of our poset. This notion, introduced in \cite{Thomas2006}, is
important for two main reasons: the spine is always a distributive
sublattice and its elements are the left modular elements of the
lattice.

Roughly speaking, the spine in our poset consists of the elements
encountered when moving from the least element \( \hat{0} \) to the
greatest element \( \hat{1} \) in the `slowest possible way'.
A precise description appears in the following lemma.

\begin{lemma}\label{lem:slowest_way}
  In $\esTam_n$, let a longest chain be
  \[
    \hat{0} = x_0 \lessdot x_1 \lessdot\cdots\lessdot x_{m-1} \lessdot x_m = \hat{1}. 
  \]
  Then, for every \( i = 1, \dots, m \), the cover relation
  \( x_{i-1} \lessdot x_i \) is one of the following three moves:
  \begin{enumerate}
  \item a move of type I, in which the changeable up-arrow moves upward by exactly one cell;
  \item a move of type II, in which the changeable left-arrow moves to the right by exactly one cell;
  \item a move of type III.
  \end{enumerate}
\end{lemma}
\begin{proof}
  First, by extremality in \Cref{prop:extremal} together with the
  enumeration of meet- and join-irreducible elements in
  Corollaries~\ref{cor:meet_irr} and \ref{cor:join_irr}, we know that
  \( m = 2 \binom{n}{2} + \binom{n-1}{2} \). Note that the elements
  \( \hat{0} \) and \( \hat{1} \) are of the forms
 
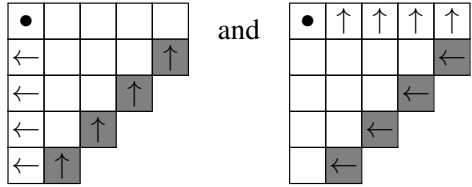
\[
   \begin{ytableau}
    \bullet & & &  & \\
     \leftarrow&&& &*(gray)\uparrow\\
     \leftarrow& &&*(gray)\uparrow \\
     \leftarrow& &*(gray)\uparrow\\
    \leftarrow&*(gray)\uparrow\\
  \end{ytableau} \qand
   \begin{ytableau}
    \bullet &\uparrow &\uparrow &\uparrow  &\uparrow \\
     &&& &*(gray)\leftarrow\\
     & &&*(gray)\leftarrow \\
     & &*(gray)\leftarrow\\
    &*(gray)\leftarrow\\
  \end{ytableau}.
 \] 

 A move of type I raises the up-arrow by at least one cell. More
 precisely, consider a column containing \( j \) cells. In
 \( \hat{0} \) the up-arrow lies on the border of this column, and in
 \( \hat{1} \) it lies in the top row; hence it can be raised at most
 \( j-1 \) times within that column, with equality only if each raise
 is by one cell. Hence, in any chain from \( \hat{0} \) to
 \( \hat{1} \), the number of type I moves is at most
 \( \sum_{ j = 1 }^{ n-1 } j = \binom{n}{2} \). Similarly, a move of
 type II shifts the left-arrow to the right by at least one cell, so
 the number of type II moves is also at most \( \binom{n}{2} \).
 Finally, type III moves can occur at most once in each interior cell,
 excluding the leftmost column, the top row, and the border cells,
 yielding at most \( \binom{n-1}{2} \) such moves in total.

 Therefore, the total number of cover relations in any chain is at
 most \( 2\binom{n}{2} + \binom{n-1}{2} \). Since our chosen chain is
 longest and has length \( m = 2\binom{n}{2} + \binom{n-1}{2} \), each
 of the three upper bounds above must be attained. In particular,
 every type I (resp. type II) cover must move the up-arrow (resp.
 left-arrow) by exactly one cell. This proves the claim.
\end{proof}

Let us recall from \Cref{def: cohook}, that the cohook \( H(i,j) \) at
the position \( (i,j) \) is the set consisting of the cell \( (i,j) \)
together with all cells to its left and above it.
\begin{definition}\label{def: forbidden cohook}
  In a fb-tableau, we say that
  \( H(i,j) \) is a \defn{forbidden cohook} if every cell in
  \( H(i,j) \) is empty.
\end{definition}

For example, there is a forbidden cohook \( H(2,4) \) in
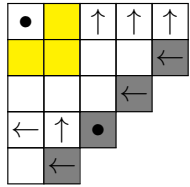
\[
  \begin{ytableau}
\bullet & *(yellow) & \uparrow &  \uparrow &  \uparrow\\
*(yellow) & *(yellow) & & &*(gray)\leftarrow\\
& & &*(gray)\leftarrow\\
\leftarrow  & \uparrow &*(gray)\bullet\\
 &*(gray)\leftarrow
\end{ytableau}.
\]

\begin{theorem}\label{thm:spine_iff}
  A fb-tableau is in the spine if and only if it does not contain a
  forbidden cohook.
\end{theorem}
\begin{proof}
  If a tableau contains a forbidden cohook \( H(i,j) \), then the
  left-arrow in the row \( i \) must have moved from the left of
  \( (i,j) \) to the right of \( (i,j) \) without passing through the
  cell \( (i,j) \). Hence, along any path from \( \hat{0} \) to this
  tableau, that left-arrow makes a move of length at least two.
  By~\Cref{lem:slowest_way}, any tableau with a forbidden cohook is
  not in the spine.

  Conversely, let \( T \) be a tableau with no forbidden cohook. We
  construct a path from \( \hat{0} \) to \( T \), and another from
  \( T \) to \( \hat{1} \), that proceeds in the slowest possible way.
  A path from \( \hat{0} \) to \( T \) is constructed as follows.
  First, move all up-arrows to positions of the up-arrows in \( T \)
  using only unit moves. Then repeat the following two unit steps
  until all left-arrows in their positions in \( T \): add a dot in
  the cell immediately to the right of a left-arrow, and move that
  left-arrow one cell to the right. These two steps are always
  possible since \(T\) does not contain a forbidden cohook. Finally,
  add dots to all cells that contain a dot in \( T \), thereby
  producing \( T \). The path from \(T\) to \( \hat{1} \) is obtained
  by applying the self-duality of \Cref{thm:selfdual} together with
  the previous construction.
\end{proof}

\subsection{A representation of the spine}\label{sec:spine_rep}

The spine of a trim lattice is a distributive lattice by \cite[Lemma
7]{Thomas2006}. Hence, by the Birkhoff's representation theorem, it is
isomorphic to the lattice of order ideals of its poset of
join-irreducible elements. This poset is the called \defn{Galois
  poset} of the trim lattice; see \cite[Proposition
2.4]{thomas_williams}.

\begin{lemma}
  Let \(L\) be a trim lattice. Then \(L\) and its spine have the same
  number of join- and meet-irreducible elements.
\end{lemma}
\begin{proof} 
  Since \( L \) is trim, it is extremal. By definition, this means
  that the length of a longest chain in \( L \) is equal to both the
  number of join- and meet-irreducible elements of \( L \).

  On the other hand, the spine of a trim lattice \( L \) is a
  distributive lattice. Every distributive lattice is extremal
  by~\cite{Markowsky1992}, so the length of a longest chain in the
  spine is equal to both the number of join- and meet-irreducible
  elements of the spine.

  By construction, the spine contains a longest chain of \( L \), so
  the spine and \( L \) have longest chains of the same length. The
  proof follows from extremality.
\end{proof}

As an immediate corollary, we obtain the following. 
\begin{corollary}\label{prop:irr_spine}
  The number of join- and meet-irreducible elements of the spine of
  \(\esTam_n\) is \(2{n\choose 2}+{n-1\choose 2}\).
\label{count}
\end{corollary}

If a tableau \(T\) is join-irreducible in \(\esTam_n\) and lies on the
spine, then it is also join-irreducible in the spine.
By~\Cref{cor:join-irr-description} and \Cref{def:join-irr}, the
join-irreducible elements labeled by \((i,j,\sigma)\) with
\(\sigma \neq \leftarrow\) lie on the spine. To avoid confusion, when
we view them as elements of the spine, we denote them by
\(S(i,j,\sigma)\).

For \(1\leq i\leq j <n \), let \(S(i,j,\leftarrow)\) be the following fb-tableau:
\begin{enumerate}
\item the left-arrow of the row \(i+1\) is in column \(j\), and all
  other left-arrows are in the leftmost column;
\item the up-arrows in columns \(n-1,n-2,\dots, j\) are in row \(i\),
  and all other up-arrows are in the border cells; all remaining
  border cells contain dots;
\item all remaining free cells are empty.
\end{enumerate}
Note that if \(i=n-1\), then this tableau coincides with the
join-irreducible \((i,j,\leftarrow)\); otherwise, the two tableaux are
different.
 
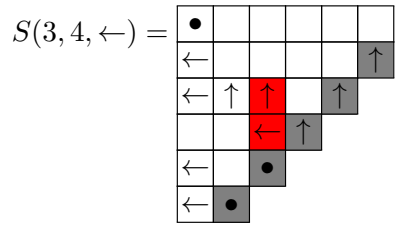
\[ 
S(3,4,\leftarrow) = \begin{ytableau}
    \bullet & & & & & \\
    \leftarrow &   & &  & &*(gray)\uparrow \\
     \leftarrow&  \uparrow& *(red)\uparrow & &  *(gray)\uparrow\\
      & & *(red)\leftarrow & *(gray)\uparrow \\
     \leftarrow& & *(gray)\bullet\\
     \leftarrow &*(gray)\bullet\\
  \end{ytableau}
\] 

\begin{lemma}\label{lem:jirr_spine}
  The join-irreducible elements of the spine are the following
  tableaux:
  \begin{enumerate}
  \item \(S(i,j,\emptyset)\) for \(1\leq i < j< n\);
  \item \(S(i,j,\bullet)\) for \(1\leq i\leq j < n\);
  \item \(S(i,j,\leftarrow)\) for \(1\leq i\leq j <n\). 
  \end{enumerate}
\end{lemma}
\begin{proof}
  By \Cref{prop:irr_spine}, we already know the number of
  join-irreducible elements of the spine, and we know that the first
  two families are join-irreducible in the spine. It therefore remains
  to check that \(S(i,j,\leftarrow)\) is on the spine and is
  join-irreducible.

  Since \(S(i,j,\leftarrow)\) has no forbidden cohook, it is on the
  spine. The only entries that can move downward are the up-arrows in
  columns \(n-1,\dots, j+1\) and the left-arrow in column \(j\).
  Moving the left-arrow one step to the left creates another
  fb-tableau on the spine. By contrast, moving one of those up-arrows
  downward to the first cell that is not pointed to by a left-arrow
  creates a smaller tableau with a forbidden cohook. Hence
  \(S(i,j,\leftarrow)\) covers a unique tableau of the spine, so it is
  join-irreducible.
\end{proof}

It remains to determine the poset induced by these elements. The
elements of type \(\emptyset\) and \(\bullet\) are join-irreducible in
\(\esTam_n\), so their relations are described
in~\Cref{prop:join-irred_poset}. We will show that the new elements of
type \(\leftarrow\) behave similarly to their counterparts
in~\(\esTam_n\), but they are no longer maximal. See \Cref{fig:spine}
for an illustration.

\begin{lemma}\label{lem:poset_irr_spine}
  Poset of join-irreducible elements of the spine is described as
  follows.
  \begin{enumerate}
  \item \(S(i,j,\leftarrow)\leq S(k,\ell,\sigma)\) if and only if
    \(\sigma = \leftarrow\), \(i=k\) and \(\ell\leq j\).
  \item \(S(i,j,\emptyset) \leq S(k,\ell,\leftarrow)\) if and only if
    \(k\leq i\) and \(\ell \leq j\).
  \item \(S(i,i,\bullet) \leq S(k,\ell,\leftarrow)\) if and only if
    \(k\leq i\) and \(\ell \leq i\).
  \item If \(i\neq j\), we have
    \(S(i,j,\bullet) \leq S(k,\ell,\leftarrow)\) if and only if
    \(i=k\) and \(\ell\leq j\).
  \item If \(\sigma \neq \leftarrow\), We have
    \(S(i,j,\emptyset) \leq S(k,\ell,\sigma)\) if and only if
    \(k \leq i\) and \(\ell = j\).
  \item If \(\sigma\neq \leftarrow\), we have
    \(S(i,j,\bullet) \leq S(k,\ell,\sigma)\) if and only if
    \(k= i, \ell = j\) and \(\sigma =\bullet\).
  \end{enumerate}
\end{lemma}
\begin{proof}
  For this proof, we ignore the top row, the leftmost column, and all
  border cells when considering row and column spaces. We refer to
  \Cref{sec:def} for the definitions of row and column spaces. We say
  that a row space is full if it contains every cell.

  The row and column spaces of \(S(i,j,\sigma)\) are described as
  follows. If \(\sigma \neq \leftarrow\), then its row space is full,
  and its column space is \(\{(m,j)\ |\  m \geq i\}\). If
  \(\sigma = \leftarrow\), then its row space is the complement of
  \(\{(i+1,m) \ |\ m \geq j\}\), and its column space is
  \(\{(m,m')\ |\ m \geq i \ \text{and}\ m' \geq j \}\).
\begin{enumerate}
\item Comparing row spaces, we see that
  \(S(i,j,\leftarrow) \leq S(k,\ell,\sigma)\) implies that
  \(\sigma= \leftarrow\) and \(i=k\) and \(j\geq \ell\). Conversely,
  if \(j\geq \ell\), then we have
  \[
    \row(S(i,\ell,\leftarrow)\subseteq \row(S(i,j,\leftarrow)),
    \qand
    \col(i,j,\leftarrow) \subseteq \col(i,\ell,\leftarrow).
  \]
  Since these tableaux do not contain free cells with a dot outside
  the border, this proves the first statement.
\item The row space of \(S(i,j,\emptyset)\) is full, and all its dots
  lie on the border. Hence
  \(S(i,j,\emptyset) \leq S(k,\ell,\leftarrow)\) holds if and only if
  \(\col(S(i,j,\sigma)) \subseteq \col(S(k,\ell,\leftarrow))\), which
  is equivalent to \( k \leq i \) and \( \ell \leq j \).
\item First assume that \(i\neq j\). If
  \(S(i,j,\bullet) \leq S(k,\ell,\leftarrow)\), then
  \(\col(S(i,j,\bullet)\subseteq \col(k,\ell,\leftarrow)\). Hence we
  have \(k\leq i\) and \(\ell\leq j\). In \(S(i,j,\bullet)\), there is
  a dot at \((i+1,j)\), whereas \(S(k,\ell,\leftarrow)\) has no dot
  there. Hence \((i+1,j)\) is not free in \(S(k,\ell,\leftarrow)\).
  All cells southwest of \((k,\ell)\) are free except those in row
  \(k+1\). Therefore \(i+1=k+1\), so \(i=k\). Conversely, these
  conditions imply \(S(i,j,\bullet)\leq S(k,\ell,\leftarrow)\).

  If \(i=j\), then the dot lies on the border. Since this cell also
  contains a dot in \(S(k,\ell,\leftarrow)\), we obtain the same
  condition as for join-irreducible of type \(\emptyset\).
\item Since the tableau \(S(i,j,\sigma)\) for
  \(\sigma\neq \leftarrow\) are join-irreducible elements of
  \(\esTam_n\), the last two statements are nothing but a
  reformulation of \Cref{prop:join-irred_poset}.
\end{enumerate}
\end{proof}

We consider the set of \defn{colored intervals} \([i,j]_\sigma\) of
size \( n \), where \( 1 \leq i \leq j \leq n-1 \) and
\(\sigma \in \{\mathtt{green},\mathtt{blue},\mathtt{red}\}\), which is
in bijection with the join-irreducible elements of the spine of
\(\esTam_n\). Using the opposite of \Cref{lem:poset_irr_spine}, we
endow it with the following partial order: for
\(1\leq i \leq j\leq n-1\),
\begin{enumerate}
\item \([i,j]_{\mathtt{red}}\preceq [k,\ell]_{\mathtt{red}}\) if and only
  if \(i=k\) and \(j\leq \ell\).
\item \([i,j]_{\mathtt{red}}\preceq [k,\ell]_\sigma\) for
  \(\sigma\neq\mathtt{red}\) if and only if \(i\leq k\) and
  \(j \leq \ell\).
\item \([i,j]_{\mathtt{blue}}\preceq [k,\ell]_\sigma\) if and only if
  \(\sigma\neq \mathtt{red}\), \(i\leq k\) and \(j=\ell\).
\item \([i,j]_{\mathtt{green}} \preceq [k,\ell]_\sigma\) if and only if
  \(\sigma=\mathtt{green}\), \(i\leq k\) and \(j=\ell\).
\end{enumerate}
We call this poset the \defn{product poset of colored intervals} of
size \(n\). We use the term `product' to emphasize that this poset is
a complicated version of the usual product order on the intervals:
\([i,j]\leq [k,\ell]\) if and only \(i\leq k\) and \(j\leq \ell\). In
\Cref{sec:poly-label}, we consider another partial order on the same
set, related to inclusion of intervals.

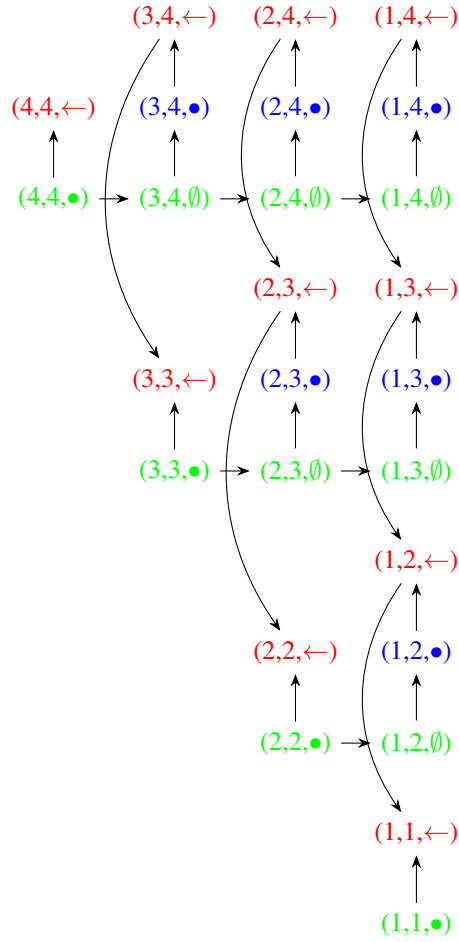
\begin{figure}[h!]
\begin{tikzpicture}[
>=Stealth,
x=1.6cm,
y=1.2cm,
every node/.style={font=\small}
]

\node (r34) at (1,0) {\color{red}(3,4,$\leftarrow$)};
\node (r24) at (2,0) {\color{red}(2,4,$\leftarrow$)};
\node (r14) at (3,0) {\color{red}(1,4,$\leftarrow$)};

\node (r44) at (0,-1) {\color{red}(4,4,$\leftarrow$)};
\node (b34) at (1,-1) {\color{blue}(3,4,$\bullet$)};
\node (b24) at (2,-1) {\color{blue}(2,4,$\bullet$)};
\node (b14) at (3,-1) {\color{blue}(1,4,$\bullet$)};

\draw[->] (b34) -- (r34);
\draw[->] (b24) -- (r24);
\draw[->] (b14) -- (r14);

\node (g44) at (0,-2) {\color{green}(4,4,$\bullet$)};
\node (g34) at (1,-2) {\color{green}(3,4,$\emptyset$)};
\node (g24) at (2,-2) {\color{green}(2,4,$\emptyset$)};
\node (g14) at (3,-2) {\color{green}(1,4,$\emptyset$)};

\draw[->] (g44) -- (r44);
\draw[->] (g34) -- (b34);
\draw[->] (g24) -- (b24);
\draw[->] (g14) -- (b14);

\draw[->] (g44) -- (g34);
\draw[->] (g34) -- (g24);
\draw[->] (g24) -- (g14);


\node (r23) at (2,-3) {\color{red}(2,3,$\leftarrow$)};
\node (r13) at (3,-3) {\color{red}(1,3,$\leftarrow$)};

\node (r33) at (1,-4) {\color{red}(3,3,$\leftarrow$)};
\node (b23) at (2,-4) {\color{blue}(2,3,$\bullet$)};
\node (b13) at (3,-4) {\color{blue}(1,3,$\bullet$)};

\node (g33) at (1,-5) {\color{green}(3,3,$\bullet$)};
\node (g23) at (2,-5) {\color{green}(2,3,$\emptyset$)};
\node (g13) at (3,-5) {\color{green}(1,3,$\emptyset$)};

\draw[->] (b23) -- (r23);
\draw[->] (b13) -- (r13);

\draw[->] (g33) -- (r33);
\draw[->] (g23) -- (b23);
\draw[->] (g13) -- (b13);

\draw[->] (g33) -- (g23);
\draw[->] (g23) -- (g13);


\node (r12) at (3,-6) {\color{red}(1,2,$\leftarrow$)};
\node (r22) at (2,-7) {\color{red}(2,2,$\leftarrow$)};
\node (b12) at (3,-7) {\color{blue}(1,2,$\bullet$)};
\node (g22) at (2,-8) {\color{green}(2,2,$\bullet$)};
\node (g12) at (3,-8) {\color{green}(1,2,$\emptyset$)};

\draw[->] (b12) -- (r12);
\draw[->] (g22) -- (r22);
\draw[->] (g12) -- (b12);
\draw[->] (g22) -- (g12);

\node (r11) at (3,-9) {\color{red}(1,1,$\leftarrow$)};
\node (g11) at (3,-10) {\color{green}(1,1,$\bullet$)};

\draw[->] (g11) -- (r11);

\draw[->, bend right=35] (r34) to (r33);
\draw[->, bend right=35] (r24) to (r23);
\draw[->, bend right=35] (r14) to (r13);

\draw[->, bend right=35] (r23) to (r22);
\draw[->, bend right=35] (r13) to (r12);

\draw[->, bend right=35] (r12) to (r11);

\end{tikzpicture}
\caption{Poset of the join-irreducible elements of the spine for n=4.}\label{fig:spine}
\end{figure}

\begin{theorem}\label{thm:spine_colored_intervals}
  The spine of \(\esTam_n\) is isomorphic to the lattice of order
  ideals on the opposite of the product poset of colored intervals of
  size \(n\).
\end{theorem}
\begin{remark}
  It is classical that the spine of the usual Tamari lattice
  \(\Tam_n\) is isomorphic to the lattice of order ideals of the
  product poset on the intervals of a total order with \(n-1\)
  elements; see, for example, \cite{thomas_williams} and
  \cite[Proposition 8.2]{zbMATH08165348}.
\end{remark}

\subsection{Counting the elements of the spine}

In this subsection, we count the number of elements in the spine using
the description obtained in the previous subsection.

For a fb-tableau, we associate its border to a word in
\(\{\bullet,\uparrow,\leftarrow\}\), obtained by listing in order the
entries in the border cells
\[
  (2,1),(3,2),\dots,(n,n-1),
\]
as in \Cref{sec:sublattices}. Define \(f(i,j,k)\), for \(i,k>0\) and
\(j\ge 0\) with \(i+j+k=n-1\), to be the number of tableaux in the
spine whose border has the form
\[
  \alpha_1\cdots \alpha_i\beta_1\cdots \beta_j\gamma_1\cdots \gamma_k,
\]
where \(\alpha_p\in\{\bullet,\uparrow\}\),
\(\beta_p=\bullet\), and
\(\gamma_p\in\{\bullet,\leftarrow\}\) for all \(p\), with the
additional conditions \(\alpha_i=\uparrow\) and
\(\gamma_1=\leftarrow\). Similarly, \(f(0,j,k)\), \(f(i,j,0)\), and
\(f(0,j,0)\) are defined in the same way, except that the first,
third, or both segments are omitted.

\begin{lemma}\label{lem:spine_uparrows}
  The number of fb-tableaux in the spine of size \(n\) with no
  \(\leftarrow\) on the border and whose border ends with
  \(\uparrow\) is equal to
  \[
    f(n-1,0,0)=2^{(n-1)(n-2)/2}.
  \]
\end{lemma}
\begin{proof}
  A fb-tableau counted by \(f(n-1,0,0)\) has border of the form
  \( \alpha_1\alpha_2\cdots \alpha_{n-1} \), where
  \(\alpha_p\in\{\bullet,\uparrow\}\) and \(\alpha_{n-1}=\uparrow\).
  Since the cell \((n,n-1)\) contains an up-arrow, all left-arrows
  must lie in the first column in order to avoid creating a forbidden
  cohook.

  For \(1\le i<n-1\), consider the column containing the border cell
  \((i+1,i)\). The cell \((i+1,i)\) either contains an up-arrow, or it
  contains a bullet with an up-arrow somewhere above it in the same
  column. If it contains a bullet and it contain an up-arrow in cell
  \((i-j,i)\) for \(0\le j< i\), the number of fillings of its column
  is \( 2^j\) if \( j \neq i-1\) and \(2^{i-1}\) if
  \( j = i-1 \). The total number is therefore
  \[
    1 + \sum_{j=0}^{i-2}2^j + 2^{i-1}=2^i.
  \]
  Multiplying the contributions obtained from each column yields the
  desired result.
\end{proof}

\begin{proposition}\label{pro:counting_spine}
  Define \(g(i,j,k)\) by the following recurrence relations. For
  \(i,k\ge 0\) and \(j>1\),
  \begin{equation}\label{eq:rec1_spine}
    g(i,j,k)
    =(2^{i+1}-1)g(i+1,j-1,k)
    +(2^{k+1}-1)g(i,j-1,k+1)
    -(2^{i+1}-1)(2^{k+1}-1)g(i+1,j-2,k+1).
  \end{equation}
  For \(i,k\ge 0\) and \(j=1\),
  \begin{equation}\label{eq:rec2_spine}
    g(i,1,k)
    =(2^{i+1}-1)g(i+1,0,k)
    +(2^{k+1}-1)g(i,0,k+1)
    -(2^{i+1}-1)(2^{k+1}-1)g(i,0,k).
  \end{equation}
  For \( i,j \geq 0 \) and \(j=0\),
  \begin{equation}\label{eq:rec3_spine}
    g(i,0,k)
    =(2^{i+1}-1)g(i,0,k-1)
    +(2^{k+1}-1)g(i-1,0,k) 
    -2(2^i-1)(2^k-1)g(i-1,0,k-1),
  \end{equation}
  with initial conditions \(g(i,j,k)=0\) for \(i<0\), \(j<0\), or
  \(k<0\), and \( g(k,0,0)=g(0,0,k)=1 \) for \(k\ge 0\). Then
  \begin{equation}\label{eq:f_wrt_g}
      f(i,j,k) = 2^{i(i-1)/2 + k(k-1)/2}\,g(i,j,k).
  \end{equation}
\end{proposition}
\begin{proof}

 \begin{figure}
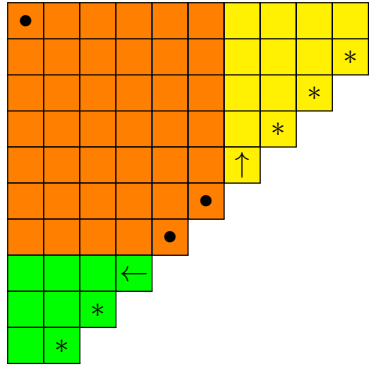

   \centering
\[
 \begin{ytableau}
   *(orange)\bullet &*(orange)&*(orange)&*(orange)&*(orange)&*(orange)&*(yellow)&*(yellow)&*(yellow)& *(yellow)  \\
   *(orange)&*(orange)&*(orange)&*(orange)&*(orange)&*(orange)&*(yellow)&*(yellow)&*(yellow)& *(yellow)*  \\
   *(orange)&*(orange)&*(orange)&*(orange)&*(orange)&*(orange)&*(yellow)&*(yellow)& *(yellow)* \\
   *(orange)&*(orange)&*(orange)&*(orange)&*(orange)&*(orange)&*(yellow)&*(yellow)* \\
   *(orange)&*(orange)&*(orange)&*(orange)&*(orange)&*(orange)&*(yellow)\uparrow \\
   *(orange)&*(orange)&*(orange)&*(orange)&*(orange)&*(orange)\bullet\\
   *(orange)&*(orange)&*(orange)&*(orange)&*(orange)\bullet\\
   *(green) &*(green)&*(green)&*(green)\leftarrow\\
   *(green) &*(green)&*(green)*\\
   *(green) &*(green)*
 \end{ytableau} 
\]
\caption{Illustration of the border cells of a tableau counted by
  \(f(i,j,k)\) with \(i=4\), \(j=2\), and \(k=3\). The tableau is
  divided into three regions according to \(i\), \(j\), and \(k\). A
  \( * \) in a yellow cell is either \(\bullet\) or \(\uparrow\),
  while a \( * \) in a green cell is either \(\bullet\) or
  \(\leftarrow\).}
   \label{fig:fixed_border}
 \end{figure}

  For the columns containing the border cells
  \[
    (2,1),(3,2),\dots,(i+1,i),
  \]
  if the cell \((i+1,i)\) contains an up-arrow, then by
  \Cref{lem:spine_uparrows} the number of possibilities for that
  portion of the tableau, highlighted in yellow in
  \Cref{fig:fixed_border}, is \( 2^{i(i-1)/2} \). Dually, for the rows
  containing the border cells
  \[
    (n,n-1),(n-1,n-2),\dots,(n-k+1,n-k),
  \]
  if the cell \((n-k+1,n-k)\) contains a left-arrow, then the number
  of possibilities for the green part in \Cref{fig:fixed_border} is
  \( 2^{k(k-1)/2} \). Let \(g(i,j,k)\) denote the number of
  possibilities for the remaining part, highlighted in orange in
  \Cref{fig:fixed_border}, after removing these two portions. This
  gives the formula \eqref{eq:f_wrt_g} for \( f(i,j,k) \) in terms of
  \( g(i,j,k) \).

  We now derive the recurrence for \(g(i,j,k)\). If both the rightmost
  column and the bottommost row contain more than one arrow, then the
  left-arrow in the rightmost column and the up-arrow in the bottommost
  row form a forbidden cohook. Hence at least one of them must contain
  exactly one arrow.

  \begin{enumerate}
  \item If the rightmost column contains exactly one arrow, then it is
    necessarily an up-arrow. The contribution of this case is
    \[
      (2^{i+1}-1)g(i+1,j-1,k).
    \]

  \item Similarly, if the bottommost row contains exactly one arrow,
    then it is necessarily a left-arrow. The contribution of this case is
    \[
      (2^{k+1}-1)g(i,j-1,k+1).
    \]

  \item If both conditions hold simultaneously, then the contribution
    of is
    \[
      (2^{i+1}-1)(2^{k+1}-1)g(i+1,j-2,k+1).
    \]
  \end{enumerate}

  Combining these three cases yields the
  recurrence~\eqref{eq:rec1_spine}. The
  recurrences~\eqref{eq:rec2_spine} and \eqref{eq:rec3_spine} can be
  obtained similarly.

  When \( j = 1 \), the first two cases are the same as before, while,
  in the third case, the rightmost column and the bottommost row share
  a border cell, which gives \eqref{eq:rec2_spine}.

  When \( j = 0 \), no border cell is involved. In the first case, the
  number of ways to fill the rightmost part of the orange region is
  again \( 2^{i+1}-1 \), and the number of possibilities for the
  remaining part is \( g(i,0,k-1) \). Thus the contribution is
  \( (2^{i+1}-1)g(i,0,k-1) \). Similarly, the second case contributes
  \( (2^{k+1}-1)g(i-1,0,k) \). In the third case, the rightmost column
  and the bottommost row again share a single cell. This cell contains
  a bullet or is empty, and the remaining parts contribute the factor
  \( (2^i-1)(2^k-1)g(i-1,0,k-1) \). This yields \eqref{eq:rec3_spine}.
\end{proof}

\begin{theorem}\label{thm:counting_spine}
  The number of fb-tableaux of size \(n\) in the spine is
  \[
    \sum_{i+j+k=n-1} f(i,j,k),
  \]
  where the numbers \(f(i,j,k)\) are defined in
  \Cref{pro:counting_spine}.
\end{theorem}

Using \Cref{thm:counting_spine}, we can explicitly compute the number
of fb-tableaux of size \(n\) in the spine for some fixed $n$. For the first few values of
\(n \geq 1 \), these numbers are
\[
  1, \ 3,\ 20,\ 288,\ 8672,\ 535680,\ 67162112,\ 16987521824.
\]

\section{Polygonality and congruence uniformity}
\label{sec:polygonality}

\subsection{Polygonality}

An interval \([x,y]\) in a lattice is the set
\(\{z\ |\ x\le z\le y\}\). A \defn{polygon} in a lattice is an interval
\( [x,y] \) that can be written as the union of two maximal chains
from \( x \) to \( y \), which intersect only at \( x \) and \( y \).
A lattice is said to be \defn{polygonal} if the following two
conditions hold:
\begin{enumerate}
\item If $y_1$ and $y_2$ are distinct and cover an element
$x$, then $[x, y_1 \vee y_2]$ is a polygon.
\item If $y_1$ and $y_2$ are distinct and are covered by an
element $x$, then $[y_1 \wedge y_2, x]$ is a polygon.
\end{enumerate}

By \Cref{thm:selfdual}, the lattice $\esTam_n$ is self-dual, we just need to prove the first condition.

\begin{proposition}\label{prop:polygonal}
The extra slow Tamari lattice $\esTam_n$ is polygonal.
\end{proposition}
\begin{proof}
  Let \(T\) be a tableau that has at least two changeable cells. Let
  us call two of those cells \(c_1\) and \(c_2\). Let \(T_1\) (resp.
  \(T_2\)) be the tableau obtained when we change \(c_1\) (resp.
  \(c_2\)). For any change started at \(c\), we call \(S_c\) the set
  of cells "travelled" by the move. For example if a left-arrow moves
  from \((i,j)\) to \((i,j-k)\) then
  \(S_c=\{(i,j),(i,j-1),\ldots ,(i,j-k)\}\). If \(S_{c_1}\) and
  \(S_{c_2}\) are disjoint then \(T_1\vee T_2\) is the tableau where
  both \(c_1\) and \(c_2\) are changed. In this case
  \([T,T_1\vee T_2]\) is a quadrilateral.

  Otherwise, the only possibility is that \(c_1\) is a change of type
  I and \(c_2\) is a change of type II and
  \(|S_{c_1}\cap S_{c_2}|=1\). On the following example, we just show
  the closest cells where the arrows can move. There could be other
  cells on the way but these are not free. Let $c_1$ (resp. $c_2$ be
  the cell filled by an up-arrow (resp. left-arrow). If we first move
  cell \(c_1\), we get
  
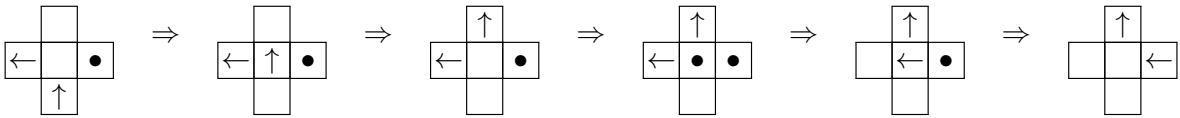
\[
    \begin{ytableau}
      \none{}& &\none{}\\
      \leftarrow & & \bullet\\
      \none[] &\uparrow  \\
    \end{ytableau}\quad\Rightarrow\quad
    \begin{ytableau}
      \none{}& &\none{}\\
      \leftarrow &\uparrow & \bullet\\
      \none[] & \\
    \end{ytableau}\quad\Rightarrow\quad
    \begin{ytableau}
      \none{}& \uparrow&\none{}\\
      \leftarrow & & \bullet\\
      \none[] & \\
    \end{ytableau}\quad\Rightarrow\quad
    \begin{ytableau}
      \none{}& \uparrow&\none{}\\
      \leftarrow &\bullet & \bullet\\
      \none[] & \\
    \end{ytableau}\quad\Rightarrow\quad
    \begin{ytableau}
      \none{}& \uparrow&\none{}\\
      &\leftarrow  & \bullet\\
      \none[] & \\
    \end{ytableau}\quad\Rightarrow\quad
    \begin{ytableau}
      \none{}& \uparrow&\none{}\\
      &&\leftarrow \\
      \none[] &  \end{ytableau}\]
  If we first move cell $c_2$, we get
  
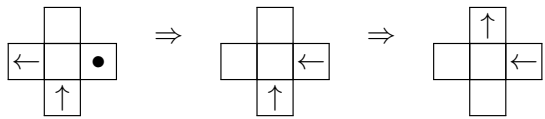
\[
    \begin{ytableau}
      \none{}& &\none{}\\
      \leftarrow & & \bullet\\
      \none[] &\uparrow  \\
    \end{ytableau}\quad\Rightarrow\quad
    \begin{ytableau}
      \none{}& &\none{}\\
      &&\leftarrow\\
      \none[] &\uparrow \\
    \end{ytableau}\quad\Rightarrow\quad
    \begin{ytableau}
      \none{}& \uparrow&\none{}\\
      &&\leftarrow \\
      \none[] & \\
    \end{ytableau}
  \]
  Therefore locally $T_1$, $T_2$ and $T_1\vee T_2$ look like:
  
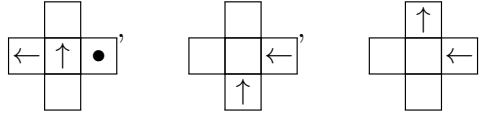
\[
    \begin{ytableau}
      \none{}& &\none{}\\
      \leftarrow & \uparrow & \bullet\\
      \none[] &  \\
    \end{ytableau},\quad\quad
    \begin{ytableau}
      \none{}& &\none{}\\
      &&\leftarrow\\
      \none[] &\uparrow \\
    \end{ytableau},\quad\quad
    \begin{ytableau}
      \none{}& \uparrow&\none{}\\
      &&\leftarrow \\
      \none[] & \\
    \end{ytableau}
  \]
  and $[T,T_1\vee T_2]$ is a heptagon, made of a path of length 2 and
  a path of length 5. \Cref{fig:polygons} illustrates the two types of
  polygons in our poset.
\begin{figure}
    \centering
    \begin{minipage}[t]{0.45\textwidth}
      \centering
    \begin{tikzpicture}[scale=.6]
      \tikzset{
        poset node/.style={inner sep=3pt, align=center, text height=2.5ex, text depth=0.8ex}, 
        lab/.style={midway, fill=white, inner sep=1pt, font=\scriptsize}
      }
      \def\h{10}
      \def\w{2}
      \node[poset node] (T) at (0,0) {$T$};
      \node[poset node] (T1) at (-\w,\h/2) {$T_1$};
      \node[poset node] (T2) at (\w,\h/2) {$T_2$};
      \node[poset node] (T12) at (0,\h) {$T_1\vee T_2$};
      \draw (T)  --  (T1) -- (T12);
      \draw (T)  --  (T2) -- (T12);
    \end{tikzpicture}
    \end{minipage}
    \hfill
    \begin{minipage}[t]{0.45\textwidth}
      \centering
    \begin{tikzpicture}[scale=.6]
      \tikzset{
        poset node/.style={inner sep=3pt, align=center, text height=2.5ex, text depth=0.8ex}, 
        lab/.style={midway, fill=white, inner sep=1pt, font=\scriptsize}
      }
      \def\h{10}
      \def\w{2}
      \node[poset node] (T) at (0,0) {$T$}; 
      \node[poset node] (T1) at (-\w,\h/5) {$T_1$};
      \node[poset node] (S1) at (-\w,2*\h/5) {$T_1'$};
      \node[poset node] (S2) at (-\w,3*\h/5) {$T_1''$};
      \node[poset node] (S3) at (-\w,4*\h/5) {$T_1'''$};
      \node[poset node] (T2) at (\w,\h/2) {$T_2$};
      \node[poset node] (T12) at (0,\h) {$T_1\vee T_2$};
      \draw (T)  --  (T1) -- (S1) -- (S2) -- (S3) -- (T12);
      \draw (T)  --  (T2) -- (T12);
    \end{tikzpicture}
    \end{minipage}
    \caption{The two types of polygons in our lattice; a quadrilateral and a heptagon.}
    \label{fig:polygons}
\end{figure}
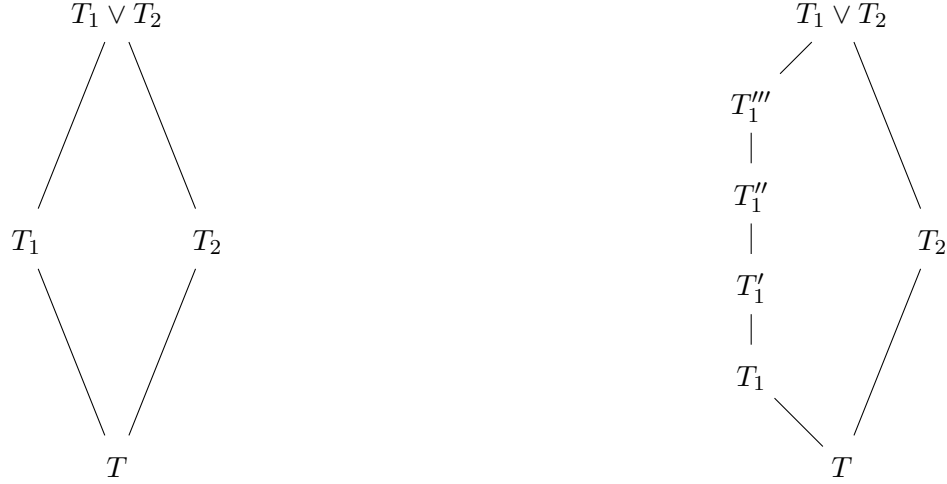
\end{proof}

\begin{remark}
  As a result of this subsection, our poset is polygonal and consists
  of quadrilaterals and heptagons (made of a path of length \( 2 \)
  and a path of length \( 5 \)) as in \Cref{fig:polygons}. The first
  cover relation of the path of length \( 2 \) within a heptagon
  produces a new cohook. Consequently, it follows immediately from
  \Cref{thm:spine_iff} that a tableau is in the spine if and only if
  it does not use any of the length 2 paths of the heptagons.
\end{remark}

\subsection{Polygonal labeling}
\label{sec:poly-label}

In this subsection, we prove that the extra slow Tamari lattice has a
polygonal labeling. This will be used in the next section to prove
that our poset is congruence uniform. Let \( L \) be a lattice, and
let \( \eta \) be a labelling of the cover relations of \( L \). Let
$r$ be a map from the labels of the cover relations to the set of
nonnegative integers. We say that \( (\eta,r) \) is a \defn{polygonal
  labeling} if the following condition holds: for every polygon
\( [x,y] \) of \( L \) and its two maximal chains
\[
  x = u_0 \lessdot u_1 \lessdot \cdots \lessdot u_{p+1} = y, \qand  x = v_0 \lessdot v_1 \lessdot \cdots \lessdot v_{q+1} = y,
\]
let \( \eta(u_i \lessdot u_{i+1}) = s_i \) and \( \eta(v_i \lessdot v_{i+1}) = t_i \).
Then 
we have
\begin{enumerate}
\item[(PL1)] \( \eta(x\lessdot u_1) = \eta(v_q\lessdot y)\) and \(\eta(x\lessdot v_1) = \eta(u_p \lessdot y) \), 
\item[(PL2)]
\(   r(s_0), r(s_p) < r(s_1), r(s_{p-1}) < r(s_2), r(s_{p-2}) < \cdots < r(s_{\lfloor p/2 \rfloor }), r(s_{p-\lfloor p/2 \rfloor}),
\)
\item[(PL3)] \(r(t_0), r(t_q) < r(t_1), r(t_{q-1}) < r(t_2), r(t_{q-2}) < \cdots < r(t_{\lfloor q/2 \rfloor}), r(t_{q - \lfloor q/2 \rfloor}) \).
\end{enumerate}
Here, the notation \( a,b < c,d \) means that \( a<c \), \( a<d \),
\( b<c \), and \( b<d \). The map $r$ is called the \defn{rank function} of the labeling. 

For example, consider the given polygon in which each edge is labelled
by the value of the corresponding rank function:
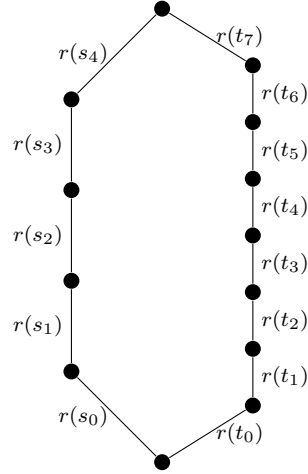
\[
 \begin{tikzpicture}[scale=.6]
   \tikzset{
     dot/.style={circle, fill=black, inner sep=0pt, minimum size=6pt},
     elabel/.style={midway, font=\scriptsize, inner sep=3pt} 
   }
  \def\h{10}
  \def\w{2}
  \node[dot] (x) at (0,0) {}; 
  \node[dot] (u1) at (-\w,\h/5) {};
  \node[dot] (u2) at (-\w,2*\h/5) {};
  \node[dot] (u3) at (-\w,3*\h/5) {};
  \node[dot] (u4) at (-\w,4*\h/5) {};
  \node[dot] (y) at (0,\h) {};
  \node[dot] (v1) at (\w,\h/8) {};
  \node[dot] (v2) at (\w,2*\h/8) {};
  \node[dot] (v3) at (\w,3*\h/8) {};
  \node[dot] (v4) at (\w,4*\h/8) {};
  \node[dot] (v5) at (\w,5*\h/8) {};
  \node[dot] (v6) at (\w,6*\h/8) {};
  \node[dot] (v7) at (\w,7*\h/8) {};
  \draw (x)
  -- node[elabel, left] {\( r(s_0) \)} (u1)
  -- node[elabel, left] {\( r(s_1) \)} (u2)
  -- node[elabel, left] {\( r(s_2) \)} (u3)
  -- node[elabel, left] {\( r(s_3) \)} (u4)
  -- node[elabel, left] {\( r(s_4) \)} (y);
  \draw (x)
  -- node[elabel, right] {\( r(t_0) \)} (v1)
  -- node[elabel, right] {\( r(t_1) \)} (v2)
  -- node[elabel, right] {\( r(t_2) \)} (v3)
  -- node[elabel, right] {\( r(t_3) \)} (v4)
  -- node[elabel, right] {\( r(t_4) \)} (v5)
  -- node[elabel, right] {\( r(t_5) \)} (v6)
  -- node[elabel, right] {\( r(t_6) \)} (v7)
  -- node[elabel, right] {\( r(t_7) \)} (y);
\end{tikzpicture}
\]
In this case, the inequality conditions become
\[
  r(s_0), r(s_4) > r(s_1), r(s_3) > r(s_2),
\]
and
\[
  r(t_0),r(t_7) > r(t_1),r(t_6) > r(t_2),r(t_5) > r(t_3),r(t_4).
\]
The first condition become $s_0 =t_7$ and $t_0=s_4$. 

In the rest of the section, we construct a polygonal labeling for the
extra slow Tamari lattice. By~\Cref{thm:semidistributive}, our lattice
is semidistributive, hence by \cite[Lemma~1.8]{AGT2003} the
cover relations can be labeled by the join-irreducibles of the
lattice. We start by recalling this labeling.

By definition, for a join-irreducible \( \jj \), there exists a unique element \( \jj_* \) such that \( \jj_* \lessdot \jj \). If \( L \) is a finite join-semidistributive lattice, then for each cover relation
\( x \lessdot y \) there exists a unique join-irreducible \( \jj \) of
\( L \) satisfying
\[
  y = x \vee \jj, \qand x = x \vee \jj_*.
\]
We then label the edge \( x \lessdot y \) by \( \jj \), and write
\( \eta(x\lessdot y) = \jj \). Let us call the assignment of such labels
to all cover relations the \defn{join-labelling} \( \eta \). Recall \Cref{def:join-irr} that the join-irreducible elements are written as
\[
  (i,j,\emptyset), (i,j,\leftarrow), \qand (i,j,\bullet),
\]
where \( (i,j) \) is the position of the unique up-arrow not on the
border, and the third entry indicates whether the cell directly below
\( (i,j) \) is empty, contains a left-arrow, or contains a dot.

\begin{lemma}\label{lem:join-HH-labelling}
  The join-labelling \( \eta \) of a finite semidistributive lattice
  satisfies the condition (PL1).
\end{lemma}
\begin{proof}
Let \( [x,y] \) be a polygon, and let
\[
  x = u_0 \lessdot u_1 \lessdot \cdots \lessdot u_{p+1} = y, \qand  x = v_0 \lessdot v_1 \lessdot \cdots \lessdot v_{q+1} = y
\]
be its two maximal chains.

Let \(\jj\) be the join-label of \(x\lessdot v_1\). Then
\(x\lor \jj = v_1\) and \(\jj_* \leq x\). Hence
\[
  v_1 = x\lor \jj \leq u_p \lor \jj \leq y.
\]
It follows that \(u_p \lor \jj\) belongs to the maximal chain
\(v_1\lessdot v_2\lessdot \dots \lessdot y\). If
\(u_p \lor \jj= v_k\), then \(u_p \leq v_k\). Since \(u_p\) is not
smaller than any of the \(v_k\) for \(k\leq q\), we get
\(u_p\lor \jj = y\). Moreover, since \(\jj_* \leq x\leq u_p\), it
follows that \(\jj\) is the join-label of \(u_p \lessdot y\).

For the same reason, the join-label of \( x\lessdot u_1 \) is the same
as that of \( v_q \lessdot y \), which completes the proof.
\end{proof}
We now define the rank function for $\jj=(i,j,\sigma)$ by 
\begin{equation}\label{eq:rank_func}
r(\jj) = r(i,j,\sigma) = j-i +\delta_{\sigma,\bullet}.
\end{equation}
\begin{lemma}\label{lem:HH-rank_func}
  The pair \( (\eta,r) \) is a polygonal labeling for $\esTam_n$.
\end{lemma}
\begin{proof}
  In a side of length $2$ of a polygon, the condition on the rank
  function is empty, so we only need to look at the long side of the
  heptagons in our lattice.

  Let \( [S,T] \) be a heptagon. In this heptagon, let
  \( S \lessdot U_1 \lessdot U_2 \lessdot U_3 \lessdot U_4 \lessdot T
  \) be a chain of length \( 5 \), and \( S \lessdot V \lessdot T \) a
  chain of length \( 2 \). As in the proof of
  Proposition~\ref{prop:polygonal}, we draw only the closest cells
  where arrows may move. Let \( (i,j) \) be the position of the middle
  cell among the five cells of \( S \). Let \( (i,j') \) and
  \( (i',j) \) be the positions of the cells containing the dot, and
  the up-arrow, respectively, for some \( i' > i \) and \( j' < j \).

  Pictorially, the five cells of the part of \( S \) under
  consideration (on the left) and the positions of three cells among
  them (on the right) can be represented as follows:
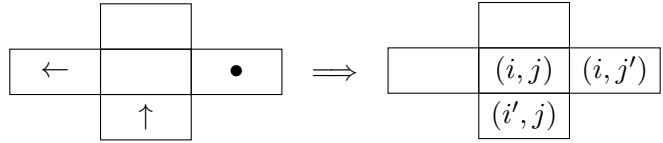
\[
  \begin{tikzpicture}[scale = .6]
    \def\h{1}
    \def\w{2}
    \newcommand{\putcell}[3]{%
      \draw
      ({(#2-1)+0.5-\w/2},{-(#1-1)-0.5-\h/2})
      rectangle
      ({(#2-1)+0.5+\w/2},{-(#1-1)-0.5+\h/2});
      \node
      at ({(#2-1)+0.5},{-(#1-1)-0.5}) {$#3$};
    }\putcell{1}{1}{\leftarrow}
    \putcell{1}{3}{}
    \putcell{1}{5}{\bullet}
    \putcell{0}{3}{}
    \putcell{2}{3}{\uparrow}
  \end{tikzpicture}
  \quad\raisebox{.8cm}{\( \Longrightarrow \)}\quad
  \begin{tikzpicture}[scale = .6]
    \def\h{1}
    \def\w{2}
    \newcommand{\putcell}[3]{%
      \draw
      ({(#2-1)+0.5-\w/2},{-(#1-1)-0.5-\h/2})
      rectangle
      ({(#2-1)+0.5+\w/2},{-(#1-1)-0.5+\h/2});
      \node
      at ({(#2-1)+0.5},{-(#1-1)-0.5}) {$#3$};
    }\putcell{1}{1}{}
    \putcell{1}{3}{(i,j)}
    \putcell{1}{5}{(i,j')}
    \putcell{0}{3}{}
    \putcell{2}{3}{(i',j)}
  \end{tikzpicture}
\]
  Then, using~\Cref{lem:join}, we can compute the join of two
  elements. For instance, let \( \jj \) be a join-irreducible of the
  form \( (i-1,j',\leftarrow) \). Then \( \jj_* = (i-1,j',\bullet) \),
  and we have
  \[
    V = S \vee \jj, \quad S = S \vee \jj_*.
  \]
  It follows that \( \eta(S \lessdot V) = (i-1, j', \leftarrow) \). In
  the same way, we obtain
  \begin{equation}\label{eq:heptagon_HH}
  \begin{aligned}
    \eta(S \lessdot U_1) &= \eta(V \lessdot T) = (i'-1,j,\emptyset), \\
    \eta(U_1 \lessdot U_2) &= (i-1,j,\emptyset), \\
    \eta(U_2 \lessdot U_3) &= (i-1,j,\bullet), \\
    \eta(U_3 \lessdot U_4) &= (i-1,j,\leftarrow), \\
    \eta(U_4 \lessdot T) &= \eta(S \lessdot V) = (i-1,j',\leftarrow).
  \end{aligned}
  \end{equation}
  See Figure~\ref{fig:heptagon} for an illustration.

  \tikzset{
    poset node/.style={inner sep=4pt, align=center}, 
    lab/.style={midway, fill=white, inner sep=1pt, font=\scriptsize}
  }

  \ytableausetup{boxsize = 1.1em}

  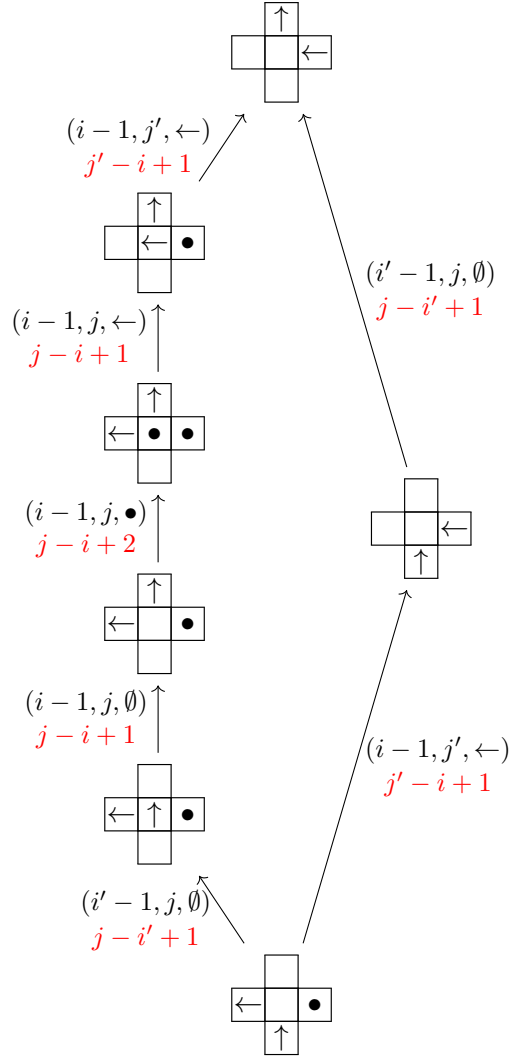
\begin{figure}
    \centering
    \begin{tikzpicture}[scale=.7]

      \def\s{1.2}
      \node[poset node] (S) at (0,0) {
        $\begin{ytableau}
          \none{}& &\none{}\\
          \leftarrow & & \bullet\\
          \none[] &\uparrow  \\
        \end{ytableau}$
      };
      \node[poset node] (T) at (0,15*\s) {
        $\begin{ytableau}
          \none{}& \uparrow&\none{}\\
          &&\leftarrow \\
          \none[] &  \end{ytableau}$
      };

      \node[poset node] (U1) at (-2*\s,3*\s) {
        $\begin{ytableau}
          \none{}& &\none{}\\
          \leftarrow &\uparrow & \bullet\\
          \none[] & \\
        \end{ytableau}$
      };
      \node[poset node] (U2) at (-2*\s,6*\s) {
        $\begin{ytableau}
          \none{}& \uparrow&\none{}\\
          \leftarrow & & \bullet\\
          \none[] & \\
        \end{ytableau}$
      };
      \node[poset node] (U3) at (-2*\s,9*\s) {
        $\begin{ytableau}
          \none{}& \uparrow&\none{}\\
          \leftarrow &\bullet & \bullet\\
          \none[] & \\
        \end{ytableau}$
      };
      \node[poset node] (U4) at (-2*\s,12*\s) {
        $\begin{ytableau}
          \none{}& \uparrow&\none{}\\
          &\leftarrow  & \bullet\\
          \none[] & \\
        \end{ytableau}$
      };

      \node[poset node] (V)  at (2.2*\s,7.5*\s) {
        $\begin{ytableau}
          \none{}& &\none{}\\
          &&\leftarrow\\
          \none[] &\uparrow \\
        \end{ytableau}$
      };

      \draw[->] (S)  -- node[left,yshift=-4pt] {\small\shortstack{\((i'-1,j,\emptyset)\)\\\red{\( j-i'+1 \)}}} (U1);
      \draw[->] (U1) -- node[left] {\small\shortstack{\((i-1,j,\emptyset)\)\\\red{\( j-i+1 \)}}} (U2);
      \draw[->] (U2) -- node[left] {\small\shortstack{\((i-1,j,\bullet)\)\\\red{\( j-i+2 \)}}} (U3);
      \draw[->] (U3) -- node[left] {\small\shortstack{\((i-1,j,\leftarrow)\)\\\red{\( j-i+1 \)}}} (U4);
      \draw[->] (U4) -- node[left] {\small\shortstack{\((i-1,j',\leftarrow)\)\\\red{\( j'-i+1 \)}}} (T);

      \draw[->] (S) -- node[right] {\small\shortstack{\((i-1,j',\leftarrow)\)\\\red{\( j'-i+1 \)}}} (V);
      \draw[->] (V) -- node[right] {\small\shortstack{\((i'-1,j,\emptyset)\)\\\red{\( j-i'+1 \)}}} (T);

    \end{tikzpicture}
    \caption{A heptagon. The black labels on the edges are the labels
      by the edge-labelling, and the red labels indicate the values of
      the rank function associated with each edge label. The middle
      cell is \( (i,j) \), with the cell to its right being
      \( (i,j') \) and the cell below it being \( (i',j) \). }
    \label{fig:heptagon}
  \end{figure}

 Since \( i' > i \) and
  \( j' < j \), the labels in \eqref{eq:heptagon_HH} satisfy
  \[
    r(\eta(S \lessdot U_1)), r(\eta(U_4 \lessdot T)) < r(\eta(U_1 \lessdot U_2)), r(\eta(U_3 \lessdot U_4)) < r(\eta(U_2 \lessdot U_3)).
  \]
  In~\Cref{fig:heptagon}, the red
  labels indicate the values of the rank function associated each edge
  label.
\end{proof}

By~\eqref{eq:heptagon_HH} in the proof above, the join-label of a
cover relation in \(\esTam_n\) is determined by the type of move as
follows. We indicate the label above the arrow \( \Rightarrow \) in
the examples in~\Cref{def:cover_relation}.
\begin{itemize}
\item For a move of type I, let \( (i,j) \) be the position of the
  changeable up-arrow. If \( (i,j) \) lies on the border, then the
  label is \( (i-1,j,\bullet) \); otherwise, it is
  \( (i-1,j,\emptyset) \).
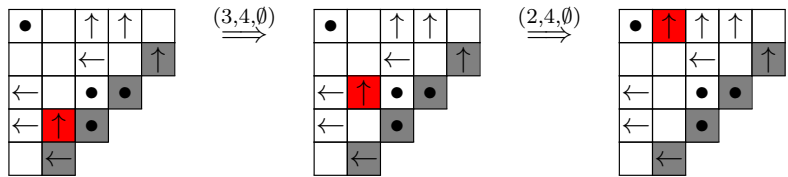
  \[
    \begin{ytableau}
      \bullet & & \uparrow & \uparrow & \\
      & & \leftarrow & &*(gray)\uparrow\\
      \leftarrow &&\bullet &*(gray)\bullet \\
      \leftarrow &*(red){\uparrow} &*(gray)\bullet\\
      &*(gray)\leftarrow
    \end{ytableau} 
    \quad\overset{(3,4,\emptyset)}{\Longrightarrow}\quad
    \begin{ytableau}
      \bullet & & \uparrow & \uparrow & \\
      & & \leftarrow & &*(gray)\uparrow\\
      \leftarrow &*(red){\uparrow}&\bullet &*(gray)\bullet \\
      \leftarrow & &*(gray)\bullet\\
      &*(gray)\leftarrow
    \end{ytableau}\quad\overset{(2,4,\emptyset)}{\Longrightarrow}\quad
    \begin{ytableau}
      \bullet &*(red){\uparrow} & \uparrow & \uparrow & \\
      & & \leftarrow & &*(gray)\uparrow\\
      \leftarrow &&\bullet &*(gray)\bullet \\
      \leftarrow & &*(gray)\bullet\\
      &*(gray)\leftarrow
    \end{ytableau}
  \]
\item For a move of type II, let \( (i,j) \) be the position of the
  closest free cell, to the right of the changeable left-arrow, that
  contains a \( \bullet \). Then the label is
  \( (i-1,j,\leftarrow) \).
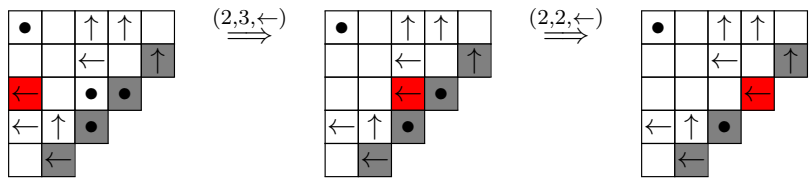
  \[
    \begin{ytableau}
      \bullet & & \uparrow & \uparrow & \\
      & & \leftarrow & &*(gray)\uparrow\\
      *(red)\leftarrow &&\bullet &*(gray)\bullet \\
      \leftarrow &{\uparrow} &*(gray)\bullet\\
      &*(gray)\leftarrow
    \end{ytableau}
    \quad\overset{(2,3,\leftarrow)}{\Longrightarrow}\quad
    \begin{ytableau}
      \bullet & & \uparrow & \uparrow & \\
      & & \leftarrow & &*(gray)\uparrow\\
      & &*(red)\leftarrow &*(gray)\bullet \\
      \leftarrow &{\uparrow}  &*(gray)\bullet\\
      &*(gray)\leftarrow
    \end{ytableau}\quad\overset{(2,2,\leftarrow)}{\Longrightarrow}\quad
    \begin{ytableau}
      \bullet & & \uparrow & \uparrow & \\
      & & \leftarrow & &*(gray)\uparrow\\
      && &*(red)\leftarrow \\
      \leftarrow &\uparrow &*(gray)\bullet\\
      &*(gray)\leftarrow
    \end{ytableau}
  \]
\item For a move of type III, let \( (i,j) \) be the position of the
  empty free cell. Then the label is \( (i-1,j,\bullet) \).
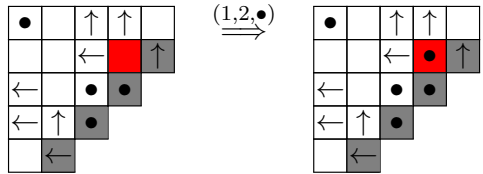
  \[
    \begin{ytableau}
      \bullet & & \uparrow & \uparrow & \\
      & & \leftarrow &*(red){} &*(gray)\uparrow\\
      \leftarrow& &   \bullet &*(gray)\bullet\\ \leftarrow &{\uparrow} &*(gray)\bullet\\
      &*(gray)\leftarrow
    \end{ytableau}\quad\overset{(1,2,\bullet)}{\Longrightarrow}\quad
    \begin{ytableau}
      \bullet & & \uparrow & \uparrow & \\
      & & \leftarrow & *(red)\bullet &*(gray)\uparrow\\
      \leftarrow &&\bullet &*(gray)\bullet \\
      \leftarrow &{\uparrow} &*(gray)\bullet\\
      &*(gray)\leftarrow
    \end{ytableau}\]
\end{itemize}

\section{Lattice of congruences}
\label{sec:lattice-congruences}

Recall the colored intervals of size \( n \) defined above
\Cref{thm:spine_colored_intervals}. We now endow them with another
partial order, different from the one introduced in
\Cref{sec:spine_rep}, as follows: for \(1\leq i \leq j\leq n-1\),
\begin{enumerate}
\item \([i,j]_{\mathtt{red}}\subseteq [i,k]_\sigma\) for all \(j\leq k\)
  and \(\sigma \in \{\mathtt{red}, \mathtt{blue}, \mathtt{green}\}\),
\item \([i,j]_{\mathtt{green}}\subseteq [k,j]_\sigma\) for all
  \(1\leq k\leq i\) and
  \(\sigma \in \{\mathtt{red}, \mathtt{blue}, \mathtt{green}\}\),
\item \([i,j]_\sigma\subseteq [k,\ell]_\tau\) if \(k<i\leq j < \ell\),
  \(\sigma\in \{\mathtt{red},\mathtt{green}\}\), and
  \(\tau \in \{\mathtt{red}, \mathtt{blue}, \mathtt{green}\}\).
\end{enumerate}
We refer to \Cref{fig:reflect_J_irr_poset} for an illustration in the
case \(n=6\). It is straightforward to check that this defines a
partial order on the set of colored intervals. We call it the
\defn{inclusion poset of colored intervals} of size \(n\). This name
emphasizes that this is a complicated version of the classical poset
of intervals ordered by inclusion. Note that the blue intervals are
always maximal elements in this poset. The main result of the section
is the following theorem.

\begin{theorem}\label{thm:con_jirr}
  The congruence lattice of \(\esTam_n\) is isomorphic to the lattice
  of order ideals on the opposite of the inclusion poset of colored
  intervals of size \( n \).
\end{theorem}

We refer to \cite{Reading2016} for the basic notation and fundamental
properties of lattice congruences. The \defn{congruence lattice}
\( \Con L \) is the poset of all congruences of \( L \), ordered by
inclusion of equivalence relations.
By~\cite[Proposition~9-5.13]{Reading2016}, if \( L \) is a lattice,
then \( \Con L \) is also a lattice. Moreover, the meet in
\( \Con L \) is given by the intersection of equivalence relations,
and the join in \( \Con L \) is given by the transitive closure of the
union of equivalence relations.

Since \( \Con L \) is a lattice, for any elements \( x \) and \( y \)
of \( L \), there exists a unique smallest congruence \( \ba \) on
\( L \) with \( x \equiv_\ba y \); we denote it by \( \con(x,y) \).
For a join-irreducible element \( \jj \) with a unique element \( \jj_* \)
such that \( \jj_* \lessdot \jj \), we write \( \con(\jj) = \con(\jj_*,\jj) \).
\begin{proposition}[{\cite[Proposition~9-5.14]{Reading2016}}]\label{pro:irr_conL}
  Let \( L \) be a finite lattice and \( \ba \in \Con L \). Then
  \( \ba \) is join-irreducible in \( \Con L \) if and only if
  \( \ba = \con(\jj) \) for some join-irreducible element \( \jj \) of
  \( L \).
\end{proposition}

By~\Cref{pro:irr_conL}, the map \( \jj \mapsto \con(\jj) \) is
surjective, but it is not injective in general. For each
meet-irreducible \( \mm \), we can define \( \mm^* \), and
consequently define a map
\( \mm \mapsto \con(\mm,\mm^*) =: \con(\mm) \) in the same manner. A
finite lattice \( L \) is said to be \defn{congruence uniform} if both
two maps \( \jj \mapsto \con(\jj) \) and \( \mm \mapsto \con(\mm) \)
are bijective.

We now show that our poset is congruence uniform.
\begin{theorem}\label{thm:HH_uniform}
  The extra slow Tamari lattice \(\esTam_n\) is congruence uniform. 
\end{theorem}
\begin{proof}
  By~\cite[Corollary~1]{CLM2004}, every semidistributive polygonal
  lattice which has a polygonal labeling is congruence uniform. In the
  terminology of that paper, such a lattice is called an
  \( \mathcal{HH} \)-lattice, and the polygons are called
  \(2\)-facets.

  By \Cref{thm:semidistributive} and \Cref{prop:polygonal}, the
  lattice \( \esTam_n \) is a semidistributive polygonal lattice. By
  \Cref{lem:HH-rank_func}, it admits a polygonal labeling. Therefore,
  \( \esTam_n \) is congruence uniform.
\end{proof}

In a lattice \( L \), we say that a cover relation \( a \lessdot b \)
\defn{forces} a cover relation \( c \lessdot d \) if
\( \con(a,b) \geq \con(c,d) \), or equivalently if
\( c \equiv_{\con(a,b)} d \). For join-irreducible elements \( \jj \)
and \( \jj' \) of \( L \), we simply say \( \jj \) \defn{forces}
\( \jj' \) if \( \con(\jj) \geq \con(\jj') \) in \( \Con L \). We
define the \defn{forcing relation} by \( \jj \geq \jj' \) if \( \jj \)
forces \( \jj' \). For a lattice \( L \), we define a new relation
\( F(L) \) on the join-irreducible elements of \( L \) as the
transitive closure of the forcing relation. For a general finite
lattice \( L \), the relation \( F(L) \) is a pre-ordered relation on
the join irreducible elements of \( L \). If two join-irreducible
elements \( \jj \) and \( \jj' \) satisfy
\( \con(\jj) = \con(\jj') \), we say that they are equivalent. Under
this equivalence relation, the forcing relation becomes a partial
order on the equivalence classes of join-irreducible elements. If
\( L \) is congruence uniform, then the map
\( \jj \mapsto \con(\jj) \) is a bijection and hence \( F(L) \) is a
partially ordered set on the join-irreducible elements of \( L \). The
lattice of congruences of a given lattice \(L\) is distributive, hence
by the Birkhoff representation theorem, we obtain the following
result.

\begin{proposition}[{\cite[Proposition~9-5.16]{Reading2016}}]\label{pro:cong_lattice_characterization}
  Let \( L \) be a finite lattice that is congruence uniform. Then
  \( \Con L \) is isomorphic to the poset on order ideals of
  \( F(L) \), ordered by inclusion.
\end{proposition}

In a polygon \( [x,y] \), the two edges incident to \( x \) are called
the \defn{bottom edges}, and the two edges incident to \( y \) are
called the \defn{top edges}. All remaining edges are called \defn{side
  edges}. Within the polygon, a cover relation \( a \lessdot b \)
forces \( c \lessdot d \) if one of the following conditions holds:
\begin{itemize}
\item \( a \lessdot b \) is a bottom edge of the polygon, and
  \( c \lessdot d \) is the opposite top edge or a side edge;
\item \( a \lessdot b \) is a top edge of the polygon, and
  \( c \lessdot d \) is the opposite bottom edge or a side edge.
\end{itemize}
For a polygonal lattice \( L \), the forcing relation to construct
\( F(L) \) has a nice property.

\begin{proposition}[{\cite[Proposition~9-6.5]{Reading2016}}]\label{pro:polygonal_forcing}
  Let \( L \) be a finite polygonal lattice. Then, if
  \( a \lessdot b \) forces \( c \lessdot d \), then there exists a
  sequence of cover relations
  \[
    (a \lessdot b) = (a_0 \lessdot b_0),~(a_1 \lessdot b_1),\dots,~(a_k \lessdot b_k) = (c \lessdot d),
  \]
  such that \( a_{i-1} \lessdot b_{i-1} \) forces
  \( a_{i} \lessdot b_{i} \) in a polygon for \( i = 1, \dots, k \).
\end{proposition}

Using the propositions above, we are now able to prove \Cref{thm:con_jirr}.

\begin{proof}[Proof of \Cref{thm:con_jirr}]
First, since our lattice is
congruence uniform by \Cref{thm:HH_uniform}, the map
\( \jj \mapsto \con(\jj) \) is a bijective, and therefore
\( F(\esTam_n) \) is a poset on the join-irreducible elements of
\( \esTam_n \).

Since we proved that the join-labelling given
in~\Cref{sec:poly-label} is a polygonal labelling, the
join-labelling coincides with the canonical labelling obtained by
assigning to each cover relation \( \jj_* \lessdot \jj \) the label
\( \jj \), such that in every polygon, the bottom edge and its
opposite top edge have the same label. Hence, it suffices to consider
the forcing orders that appear in the heptagons in our
poset. 

As illustrated in \Cref{fig:heptagon}, from a heptagon in
\( \esTam_n \), we obtain the following order in \( F(\esTam_n) \):
From the right bottom edge,
\[
  (i-1,j',\leftarrow) \geq (i-1,j,\sigma) \quad\text{for any \( \sigma \in \{\leftarrow, \bullet, \emptyset\} \) and \( j' < j \)},
\]
from the right top edge,
\[
  (i'-1, j, \emptyset) \geq (i-1,j,\sigma) \quad\text{for any \( \sigma \in \{\leftarrow, \bullet, \emptyset\} \) and \( i' > i \)}.
\]
This gives
\begin{enumerate}
\item \([i,j]_{\mathtt{red}}\subseteq [i,k]_\sigma\) for all
  \(k \geq j\) and
  \(\sigma \in \{\mathtt{red}, \mathtt{blue}, \mathtt{green}\}\), and
\item \([i,j]_{\mathtt{green}}\subseteq [k,j]_\sigma\) for all
  \(1\leq k\leq i\) and
  \(\sigma \in \{\mathtt{red}, \mathtt{blue}, \mathtt{green}\}\).
\end{enumerate}
By transitivity of the forcing order, we see that
\((k,\ell,d) \leq (i,j,c)\) for \( c \in \{\leftarrow, \emptyset\} \)
and \( d \in \{\leftarrow, \bullet, \emptyset\} \) implies that
\begin{enumerate}
\item[(3)] \([i,j]_\sigma\subseteq [k,\ell]_\tau\) for
  \(\sigma\in \{\mathtt{red},\mathtt{green}\}\), and
  \(\tau \in \{\mathtt{red}, \mathtt{blue}, \mathtt{green}\}\).
\end{enumerate}
It remains to prove the converse implication.

The cover relations in the inclusion poset of colored intervals are
given by
\begin{align*}
  [i,j]_{\mathtt{red}} &\subseteq [i,j+1]_\sigma \quad \text{for \( i \leq j < n-1 \)}, \\
  [i,j]_{\mathtt{green}} &\subseteq [i-1,j]_\sigma \quad \text{for \( 1 < i \leq j \)},
\end{align*}
for \(\sigma \in \{\mathtt{red}, \mathtt{blue}, \mathtt{green}\}\).
Therefore, it is enough to show that the corresponding
join-irreducible elements in \( F(\esTam_n) \) satisfy the opposite
cover relations.

A heptagon can be constructed whenever there is an empty cell with an
up-arrow below it, a left-arrow to its left, and a dot to its right.
This gives the following cover relations in \( F(\esTam_n) \). By
setting \( j' = j-1 \) and \( i' = i+1 \), we obtain the following.
For \( 1 \leq i, j \leq n-1 \), and \( j - i \geq 1 \),
\[
  (i,j,\leftarrow)\lessdot (i,j-1,\leftarrow), \quad (i,j,\bullet)\lessdot (i,j-1,\leftarrow), \quad\text{and}\quad (i,j,\emptyset) \lessdot (i,j-1,\leftarrow).
\]
For \( 1 \leq i, j \leq n-1 \), and \( j - i \geq 2 \),
\[
  (i,j,\leftarrow) \lessdot (i+1,j,\emptyset), \quad (i,j,\bullet) \lessdot (i+1,j,\emptyset), \quad\text{and}\quad (i,j,\emptyset) \lessdot (i+1,j,\emptyset).
\]
Finally, for \( 1 \leq i, j \leq n-1 \), and \( j - i = 1 \),
\[
  (i,i+1,\leftarrow) \lessdot (i+1,i+1,\bullet), \quad (i,i+1,\bullet) \lessdot (i+1,i+1,\bullet), \quad\text{and}\quad (i,i+1,\emptyset) \lessdot (i+1,i+1,\bullet).
\]
Thus, every cover relation in the inclusion poset is reversed in
\( F(\esTam_n) \), as desired. See~\Cref{fig:J_irr_poset} for an
illustration.
\begin{figure}
  \centering
  \begin{tikzpicture}[>=latex, scale=.65]
    \node (P11) at (-3,0) {
      $\begin{ytableau}
        \bullet&*(yellow)\uparrow&& \\
        &*(yellow)\leftarrow&&\uparrow \\
        \leftarrow&&\uparrow \\
        \leftarrow&\bullet
      \end{ytableau}$
    };
    \node (P12) at (0,0) {
      $\begin{ytableau}
        \bullet&*(yellow)\uparrow&& \\
        \leftarrow&*(yellow)\bullet&&\uparrow \\
        \leftarrow&&\uparrow \\
        \leftarrow&\bullet
      \end{ytableau}$
    };
    \node (P13) at (3,0) {
      $\begin{ytableau}
        \bullet&*(yellow)\uparrow&& \\
        \leftarrow&*(yellow)&&\uparrow \\
        \leftarrow&&\uparrow \\
        \leftarrow&\bullet
      \end{ytableau}$
    };
    \node (P21) at (-8,5) {
      $\begin{ytableau}
        \bullet&&& \\
        \leftarrow&*(yellow)\uparrow&&\uparrow \\
        &*(yellow)\leftarrow&\uparrow \\
        \leftarrow&\bullet
      \end{ytableau}$
    };

    \node (P22) at (-5,5) {
      $\begin{ytableau}
        \bullet&&& \\
        \leftarrow&*(yellow)\uparrow&&\uparrow \\
        \leftarrow&*(yellow)\bullet&\uparrow \\
        \leftarrow&\bullet
      \end{ytableau}$
    };
    \node (P23) at (-2,5) {
      $\begin{ytableau}
        \bullet&&& \\
        \leftarrow&*(yellow)\uparrow&&\uparrow \\
        \leftarrow&*(yellow)&\uparrow \\
        \leftarrow&\bullet
      \end{ytableau}$
    };
    \node (P24) at (2,5) {
      $\begin{ytableau}
        \bullet&&*(yellow)\uparrow& \\
        &&*(yellow)\leftarrow&\uparrow \\
        \leftarrow&&\bullet \\
        \leftarrow&\uparrow
      \end{ytableau}$
    };
    \node (P25) at (5,5) {
      $\begin{ytableau}
        \bullet&&*(yellow)\uparrow& \\
        \leftarrow&&*(yellow)\bullet&\uparrow \\
        \leftarrow&&\bullet \\
        \leftarrow&\uparrow
      \end{ytableau}$
    };

    \node (P26)  at (8,5) {
      $\begin{ytableau}
        \bullet&&*(yellow)\uparrow& \\
        \leftarrow&&*(yellow)&\uparrow \\
        \leftarrow&&\bullet \\
        \leftarrow&\uparrow
      \end{ytableau}$
    };
    \node (P31)  at (-8.5,10) {
      $\begin{ytableau}
        \bullet&&& \\
        \leftarrow&&&\uparrow \\
        \leftarrow&*(yellow)\uparrow&\uparrow \\
        &*(yellow)\leftarrow
      \end{ytableau}$
    };
    \node (P32)  at (-5.5,10) {
      $\begin{ytableau}
        \bullet&&& \\
        \leftarrow&&&\uparrow \\
        \leftarrow&*(yellow)\uparrow&\uparrow \\
        \leftarrow&*(yellow)\bullet
      \end{ytableau}$
    };
    \node (P33)  at (-1.5,10) {
      $\begin{ytableau}
        \bullet&&& \\
        \leftarrow&&*(yellow)\uparrow&\uparrow \\
        &&*(yellow)\leftarrow \\
        \leftarrow&\uparrow
      \end{ytableau}$
    };
    \node (P34)  at (1.5,10) {
      $\begin{ytableau}
        \bullet&&& \\
        \leftarrow&&*(yellow)\uparrow&\uparrow \\
        \leftarrow&&*(yellow)\bullet \\
        \leftarrow&\uparrow
      \end{ytableau}$
    };
    \node (P35)  at (5.5,10) {
      $\begin{ytableau}
        \bullet&&&*(yellow)\uparrow \\
        &&&*(yellow)\leftarrow \\
        \leftarrow&&\uparrow \\
        \leftarrow&\uparrow
      \end{ytableau}$
    };
    \node (P36)  at (8.5,10) {
      $\begin{ytableau}
        \bullet&&&*(yellow)\uparrow \\
        \leftarrow&&&*(yellow)\bullet \\
        \leftarrow&&\uparrow \\
        \leftarrow&\uparrow
      \end{ytableau}$
    };

    \draw (P11.north)  -- node {} (P23.south);
    \draw (P12.north)  -- node {} (P23.south);
    \draw (P13.north)  -- node {} (P23.south);
    \draw (P11.north)  -- node {} (P24.south);
    \draw (P12.north)  -- node {} (P24.south);
    \draw (P13.north)  -- node {} (P24.south);

    \draw (P21.north)  -- node {} (P32.south);
    \draw (P22.north)  -- node {} (P32.south);
    \draw (P23.north)  -- node {} (P32.south);
    \draw (P21.north)  -- node {} (P33.south);
    \draw (P22.north)  -- node {} (P33.south);
    \draw (P23.north)  -- node {} (P33.south);

    \draw (P24.north)  -- node {} (P34.south);
    \draw (P25.north)  -- node {} (P34.south);
    \draw (P26.north)  -- node {} (P34.south);
    \draw (P24.north)  -- node {} (P35.south);
    \draw (P25.north)  -- node {} (P35.south);
    \draw (P26.north)  -- node {} (P35.south);
  \end{tikzpicture}
  \caption{The poset \( F(\esTam_n) \) for \( n=4 \). Each
    join-irreducible \( (i,j,\sigma) \) is determined by the two
    yellow cells: the coordinates \( (i,j) \) of the up-arrow and the
    type \( \sigma \) of the cell immediately below the up-arrow.}
  \label{fig:J_irr_poset}
\end{figure}
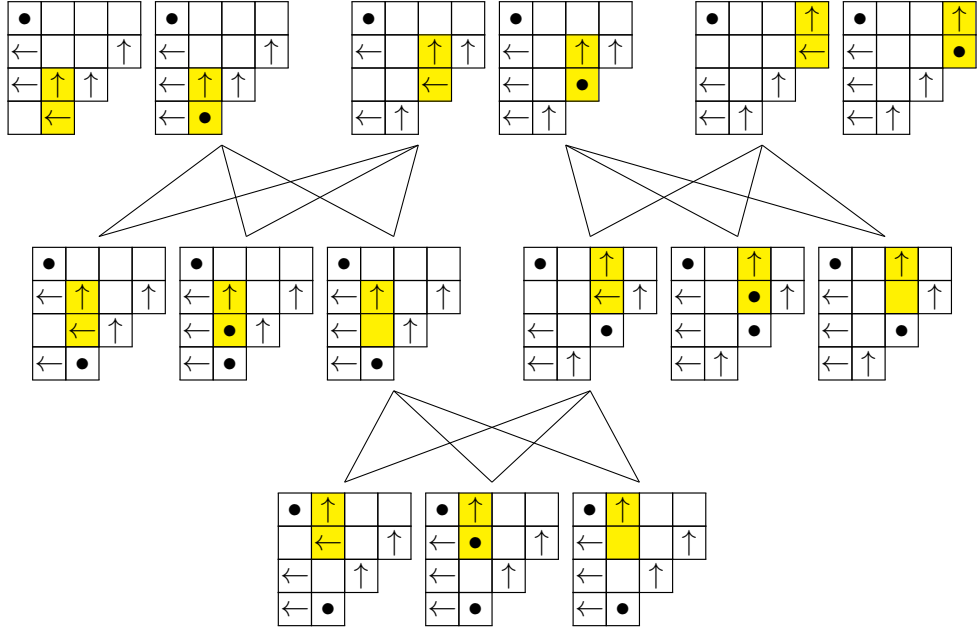
\end{proof}

\subsection{Enumeration}
\label{sec:enumeration-cong-lattice}

We characterized the congruence lattice of \( \esTam_n \), and we now
proceed to determine the number of elements in this congruence
lattice. The order ideals of \( F(\esTam_n) \) are in bijection with
its antichains, which is equivalent to considering its dual order
ideals. To count elements, we represent the elements of the poset by
dots, as illustrated in \Cref{fig:reflect_J_irr_poset} (Left). Red,
blue, and green dots correspond respectively to the join-irreducible
element of the forms \( (*,*,\leftarrow) \), \( (*,*,\bullet) \), and
\( (*,*,\emptyset) \), except that the green dots in the bottom row
correspond to elements of the form \( (*,*,\bullet) \).

As illustrated in \Cref{fig:reflect_J_irr_poset} (Right), after
appropriately grouping the colors, we consider a cell to lie below the
path whenever at least one of its colors belongs to the given order
ideal. Consequently, each order ideal corresponds to a Dyck path with
colorings assigned to the cells below the path. The coloring of each
cell is defined as follows: in the bottom row, each cell is colored by
a subset of \( \{ \mathtt{red},\mathtt{blue} \} \), whereas in the higher
rows, each cell is colored by a subset of
\( \{ \mathtt{red},\mathtt{blue},\mathtt{green} \} \).

\begin{figure}
  \centering
\begin{minipage}[t]{0.48\textwidth}
    \centering
\begin{tikzpicture}[
  thicknode/.style={
    circle,          
    minimum size=7pt,
    inner sep=0pt    
  },
  rednode/.style={thicknode, fill=red},
  bluenode/.style={thicknode, fill=blue},
  greennode/.style={thicknode, fill=green!40!black},
  scale = .9
]
    \node[rednode] (P11) at (-4-.2,0) {};
    \node[greennode] (P12) at (-4+.2,0) {};
    \node[rednode] (P13) at (-2-.2,0) {};
    \node[greennode] (P14) at (-2+.2,0) {};
    \node[rednode] (P15) at (0-.2,0) {};
    \node[greennode] (P16) at (0+.2,0) {};
    \node[rednode] (P17) at (2-.2,0) {};
    \node[greennode] (P18) at (2+.2,0) {};
    \node[rednode] (P19) at (4-.2,0) {};
    \node[greennode] (P110) at (4+.2,0) {};

    \node[rednode] (P21) at (-3-.4,1) {};
    \node[bluenode] (P22) at (-3,1) {};
    \node[greennode] (P23) at (-3+.4,1) {};
    \node[rednode] (P24) at (-1-.4,1) {};
    \node[bluenode] (P25) at (-1,1) {};
    \node[greennode] (P26) at (-1+.4,1) {};
    \node[rednode] (P27) at (1-.4,1) {};
    \node[bluenode] (P28) at (1,1) {};
    \node[greennode] (P29) at (1+.4,1) {};
    \node[rednode] (P210) at (3-.4,1) {};
    \node[bluenode] (P211) at (3,1) {};
    \node[greennode] (P212) at (3+.4,1) {};

    \node[rednode] (P31) at (-2-.4,2) {};
    \node[bluenode] (P32) at (-2,2) {};
    \node[greennode] (P33) at (-2+.4,2) {};
    \node[rednode] (P34) at (0-.4,2) {};
    \node[bluenode] (P35) at (0,2) {};
    \node[greennode] (P36) at (0+.4,2) {};
    \node[rednode] (P37) at (2-.4,2) {};
    \node[bluenode] (P38) at (2,2) {};
    \node[greennode] (P39) at (2+.4,2) {};

    \node[rednode] (P41) at (-1-.4,3) {};
    \node[bluenode] (P42) at (-1,3) {};
    \node[greennode] (P43) at (-1+.4,3) {};
    \node[rednode] (P44) at (1-.4,3) {};
    \node[bluenode] (P45) at (1,3) {};
    \node[greennode] (P46) at (1+.4,3) {};

    \node[rednode] (P51) at (0-.4,4) {};
    \node[bluenode] (P52) at (0,4) {};
    \node[greennode] (P53) at (0+.4,4) {};

    \draw (P12)  -- node {} (P21); \draw (P12)  -- node {} (P22); \draw (P12)  -- node {} (P23);
    \draw (P13)  -- node {} (P21); \draw (P13)  -- node {} (P22); \draw (P13)  -- node {} (P23);
    \draw (P14)  -- node {} (P24); \draw (P14)  -- node {} (P25); \draw (P14)  -- node {} (P26);
    \draw (P15)  -- node {} (P24); \draw (P15)  -- node {} (P25); \draw (P15)  -- node {} (P26);
    \draw (P16)  -- node {} (P27); \draw (P16)  -- node {} (P28); \draw (P16)  -- node {} (P29);
    \draw (P17)  -- node {} (P27); \draw (P17)  -- node {} (P28); \draw (P17)  -- node {} (P29);
    \draw (P18)  -- node {} (P210); \draw (P18)  -- node {} (P211); \draw (P18)  -- node {} (P212);
    \draw (P19)  -- node {} (P210); \draw (P19)  -- node {} (P211); \draw (P19)  -- node {} (P212);

    \draw (P23)  -- node {} (P31); \draw (P23)  -- node {} (P32); \draw (P23)  -- node {} (P33);
    \draw (P24)  -- node {} (P31); \draw (P24)  -- node {} (P32); \draw (P24)  -- node {} (P33);
    \draw (P26)  -- node {} (P34); \draw (P26)  -- node {} (P35); \draw (P26)  -- node {} (P36);
    \draw (P27)  -- node {} (P34); \draw (P27)  -- node {} (P35); \draw (P27)  -- node {} (P36);
    \draw (P29)  -- node {} (P37); \draw (P29)  -- node {} (P38); \draw (P29)  -- node {} (P39);
    \draw (P210)  -- node {} (P37); \draw (P210)  -- node {} (P38); \draw (P210)  -- node {} (P39);

    \draw (P33)  -- node {} (P41); \draw (P33)  -- node {} (P42); \draw (P33)  -- node {} (P43);
    \draw (P34)  -- node {} (P41); \draw (P34)  -- node {} (P42); \draw (P34)  -- node {} (P43);
    \draw (P36)  -- node {} (P44); \draw (P36)  -- node {} (P45); \draw (P36)  -- node {} (P46);
    \draw (P37)  -- node {} (P44); \draw (P37)  -- node {} (P45); \draw (P37)  -- node {} (P46);

    \draw (P43)  -- node {} (P51); \draw (P43)  -- node {} (P52); \draw (P43)  -- node {} (P53);
    \draw (P44)  -- node {} (P51); \draw (P44)  -- node {} (P52); \draw (P44)  -- node {} (P53);
  \end{tikzpicture}

  \end{minipage}\hfill
  \begin{minipage}[t]{0.48\textwidth}
    \centering
\begin{tikzpicture}[
  thicknode/.style={
    circle,          
    minimum size=6pt,
    inner sep=0pt    
  },
  rednode/.style={thicknode, fill=red},
  bluenode/.style={thicknode, fill=blue},
  greennode/.style={thicknode, fill=green!40!black},
  scale = .7
]
    \node[rednode] (P11) at (-4-.2,0) {};
    \node[greennode] (P12) at (-4+.2,0) {};
    \node[rednode] (P13) at (-2-.2,0) {};
    \node[greennode] (P14) at (-2+.2,0) {};
    \node[rednode] (P15) at (0-.2,0) {};
    \node[greennode] (P16) at (0+.2,0) {};
    \node[rednode] (P17) at (2-.2,0) {};
    \node[greennode] (P18) at (2+.2,0) {};
    \node[rednode] (P19) at (4-.2,0) {};
    \node[greennode] (P110) at (4+.2,0) {};

    \node[rednode] (P21) at (-3-.4,1) {};
    \node[bluenode] (P22) at (-3,1) {};
    \node[greennode] (P23) at (-3+.4,1) {};
    \node[rednode] (P24) at (-1-.4,1) {};
    \node[bluenode] (P25) at (-1,1) {};
    \node[greennode] (P26) at (-1+.4,1) {};
    \node[rednode] (P27) at (1-.4,1) {};
    \node[bluenode] (P28) at (1,1) {};
    \node[greennode] (P29) at (1+.4,1) {};
    \node[rednode] (P210) at (3-.4,1) {};
    \node[bluenode] (P211) at (3,1) {};
    \node[greennode] (P212) at (3+.4,1) {};

    \node[rednode] (P31) at (-2-.4,2) {};
    \node[bluenode] (P32) at (-2,2) {};
    \node[greennode] (P33) at (-2+.4,2) {};
    \node[rednode] (P34) at (0-.4,2) {};
    \node[bluenode] (P35) at (0,2) {};
    \node[greennode] (P36) at (0+.4,2) {};
    \node[rednode] (P37) at (2-.4,2) {};
    \node[bluenode] (P38) at (2,2) {};
    \node[greennode] (P39) at (2+.4,2) {};

    \node[rednode] (P41) at (-1-.4,3) {};
    \node[bluenode] (P42) at (-1,3) {};
    \node[greennode] (P43) at (-1+.4,3) {};
    \node[rednode] (P44) at (1-.4,3) {};
    \node[bluenode] (P45) at (1,3) {};
    \node[greennode] (P46) at (1+.4,3) {};

    \node[rednode] (P51) at (0-.4,4) {};
    \node[bluenode] (P52) at (0,4) {};
    \node[greennode] (P53) at (0+.4,4) {};

    \draw (P12)  -- node {} (P21); \draw (P12)  -- node {} (P22); \draw (P12)  -- node {} (P23);
    \draw (P13)  -- node {} (P21); \draw (P13)  -- node {} (P22); \draw (P13)  -- node {} (P23);
    \draw (P14)  -- node {} (P24); \draw (P14)  -- node {} (P25); \draw (P14)  -- node {} (P26);
    \draw (P15)  -- node {} (P24); \draw (P15)  -- node {} (P25); \draw (P15)  -- node {} (P26);
    \draw (P16)  -- node {} (P27); \draw (P16)  -- node {} (P28); \draw (P16)  -- node {} (P29);
    \draw (P17)  -- node {} (P27); \draw (P17)  -- node {} (P28); \draw (P17)  -- node {} (P29);
    \draw (P18)  -- node {} (P210); \draw (P18)  -- node {} (P211); \draw (P18)  -- node {} (P212);
    \draw (P19)  -- node {} (P210); \draw (P19)  -- node {} (P211); \draw (P19)  -- node {} (P212);

    \draw (P23)  -- node {} (P31); \draw (P23)  -- node {} (P32); \draw (P23)  -- node {} (P33);
    \draw (P24)  -- node {} (P31); \draw (P24)  -- node {} (P32); \draw (P24)  -- node {} (P33);
    \draw (P26)  -- node {} (P34); \draw (P26)  -- node {} (P35); \draw (P26)  -- node {} (P36);
    \draw (P27)  -- node {} (P34); \draw (P27)  -- node {} (P35); \draw (P27)  -- node {} (P36);
    \draw (P29)  -- node {} (P37); \draw (P29)  -- node {} (P38); \draw (P29)  -- node {} (P39);
    \draw (P210)  -- node {} (P37); \draw (P210)  -- node {} (P38); \draw (P210)  -- node {} (P39);

    \draw (P33)  -- node {} (P41); \draw (P33)  -- node {} (P42); \draw (P33)  -- node {} (P43);
    \draw (P34)  -- node {} (P41); \draw (P34)  -- node {} (P42); \draw (P34)  -- node {} (P43);
    \draw (P36)  -- node {} (P44); \draw (P36)  -- node {} (P45); \draw (P36)  -- node {} (P46);
    \draw (P37)  -- node {} (P44); \draw (P37)  -- node {} (P45); \draw (P37)  -- node {} (P46);

    \draw (P43)  -- node {} (P51); \draw (P43)  -- node {} (P52); \draw (P43)  -- node {} (P53);
    \draw (P44)  -- node {} (P51); \draw (P44)  -- node {} (P52); \draw (P44)  -- node {} (P53);

    \fill[black,opacity=0.3] (P12) circle [radius=3mm];
    \fill[black,opacity=0.3] (P13) circle [radius=3mm];
    \fill[black,opacity=0.3] (P14) circle [radius=3mm];
    \fill[black,opacity=0.3] (P15) circle [radius=3mm];
    \fill[black,opacity=0.3] (P16) circle [radius=3mm];
    \fill[black,opacity=0.3] (P17) circle [radius=3mm];
    \fill[black,opacity=0.3] (P18) circle [radius=3mm];
    \fill[black,opacity=0.3] (P19) circle [radius=3mm];
    \fill[black,opacity=0.3] (P110) circle [radius=3mm];
    \fill[black,opacity=0.3] (P25) circle [radius=3mm];
    \fill[black,opacity=0.3] (P26) circle [radius=3mm];
    \fill[black,opacity=0.3] (P27) circle [radius=3mm];
    \fill[black,opacity=0.3] (P29) circle [radius=3mm];
    \fill[black,opacity=0.3] (P210) circle [radius=3mm];
    \fill[black,opacity=0.3] (P211) circle [radius=3mm];
    \fill[black,opacity=0.3] (P35) circle [radius=3mm];
    \fill[black,opacity=0.3] (P37) circle [radius=3mm];
    \fill[black,opacity=0.3] (P39) circle [radius=3mm];

    \draw[very thick] (-6,-1) -- (-4,1) -- (-3,0) -- (0,3) -- (1,2) -- (2,3) -- (6,-1);
    \draw[gray] (-6,-1) -- (0,5);
    \draw[gray] (-4,-1) -- (1,4);
    \draw[gray] (-2,-1) -- (2,3);
    \draw[gray] (0,-1) -- (3,2);
    \draw[gray] (2,-1) -- (4,1);
    \draw[gray] (4,-1) -- (5,0);
    \draw[gray] (6,-1) -- (0,5);
    \draw[gray] (4,-1) -- (-1,4);
    \draw[gray] (2,-1) -- (-2,3);
    \draw[gray] (0,-1) -- (-3,2);
    \draw[gray] (-2,-1) -- (-4,1);
    \draw[gray] (-4,-1) -- (-5,0);
  \end{tikzpicture}
  \end{minipage}

  \caption{(Left) A simplified representation of the dual of
    \( F(\esTam_n) \) for \( n=6 \) (Right) An example of its order
    ideal and the corresponding Dyck path.}
  \label{fig:reflect_J_irr_poset}
\end{figure}
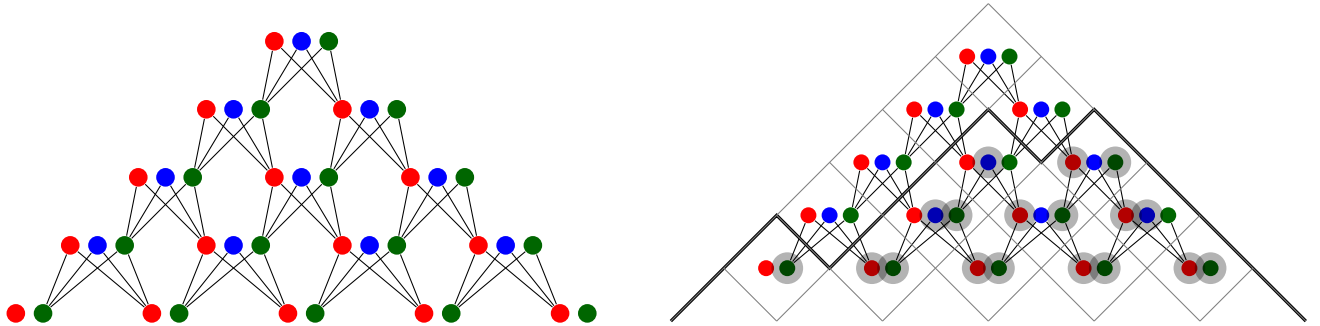

We classify the cells according to the number of edges they share with
the path. Each cell can share at most two of its upper edges with the
path. A cell that shares both of its edges with the boundary is called
a \defn{corner cell}. A cell that shares exactly one edge with the
boundary is called a \defn{boundary cell}; if it shares the right
edge, we refer to it as a left boundary cell, and if it shares the
left edge, as a right boundary cell. A cell that shares no edge with
the boundary is called an \defn{inner cell}.

The \defn{height} of a cell is the height of its midpoint, where the
up-steps and down-steps of a Dyck path are regarded as \( (1,1) \) and
\( (1,-1) \), respectively.

\begin{theorem}\label{the:congruence_enumeration}
  We define the weight of a Dyck path as the product of the weights of
  all the cells below the path, where the weight \( \wt(c) \) of each
  cell \( c \) is given by
  \[
    \wt(c) = 
    \begin{cases}
     3 & \text{if \( c \) is a corner cell of height \( 1 \)},\\
     2 & \text{if \( c \) is a boundary cell of height \( 1 \)},\\
     1 & \text{if \( c \) is an inner cell of height \( 1 \)},\\
     7 & \text{if \( c \) is a corner cell of height \( > 1 \)}, \\
     4 & \text{if \( c \) is a boundary cell of height \( > 1 \)},\\
     2 & \text{if \( c \) is an inner cell of height \( > 1 \)}.
    \end{cases}
  \]
  The number of elements in the congruence lattice of \( \esTam_n \)
  equals the total weight sum of all Dyck paths from \( (0,0) \) to
  \( (2n,0) \).
\end{theorem}
\begin{proof}
  For a cell at height \( 1 \), its color may be any subset of
  \( \{ \mathtt{red},\mathtt{blue} \} \). If there exists a cell
  immediately above it on the right, the cell must include blue in its
  coloring; similarly, if there exists a cell immediately above it on
  the left, it must include red. Likewise, for a cell at height
  greater than \( 1 \), its color may be any subset of
  \( \{ \mathtt{red},\mathtt{blue},\mathtt{green} \} \). If there exists a
  cell immediately above it on the right (resp. left), the cell must
  include green (resp. red) in its coloring. Under these conditions,
  we can classify all possible colorings as follows.
  For cells of height \( 1 \):
  \begin{itemize}
  \item If it is a corner cell, there are three possible colorings:
    nonempty subsets of \( \{ \mathtt{red},\mathtt{blue} \} \).
  \item If it is a boundary cell, there are two possible colorings:
    subsets of \( \{ \mathtt{red},\mathtt{blue} \} \) containing blue if
    it is a left boundary cell, or containing red if it is a right
    boundary cell.
  \item If it is an inner cell, there is one possible coloring:
    \( \{ \mathtt{red},\mathtt{blue} \} \).
  \end{itemize}
  For cells of height greater than \( 1 \):
  \begin{itemize}
  \item If it is a corner cell, there are seven possible colorings:
    nonempty subsets of
    \( \{ \mathtt{red},\mathtt{blue},\mathtt{green} \} \).
  \item If it is a boundary cell, there are four possible colorings:
    subsets of \( \{ \mathtt{red},\mathtt{blue},\mathtt{green} \} \)
    containing green if it is a left boundary cell, or containing red
    if it is a right boundary cell.
  \item If it is an inner cell, there are two possible colorings:
    subsets of \( \{ \mathtt{red},\mathtt{blue},\mathtt{green} \} \)
    containing both red and green.
  \end{itemize}
  This completes the proof.
\end{proof}

Using the weighted Dyck path model
in~\Cref{the:congruence_enumeration}, one can compute the number of
elements in the congruence lattice of \( \esTam_n \). For example,
consider the five Dyck paths from \( (0,0) \) to \( (6,0) \)
corresponding to the case \( n=3 \), where the weight of each cell is
represented as shown in
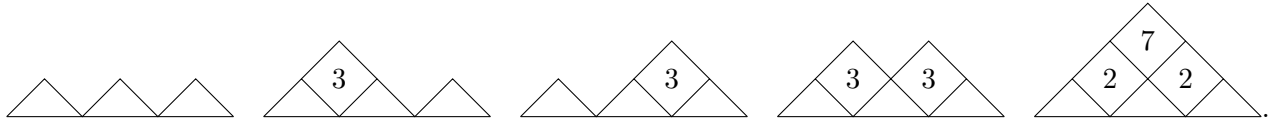
\[
  \begin{tikzpicture}[scale=.5]
    \draw (0,0) -- (6,0);
    \draw (0,0) -- ++(1,1) -- ++(1,-1) -- ++(1,1) -- ++(1,-1) -- ++(1,1) -- ++(1,-1);
  \end{tikzpicture}
  \quad
  \begin{tikzpicture}[scale=.5]
    \draw (0,0) -- (6,0);
    \draw (0,0) -- ++(1,1) -- ++(1,-1) -- ++(1,1) -- ++(1,-1) -- ++(1,1) -- ++(1,-1);
    \draw (1,1) -- (2,2) -- (3,1);
    \node at (2,1) {\( 3 \)};
  \end{tikzpicture}
  \quad
  \begin{tikzpicture}[scale=.5]
    \draw (0,0) -- (6,0);
    \draw (0,0) -- ++(1,1) -- ++(1,-1) -- ++(1,1) -- ++(1,-1) -- ++(1,1) -- ++(1,-1);
    \draw (3,1) -- (4,2) -- (5,1);
    \node at (4,1) {\( 3 \)};
  \end{tikzpicture}
  \quad
  \begin{tikzpicture}[scale=.5]
    \draw (0,0) -- (6,0);
    \draw (0,0) -- ++(1,1) -- ++(1,-1) -- ++(1,1) -- ++(1,-1) -- ++(1,1) -- ++(1,-1);
    \draw (1,1) -- (2,2) -- (3,1) -- (4,2) -- (5,1);
    \node at (2,1) {\( 3 \)};
    \node at (4,1) {\( 3 \)};
  \end{tikzpicture}
  \quad
  \begin{tikzpicture}[scale=.5]
    \draw (0,0) -- (6,0);
    \draw (0,0) -- ++(1,1) -- ++(1,-1) -- ++(1,1) -- ++(1,-1) -- ++(1,1) -- ++(1,-1);
    \draw (1,1) -- (2,2) -- (3,1) -- (4,2) -- (5,1);
    \draw (2,2) -- (3,3) -- (4,2);
    \node at (2,1) {\( 2 \)};
    \node at (4,1) {\( 2 \)};
    \node at (3,2) {\( 7 \)};
  \end{tikzpicture}.
\]
Thus, the number of elements in the congruence lattice of
\( \esTam_3 \) is \( 1 + 3 + 3 + 9 + 28 = 44 \). 

We now derive a continued fraction for the generating function
associated with the weighted Dyck path model of
\Cref{the:congruence_enumeration}. 

\begin{theorem}\label{prop:esTam_cong_cf}
  Let \( g(n) \) be the number of elements in the congruence lattice
  of \( \esTam_n \), and let
  \[
    G(x) := \sum_{n\ge 0} g(n) x^n.
  \]
  Then
  \[
    G(x)=
    \cfrac{1}{1-\cfrac{x}{1+x-\cfrac{4x}{1+x-\cfrac{8x}{1+2x-\cfrac{16x}{1+4x-\cfrac{32x}{\ddots}}}}}}.
  \]
\end{theorem}
\begin{proof}
  We reinterpret the weight in \Cref{the:congruence_enumeration} in
  terms of down steps and peaks.

  Consider a Dyck path together with the weights assigned to its
  cells. We redistribute the factor \( 2 \) from each left boundary
  cell to the corresponding right boundary cell; that is, we divide
  the weight of each lest boundary cell by \( 2 \) and multiply the
  weight of the corresponding right boundary cell by \( 2 \). For
  example,
  
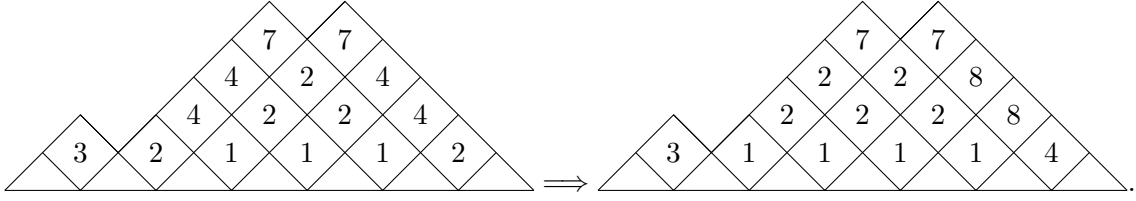
\[
    \begin{tikzpicture}[scale=.5]
      \draw (0,0) -- (14,0);
      \draw (0,0) -- ++(1,1) -- ++(1,1) -- ++(1,-1) -- ++(1,1) -- ++(1,1) -- ++(1,1) -- ++(1,1) -- ++(1,-1) -- ++(1,1) -- ++(1,-1) -- ++(1,-1) -- ++(1,-1) -- ++(1,-1) -- ++(1,-1);
      \draw (1,1) -- (2,0) -- (4,2);
      \draw (2,2) -- (4,0) -- (9,5);
      \draw (4,2) -- (6,0) -- (10,4);
      \draw (5,3) -- (8,0) -- (11,3);
      \draw (6,4) -- (10,0) -- (12,2);
      \draw (8,4) -- (12,0) -- (13,1);
      \node at (2,1) {\( 3 \)};
      \node at (4,1) {\( 2 \)};
      \node at (6,1) {\( 1 \)};
      \node at (8,1) {\( 1 \)};
      \node at (10,1) {\( 1 \)};
      \node at (12,1) {\( 2 \)};
      \node at (5,2) {\( 4 \)};
      \node at (7,2) {\( 2 \)};
      \node at (9,2) {\( 2 \)};
      \node at (11,2) {\( 4 \)};
      \node at (6,3) {\( 4 \)};
      \node at (8,3) {\( 2 \)};
      \node at (10,3) {\( 4 \)};
      \node at (7,4) {\( 7 \)};
      \node at (9,4) {\( 7 \)};
    \end{tikzpicture}
    \Longrightarrow
    \begin{tikzpicture}[scale=.5]
      \draw (0,0) -- (14,0);
      \draw (0,0) -- ++(1,1) -- ++(1,1) -- ++(1,-1) -- ++(1,1) -- ++(1,1) -- ++(1,1) -- ++(1,1) -- ++(1,-1) -- ++(1,1) -- ++(1,-1) -- ++(1,-1) -- ++(1,-1) -- ++(1,-1) -- ++(1,-1);
      \draw (1,1) -- (2,0) -- (4,2);
      \draw (2,2) -- (4,0) -- (9,5);
      \draw (4,2) -- (6,0) -- (10,4);
      \draw (5,3) -- (8,0) -- (11,3);
      \draw (6,4) -- (10,0) -- (12,2);
      \draw (8,4) -- (12,0) -- (13,1);
      \node at (2,1) {\( 3 \)};
      \node at (4,1) {\( 1 \)};
      \node at (6,1) {\( 1 \)};
      \node at (8,1) {\( 1 \)};
      \node at (10,1) {\( 1 \)};
      \node at (12,1) {\( 4 \)};
      \node at (5,2) {\( 2 \)};
      \node at (7,2) {\( 2 \)};
      \node at (9,2) {\( 2 \)};
      \node at (11,2) {\( 8 \)};
      \node at (6,3) {\( 2 \)};
      \node at (8,3) {\( 2 \)};
      \node at (10,3) {\( 8 \)};
      \node at (7,4) {\( 7 \)};
      \node at (9,4) {\( 7 \)};
    \end{tikzpicture}.
  \]
  We then multiply the weights of the cells along each diagonal. In
  this way, the cell weights can be reinterpreted as weights on down
  steps. For the Dyck path above, the corresponding weights of down
  steps are shown below:
  
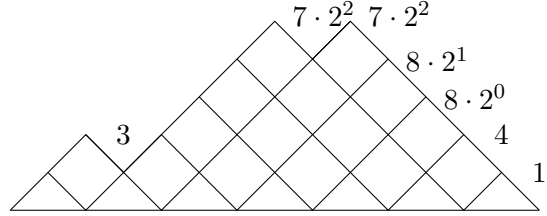
\[
    \begin{tikzpicture}[scale=.5]
      \draw (0,0) -- (14,0);
      \draw (0,0) -- ++(1,1) -- ++(1,1) -- ++(1,-1) -- ++(1,1) -- ++(1,1) -- ++(1,1) -- ++(1,1) -- ++(1,-1) -- ++(1,1) -- ++(1,-1) -- ++(1,-1) -- ++(1,-1) -- ++(1,-1) -- ++(1,-1);
      \draw (1,1) -- (2,0) -- (4,2);
      \draw (2,2) -- (4,0) -- (9,5);
      \draw (4,2) -- (6,0) -- (10,4);
      \draw (5,3) -- (8,0) -- (11,3);
      \draw (6,4) -- (10,0) -- (12,2);
      \draw (8,4) -- (12,0) -- (13,1);
      \node at (3,2) {\( 3 \)};
      \node at (13,2) {\( 4 \)};
      \node at (14,1) {\( 1 \)};
      \node at (12.3,3) {\( 8\cdot 2^0 \)};
      \node at (11.3,4) {\( 8\cdot 2^1 \)};
      \node at (8.3,5.2) {\( 7\cdot 2^2 \)};
      \node at (10.3,5.2) {\( 7\cdot 2^2 \)};
    \end{tikzpicture}
  \]

  Accordingly, for a Dyck path \(p\), define its weight \(\wt'(p)\) as
  the product of the weights of its steps, where
  \begin{itemize}
  \item every up step has weight \(1\);
  \item a down step \((a,1)\to(a+1,0)\) has weight \(1\);
  \item a down step \((a,2)\to(a+1,1)\) has weight \(3\) if \((a,2)\) is a peak, and weight \(4\) otherwise;
  \item for \(h\geq 2\), a down step \((a,h+1)\to(a+1,h)\) has weight \(7\cdot 2^{h-2}\) if \((a,h+1)\) is a peak, and weight \(8\cdot 2^{h-2}\) otherwise.
  \end{itemize}
  It is immediate that \( \wt(p) = \wt'(p) \), where \( \wt \) is the
  weight function obtained from the original cell model in
  \Cref{the:congruence_enumeration}.

  For \(h\ge 0\), let \(D_h(x)\) be the generating function for Dyck
  paths that start and end at height \(h\) and never go below height
  \(h\), with the above weight assignment shifted so that height \(h\)
  plays the role of the ground level.

  We first consider \(D_0(x)\). Any such path is either empty, or can
  be uniquely decomposed as \(Up_1Dp_2\), where \(p_1\) is a Dyck path
  based at height \(1\), and \(p_2\) is a Dyck path based at height
  \(0\). Since the down step returning to height \(0\) has weight
  \(1\), we obtain
  \[
    D_0(x)=1+xD_1(x)D_0(x),
  \]
  and hence
  \[
    D_0(x)=\frac{1}{1-xD_1(x)}.
  \]

  Next, consider \(D_1(x)\). A nonempty path based at height \(1\) is
  again of the form \(Up_1Dp_2\). If \(p_1\) is empty, then we obtain a
  peak at height \(2\), which contributes weight \(3\). If \(p_1\) is
  nonempty, then the returning down step is not part of a peak and has
  weight \(4\). Therefore
  \[
    D_1(x)=1+3xD_1(x)+4x\bigl(D_2(x)-1\bigr)D_1(x),
  \]
  so
  \[
    D_1(x)=\frac{1}{1+x-4xD_2(x)}.
  \]

  For heights \(h\geq 2\), the same decomposition applies, and thus
  the generating function \(D_2(x)\) satisfies
  \[
    D_2(x)=1+7xD_2(x)+8x\bigl(D_2(2x)-1\bigr)D_2(x),
  \]
  and hence
  \[
    D_2(x)=\frac{1}{1+x-8xD_2(2x)}.
  \]
  Iterating this relation yields
  \[
    D_2(x)=
    \cfrac{1}{1+x-\cfrac{8x}{1+2x-\cfrac{16x}{1+4x-\cfrac{32x}{\ddots}}}}.
  \]
  Substituting this into the formulas for \(D_1(x)\) and \(D_0(x)\)
  gives the desired continued fraction for \(D_0(x)\). Since
  \(D_0(x)\) is precisely the generating function for the total weight
  sum of Dyck paths in \Cref{the:congruence_enumeration}, the result
  follows.
\end{proof}

Expanding the continued fraction gives
\[
  G(x)=1+x+4x^2+44x^3+932x^4+35788x^5+2532868x^6 + 341012204x^7 + \cdots,
\]
so the numbers of elements in the congruence lattices of
\( \esTam_n \) for \( n \geq 1 \) begin with
\[
  1,~4,~44,~932,~35788,~2532868,~341012204,\dots .
\]
For the slow Tamari lattice \( \sTam_n \), there is also a nice
generating function for the number of elements in the congruence
lattice, expressed in terms of the generating function of the Narayana
polynomials; see~\Cref{the:1}.

\section{The slow Tamari lattice}\label{sec:slow_Tamari}

Recall that a fb-tableau \( T \) is said to be small if it contains
exactly one dot in the root. We proved in \Cref{prop:sublattices} that
the set \( \sTam_n \) of small fb-tableaux of size \( n \) with the
relation \( \lhd \) forms a sublattice of \( \esTam_n \), and in
\Cref{def:sTam}, we call the lattice \( (\sTam_n,\lhd) \) the slow
Tamari lattice.

In this section, we prove several properties of the slow Tamari
lattice. Since most of these properties can be proved by arguments
similar to those used for \( \esTam_n \), we focus on stating the
results and omit the detailed proofs.

\begin{lemma}[Description of cover relations]\label{lem:sTam_cover}
  For two small tableaux \( T \) and \( T' \), \( T \lhd T' \) is a
  cover relation if \( T' \) is obtained from \( T \) by one of the
  following two moves:
  \begin{itemize}
  \item An up-arrow in \( T \) moves upward within its column to the
    closest cell above it that is not pointed to by a left-arrow.
  \item A left-arrow in \( T \) moves rightward within its row to the
    closest cell to its right that is not pointed to by an up-arrow, if this cell is not a border cell. Otherwise, the left-arrow replaces its up-arrow, and the up-arrow of the border cell moves upward within its column to the
    closest cell above it that is not pointed to by a left-arrow.
  \end{itemize}

\end{lemma}
\begin{proof}
The proof is an adaptation to \Cref{lem:cover1} and \Cref{pro:cover}. For a similar result, stated in algebraic terms, one can also look at \cite[Corollary 9.9]{CMRS2021}.
\end{proof}

\begin{lemma}\label{lem:sTam_selfdual}
  The lattice \( \sTam_n \) is self-dual.
\end{lemma}
\begin{proof}
The self-duality is given by the usual conjugation of tableaux. 
\end{proof}

\begin{lemma}[Description of join-irreducible elements]\label{lem:sTam_join_irr}
  Join-irreducible elements of \( \sTam_n \) can be expressed as
  \( (i,j,\sigma) \) defined in \Cref{def:join-irr}, with the
  restriction that \( \sigma \neq \bullet \). Concretely, for
  \( 1 \leq i \leq j < n \), we take
  \( \sigma \in \{\emptyset, \leftarrow \} \), but
  \( \sigma \not= \emptyset \) when \( i = j \). The corresponding
  tableau \( (i,j,\sigma) \) is obtained by placing all
  \( \uparrow \)s on the border except one at position \( (i,j) \),
  placing \( \sigma \) in the cell immediately below \( (i,j) \),
  placing a \( \leftarrow \) in the border cell of the column
  containing \( (i,j) \), and placing all remaining \( \leftarrow \)s
  in the leftmost column.
\end{lemma}
\begin{proof}
This is an easy consequence of \Cref{lem:sTam_cover} and \Cref{lem:sTam_selfdual}.
\end{proof}
\begin{lemma}\label{lem:stamjiir}
  The lattice \( \sTam_n \) has \( (n-1)^2 \) meet-irreducible elements and
  \( (n-1)^2 \) join-irreducible elements.
\end{lemma}

\begin{lemma}\label{lem:sTam_trim}
  The lattice \( \sTam_n \) is semidistributive and extremal. Hence,
  it is trim.
\end{lemma}
\begin{proof}
  Since \(\sTam_n\) is a sublattice of \(\esTam_n\), it follows from
  \Cref{thm:semidistributive} that \( \sTam_n \) is semidistributive.
  Trimness is not preserved by taking sublattices in general. However,
  it is easy to prove that \(\sTam_n\) is extremal, as in the proof of
  \Cref{prop:extremal}: start with the smallest tableau and move, as
  slowly as possible, all the arrows from bottom to top and from left
  to right.
\end{proof}

For the extra slow Tamari lattice, \Cref{pro:counting_spine} and
\Cref{thm:counting_spine} give a recursive formula for the number of
fb-tableaux of size \( n \) in the spine of \( \esTam_n \). We now
consider the slow Tamari lattice. The same approach applies in this
setting as well, but the recurrence becomes considerably simpler.

\begin{lemma}\label{lem:4}
  A small fb-tableau is on the spine in \( \sTam_n \) if and only if
  it does not contain a forbidden cohook.
\end{lemma}

\begin{proposition}\label{pro:1}
  The elements of the spine in \( \sTam_n \) is given by
  \( \sum_{ i + j = n+1 } f_{i,j} \), where \( f_{i,j} \) for
  \( i,j \geq 1 \), satisfies the recurrence
  \begin{equation}\label{eq:sTam_spine_rec}
    f_{i,j} = i f_{i,j-1} + j f_{i-1,j} - (i-1)(j-1) f_{i-1,j-1}, \quad \text{for \( i,j > 1 \)},
  \end{equation}
  with \( f_{i,1} = f_{1,j} = 1 \) for \( i,j > 0 \).
\end{proposition}
\begin{proof}
  The proof follows the same decomposition as in
  \Cref{pro:counting_spine}, but in the slow Tamari case the situation
  is simpler since dots are no longer allowed. As a result, only the
  analogue of the case \( j=0 \) remains.

  More precisely, let \( g(i,0,k) \) denote the number of possible
  fillings of the remaining region after removing the left and right
  part, exactly in the proof of \Cref{pro:counting_spine}. Then this
  gives
  \[
    g(i,0,k) = (i+1) g(i,0,k-1) + (k+1) g(i-1,0,k) - ik \cdot g(i-1,0,k-1),
  \]
  with initial conditions \( g(k,0,0) = g(0,0,k) = 1 \) for
  \( k \geq 0 \).

  Setting \( f_{i,j} = g(i-1,0,j-1) \), we obtain
  \eqref{eq:sTam_spine_rec}. Summing over all possible decompositions
  of the border completes the proof.
\end{proof}

\begin{remark}\label{rem:sTam_spine}
  Interestingly, recursion~\eqref{eq:sTam_spine_rec} coincides with
  the recursion in \cite[Equation~(1)]{EW2009}. As stated in
  \cite[Theorem~6.8]{EW2009}, from the recursion, we also obtain the
  two-variable generating function
  \( F(u,v) = \sum_{ i,j \geq 1 } f_{i,j}u^i v^j \) for numbers
  \( f_{i,j} \) that satisfies
  \[
    F(u,v) = uv + uv \frac{\partial}{\partial u}F(u,v) + uv \frac{\partial}{\partial v}F(u,v)- u^2 v^2 \frac{\partial^2}{\partial u \partial v}F(u,v).
  \]
\end{remark}

\begin{lemma}\label{lem:1}
  The lattice \( \sTam_n \) is polygonal.
\end{lemma}
\begin{proof}
  The proof is similar to that of \Cref{prop:polygonal}. Two types of
  polygons in \( \sTam_n \) are quadrilaterals and hexagons
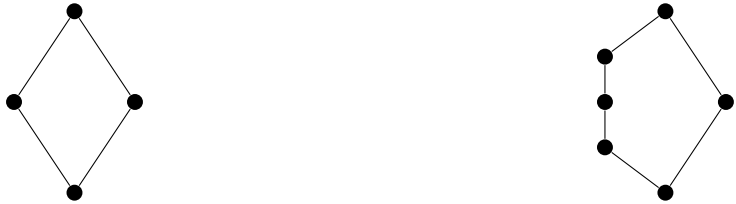
\[
  \begin{minipage}[t]{0.45\textwidth}
    \centering
    \begin{tikzpicture}[scale=.4]
      \tikzset{
        dot/.style={circle, fill=black, inner sep=0pt, minimum size=6pt},
        elabel/.style={midway, font=\scriptsize, inner sep=3pt} 
      }
      \def\h{6}
      \def\w{2}
      \node[dot] (x) at (0,0) {}; 
      \node[dot] (u) at (-\w,\h/2) {};
      \node[dot] (y) at (0,\h) {};
      \node[dot] (v) at (\w,\h/2) {};
      \draw (x) -- (u) -- (y);
      \draw (x) -- (v) -- (y);
    \end{tikzpicture}
  \end{minipage}
  \hfill
  \begin{minipage}[t]{0.45\textwidth}
    \centering
    \begin{tikzpicture}[scale=.4]
      \tikzset{
        dot/.style={circle, fill=black, inner sep=0pt, minimum size=6pt},
        elabel/.style={midway, font=\scriptsize, inner sep=3pt} 
      }
      \def\h{6}
      \def\w{2}
      \node[dot] (x) at (0,0) {}; 
      \node[dot] (u1) at (-\w,\h/4) {};
      \node[dot] (u2) at (-\w,2*\h/4) {};
      \node[dot] (u3) at (-\w,3*\h/4) {};
      \node[dot] (y) at (0,\h) {};
      \node[dot] (v) at (\w,\h/2) {};
      \draw (x) -- (u1) -- (u2) -- (u3) -- (y);
      \draw (x) -- (v) -- (y);
    \end{tikzpicture}
  \end{minipage}.
\]
\end{proof}

The join-labelling \( \eta \) defined in
\Cref{sec:polygonality} can also be considered for
\( \sTam_n \). The associated rank function \( r \) is given by the
same definition as in \eqref{eq:rank_func}. Since small tableaux
contain no dots except the root, it follows that for a
join-irreducible \( (i,j,\sigma) \) of \( \sTam_n \), we have
\( r(i,j,\sigma) = j-i \).

\begin{lemma}\label{lem:2}
  The join labeling \( \eta \) of \( \sTam_n \) together with its
  rank function \( r \), provides a polygonal labeling of
  \( \sTam_n \). Moreover, \( \sTam_n \) is congruence uniform.
\end{lemma}
\begin{remark}
  Since \(\sTam_n\) is a sublattice of \(\esTam_n\), it follows from
  \cite[Theorem 4.3]{day1979} that \(\sTam_n\) is congruence uniform.
  Therefore, the polygonal labeling is not strictly necessary here.
  However, it is helpful for understanding the forcing order.
\end{remark}
We now turn to the congruence lattice of \( \sTam_n \). For
\( \esTam_n \), we obtained the forcing relations among
join-irreducible elements from the heptagon, described the cover
relations in the poset \( F(\esTam_n) \), and then used this
description to express the number of elements in the congruence
lattice of \( \esTam_n \) in terms of weighted Dyck paths of length
\( 2n \). By the same method, in the case of \( \sTam_n \), the
corresponding forcing relations arise from the following hexagon:
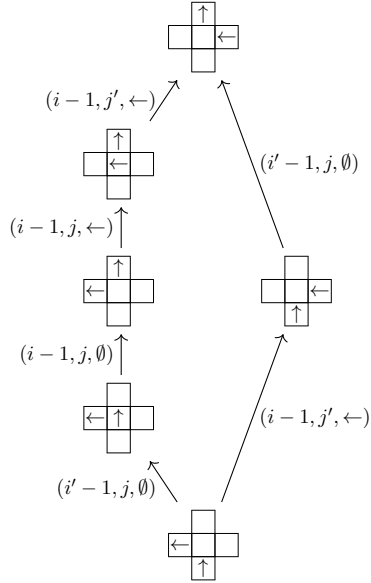
\tikzset{ poset node/.style={inner sep=4pt, align=center},
  lab/.style={midway, fill=white, inner sep=1pt, font=\scriptsize} }
\ytableausetup{boxsize = 1.1em}
\[
  \begin{tikzpicture}[scale=.7, transform shape]
    \def\s{.8}
    \node[poset node] (S) at (0,0) {
      $\begin{ytableau}
        \none{}& &\none{}\\
        \leftarrow & & \\
        \none[] &\uparrow  \\
      \end{ytableau}$
    };
    \node[poset node] (T) at (0,12*\s) {
      $\begin{ytableau}
        \none{}& \uparrow&\none{}\\
        &&\leftarrow \\
        \none[] &  \end{ytableau}$
    };

    \node[poset node] (U1) at (-2*\s,3*\s) {
      $\begin{ytableau}
        \none{}& &\none{}\\
        \leftarrow &\uparrow & \\
        \none[] & \\
      \end{ytableau}$
    };
    \node[poset node] (U2) at (-2*\s,6*\s) {
      $\begin{ytableau}
        \none{}& \uparrow&\none{}\\
        \leftarrow & & \\
        \none[] & \\
      \end{ytableau}$
    };
    \node[poset node] (U4) at (-2*\s,9*\s) {
      $\begin{ytableau}
        \none{}& \uparrow&\none{}\\
        &\leftarrow  & \\
        \none[] & \\
      \end{ytableau}$
    };

    \node[poset node] (V)  at (2.2*\s,6*\s) {
      $\begin{ytableau}
        \none{}& &\none{}\\
        &&\leftarrow\\
        \none[] &\uparrow \\
      \end{ytableau}$
    };

    \draw[->] (S)  -- node[left,yshift=-4pt] {\((i'-1,j,\emptyset)\)} (U1);
    \draw[->] (U1) -- node[left] {\((i-1,j,\emptyset)\)} (U2);
    \draw[->] (U2) -- node[left] {\((i-1,j,\leftarrow)\)} (U4);
    \draw[->] (U4) -- node[left] {\((i-1,j',\leftarrow)\)} (T);

    \draw[->] (S) -- node[right] {\((i-1,j',\leftarrow)\)} (V);
    \draw[->] (V) -- node[right] {\((i'-1,j,\emptyset)\)} (T);

  \end{tikzpicture}
\]
and we likewise obtain the cover relations in the poset
\( F(\sTam_n) \) of join-irreducible elements, as stated below.

\begin{lemma}\label{lem:3}
  The cover relations of the poset \( F(\sTam_n) \) are as follows.
  For \( 1 \leq i, j \leq n-1 \) and \( j-i > 1 \),
\[
  (i,j,\leftarrow) \lessdot (i+1,j,\emptyset) \qand (i,j,\emptyset) \lessdot (i+1,j,\emptyset),
\]
and
\[
  (i,j,\leftarrow) \lessdot (i,j-1,\leftarrow) \qand (i,j,\emptyset) \lessdot (i,j-1,\leftarrow).
\]
For the case when \( j-i = 1 \), we can write cover relations as
\[
  (i-1,i,\sigma) \lessdot (i,i,\leftarrow) \qand (i,i+1,\sigma) \lessdot (i,i,\leftarrow),
\]
for any \( \sigma \in \{\leftarrow, \emptyset\} \) and
\( 1 \leq i \leq n-1 \). See~\Cref{fig:sTam_J_irr}~(Left).
\end{lemma}

\ytableausetup{boxsize = 1.2 em}
\begin{figure}
  \centering
  \begin{minipage}[t]{0.48\textwidth}
    \centering
    \begin{tikzpicture}[>=latex, scale=.65, transform shape]
      \def\h{3.5}
      \node (P11) at (-1,0) {
        $\begin{ytableau}
          \bullet&*(yellow)\uparrow&& \\
          &*(yellow)\leftarrow&&\uparrow \\
          \leftarrow&&\uparrow \\
          &\leftarrow
        \end{ytableau}$
      };
      \node (P13) at (1,0) {
        $\begin{ytableau}
          \bullet&*(yellow)\uparrow&& \\
          \leftarrow&*(yellow)&&\uparrow \\
          \leftarrow&&\uparrow \\
          &\leftarrow
        \end{ytableau}$
      };
      \node (P21) at (-3.5,\h) {
        $\begin{ytableau}
          \bullet&&& \\
          \leftarrow&*(yellow)\uparrow&&\uparrow \\
          &*(yellow)\leftarrow&\uparrow \\
          &\leftarrow
        \end{ytableau}$
      };
      \node (P23) at (-1.5,\h) {
        $\begin{ytableau}
          \bullet&&& \\
          \leftarrow&*(yellow)\uparrow&&\uparrow \\
          \leftarrow&*(yellow)&\uparrow \\
          &\leftarrow
        \end{ytableau}$
      };
      \node (P24) at (1.5,\h) {
        $\begin{ytableau}
          \bullet&&*(yellow)\uparrow& \\
          &&*(yellow)\leftarrow&\uparrow \\
          &&\leftarrow \\
          \leftarrow&\uparrow
        \end{ytableau}$
      };
      \node (P26)  at (3.5,\h) {
        $\begin{ytableau}
          \bullet&&*(yellow)\uparrow& \\
          \leftarrow&&*(yellow)&\uparrow \\
          &&\leftarrow \\
          \leftarrow&\uparrow
        \end{ytableau}$
      };
      \node (P31)  at (-5,2*\h) {
        $\begin{ytableau}
          \bullet&&& \\
          \leftarrow&&&\uparrow \\
          \leftarrow&*(yellow)\uparrow&\uparrow \\
          &*(yellow)\leftarrow
        \end{ytableau}$
      };
      \node (P33)  at (0,2*\h) {
        $\begin{ytableau}
          \bullet&&& \\
          \leftarrow&&*(yellow)\uparrow&\uparrow \\
          &&*(yellow)\leftarrow \\
          \leftarrow&\uparrow
        \end{ytableau}$
      };
      \node (P35)  at (5,2*\h) {
        $\begin{ytableau}
          \bullet&&&*(yellow)\uparrow \\
          &&&*(yellow)\leftarrow \\
          \leftarrow&&\uparrow \\
          \leftarrow&\uparrow
        \end{ytableau}$
      };

      \draw (P11.north)  -- node {} (P23.south);
      \draw (P13.north)  -- node {} (P23.south);
      \draw (P11.north)  -- node {} (P24.south);
      \draw (P13.north)  -- node {} (P24.south);

      \draw (P21.north)  -- node {} (P31.south);
      \draw (P23.north)  -- node {} (P31.south);
      \draw (P21.north)  -- node {} (P33.south);
      \draw (P23.north)  -- node {} (P33.south);

      \draw (P24.north)  -- node {} (P33.south);
      \draw (P26.north)  -- node {} (P33.south);
      \draw (P24.north)  -- node {} (P35.south);
      \draw (P26.north)  -- node {} (P35.south);
    \end{tikzpicture}
  \end{minipage}\hfill
  \begin{minipage}[t]{0.48\textwidth}
    \centering
    \begin{tikzpicture}[
      thicknode/.style={
        circle,          
        minimum size=6pt,
        inner sep=0pt    
      },
      rednode/.style={thicknode, fill=red},
      bluenode/.style={thicknode, fill=blue},
      greennode/.style={thicknode, fill=green!40!black},
      scale = .7
      ]
      \node[rednode] (P11) at (-4,0) {};
      \node[rednode] (P13) at (-2,0) {};
      \node[rednode] (P15) at (0,0) {};
      \node[rednode] (P17) at (2,0) {};
      \node[rednode] (P19) at (4,0) {};

      \node[rednode] (P21) at (-3-.2,1) {};
      \node[greennode] (P23) at (-3+.2,1) {};
      \node[rednode] (P24) at (-1-.2,1) {};
      \node[greennode] (P26) at (-1+.2,1) {};
      \node[rednode] (P27) at (1-.2,1) {};
      \node[greennode] (P29) at (1+.2,1) {};
      \node[rednode] (P210) at (3-.2,1) {};
      \node[greennode] (P212) at (3+.2,1) {};

      \node[rednode] (P31) at (-2-.2,2) {};
      \node[greennode] (P33) at (-2+.2,2) {};
      \node[rednode] (P34) at (0-.2,2) {};
      \node[greennode] (P36) at (0+.2,2) {};
      \node[rednode] (P37) at (2-.2,2) {};
      \node[greennode] (P39) at (2+.2,2) {};

      \node[rednode] (P41) at (-1-.2,3) {};
      \node[greennode] (P43) at (-1+.2,3) {};
      \node[rednode] (P44) at (1-.2,3) {};
      \node[greennode] (P46) at (1+.2,3) {};

      \node[rednode] (P51) at (0-.2,4) {};
      \node[greennode] (P53) at (0+.2,4) {};

      \draw (P11)  -- node {} (P21);  \draw (P11)  -- node {} (P23);
      \draw (P13)  -- node {} (P21);  \draw (P13)  -- node {} (P23);
      \draw (P13)  -- node {} (P24);  \draw (P13)  -- node {} (P26);
      \draw (P15)  -- node {} (P24);  \draw (P15)  -- node {} (P26);
      \draw (P15)  -- node {} (P27);  \draw (P15)  -- node {} (P29);
      \draw (P17)  -- node {} (P27);  \draw (P17)  -- node {} (P29);
      \draw (P17)  -- node {} (P210);  \draw (P17)  -- node {} (P212);
      \draw (P19)  -- node {} (P210);  \draw (P19)  -- node {} (P212);

      \draw (P23)  -- node {} (P31);  \draw (P23)  -- node {} (P33);
      \draw (P24)  -- node {} (P31);  \draw (P24)  -- node {} (P33);
      \draw (P26)  -- node {} (P34);  \draw (P26)  -- node {} (P36);
      \draw (P27)  -- node {} (P34);  \draw (P27)  -- node {} (P36);
      \draw (P29)  -- node {} (P37);  \draw (P29)  -- node {} (P39);
      \draw (P210)  -- node {} (P37);  \draw (P210)  -- node {} (P39);

      \draw (P33)  -- node {} (P41);  \draw (P33)  -- node {} (P43);
      \draw (P34)  -- node {} (P41);  \draw (P34)  -- node {} (P43);
      \draw (P36)  -- node {} (P44);  \draw (P36)  -- node {} (P46);
      \draw (P37)  -- node {} (P44);  \draw (P37)  -- node {} (P46);

      \draw (P43)  -- node {} (P51);  \draw (P43)  -- node {} (P53);
      \draw (P44)  -- node {} (P51);  \draw (P44)  -- node {} (P53);

      \fill[black,opacity=0.3] (P11) circle [radius=3mm];
      \fill[black,opacity=0.3] (P13) circle [radius=3mm];
      \fill[black,opacity=0.3] (P15) circle [radius=3mm];
      \fill[black,opacity=0.3] (P17) circle [radius=3mm];
      \fill[black,opacity=0.3] (P19) circle [radius=3mm];
      \fill[black,opacity=0.3] (P26) circle [radius=3mm];
      \fill[black,opacity=0.3] (P27) circle [radius=3mm];
      \fill[black,opacity=0.3] (P29) circle [radius=3mm];
      \fill[black,opacity=0.3] (P210) circle [radius=3mm];
      \fill[black,opacity=0.3] (P34) circle [radius=3mm];
      \fill[black,opacity=0.3] (P37) circle [radius=3mm];
      \fill[black,opacity=0.3] (P39) circle [radius=3mm];

      \draw[very thick] (-6,-1) -- (-4,1) -- (-3,0) -- (0,3) -- (1,2) -- (2,3) -- (6,-1);
      \draw[gray] (-6,-1) -- (0,5);
      \draw[gray] (-4,-1) -- (1,4);
      \draw[gray] (-2,-1) -- (2,3);
      \draw[gray] (0,-1) -- (3,2);
      \draw[gray] (2,-1) -- (4,1);
      \draw[gray] (4,-1) -- (5,0);
      \draw[gray] (6,-1) -- (0,5);
      \draw[gray] (4,-1) -- (-1,4);
      \draw[gray] (2,-1) -- (-2,3);
      \draw[gray] (0,-1) -- (-3,2);
      \draw[gray] (-2,-1) -- (-4,1);
      \draw[gray] (-4,-1) -- (-5,0);
    \end{tikzpicture}
  \end{minipage}

  \caption{(Left) The poset \( F(\sTam_n) \) for \( n = 4 \); (Right)
    A simplified representation of the dual of \( F(\sTam_n) \) for
    \( n=6 \) with its order ideal and the corresponding Dyck path.}
  \label{fig:sTam_J_irr}
\end{figure}
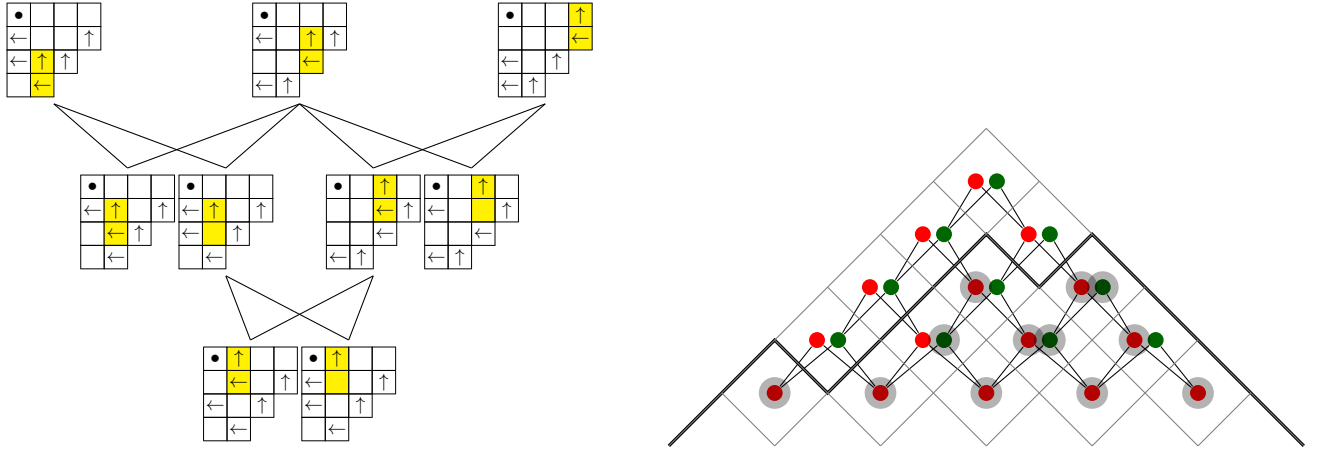

Moreover, the number of elements in the congruence lattice of
\( \sTam_n \) can also be expressed as a weighted sum over Dyck paths.
See~\Cref{fig:sTam_J_irr}~(Right). In this setting the model is
simpler, which allows us to derive an explicit generating function as
well.

\begin{proposition}\label{pro:2}
  We define the weight of a Dyck path as the product of the weights of
  all the cells below the path, where the weight \( \wt(c) \) of each
  cell \( c \) is given by
  \[
    \wt(c) = 
    \begin{cases}
     1 & \text{if \( c \) is any cell of height \( 1 \)},\\
     3 & \text{if \( c \) is a corner cell of height \( > 1 \)}, \\
     2 & \text{if \( c \) is a boundary cell of height \( > 1 \)},\\
     1 & \text{if \( c \) is an inner cell of height \( > 1 \)}.
    \end{cases}
  \]
  The number of elements in the congruence lattice of \( \sTam_n \)
  equals the total weight sum of all Dyck paths from \( (0,0) \) to
  \( (2n,0) \).
\end{proposition}

We derive a generating function for the number of elements in the
congruence lattice of \( \sTam_n \).

\begin{proposition}\label{the:1}
  Let \( f(n) \) be the number of elements in the congruence lattice
  of \( \sTam_n \). Then, the generating function for \( f(n) \) is
  given by
\begin{equation}\label{eq:gf_cong_lattice_slow1}
  F(x) := \sum_{ n \geq 0 } f(n) x^n = \cfrac{1}{1 - \cfrac{x}{1 - x N\left(\frac{3}{4},4x\right)}},
\end{equation}
where \( N(\ell,x) \) is the \defn{Narayana polynomial} defined by
\[
  N(\ell,x) := \frac{1 + x - \ell x - \sqrt{1-2(\ell +1)x + (\ell -1)^2 x^2}}{2x}.
\]
Equivalently,
\begin{equation}\label{eq:gf_cong_lattice_slow2}
  F(x) = \cfrac{1}{1 - \cfrac{x}{1 - \cfrac{x}{1+x - \cfrac{4x}{1+x- \cfrac{4x}{1+x - \cfrac{4x}{1+x- \cdots}}}} }}.
\end{equation}

\end{proposition}
\begin{proof}
  Since every inner cell has weight \( 1 \), we can re-express the
  weight of a Dyck path in terms of down steps and peaks. For a Dyck
  path \( p \), define the weight \( \wt_s(p) \) to be the product of
  the weights of all steps and peaks of \( p \), with the following
  assignment:
  \begin{itemize}
  \item every up step has weight \( 1 \);
  \item a down step has weight \( 1 \) if it starts at height
    \( \leq s \), and weight \( 4 \) if it starts at height \( >s \);
  \item a peak has weight \( 1 \) if it is at height \( \leq s \), and
    weight \( \frac{3}{4} \) if it is at height \( >s \).
  \end{itemize}
  Let \( f_s(n) \) be the weight sum of all Dyck paths from
  \( (0,0) \) to \( (2n,0) \) with the weight function \( \wt_s \),
  and let
  \[
    F_s(x) = \sum_{ n \geq 0 } f_s(n) x^n.  
  \]

  Then, the weight \( \wt(p) \) in \Cref{pro:2} is equal to
  \( \wt_2(p) \); this reformulation can be understood as follows. The
  two boundary cell contributions of weight \( 2 \) at height \( >1 \)
  combine into a single factor \( 4 \), which we assign to the
  corresponding down step. A corner cell of weight \( 3 \) at height
  \( >1 \) can be regarded as the product of the weights of the
  adjacent peak and down step, namely \( (3/4)\cdot 4 = 3 \). Hence,
  \( F(x) = F_2(x) \).

  Every Dyck path is either the path of length \( 0 \) (one point) or
  begins with an up step. Moreover, any nonempty Dyck path \( p \)
  admits the standard first-return decomposition of the form
  \( p = Up_1Dp_2 \), where \( U \) and \( D \) denote an up step and
  a down step, and \( p_1, p_2 \) are Dyck paths. Under this
  decomposition, the subpath \( p_1 \) begins with height \( 1 \), hence
  we have \( \wt_s(p) = \wt_{s-1}(p_1) \wt_s(p_2) \). It follows that
  \[
    F(x) = F_2(x) = 1 + x F_1(x)F_2(x).
  \]

  Applying the same first-return decomposition gives
  \[
    F_1(x) = 1 + x F_0(x) F_1(x).
  \]
  Combining the two equations above, we have
  \begin{equation}\label{eq:sTam_cong1}
    F(x) = \frac{1}{1-xF_1(x)} = \frac{1}{1-\frac{x}{1-x F_0(x)}}.
  \end{equation}
  The Narayana polynomial \( N(\ell,x) \) is known to be the
  generating function for Dyck paths in which each peak has weight
  \( \ell \) and each up and down step has weight \( 1 \). Hence, we
  have
  \begin{equation}\label{eq:sTam_cong2}
    F_0(x) = N\left(\frac{3}{4}, 4x\right).
  \end{equation}
  By \eqref{eq:sTam_cong1} and \eqref{eq:sTam_cong2}, we
  obtain~\eqref{eq:gf_cong_lattice_slow1}.

  \Cref{eq:gf_cong_lattice_slow2} can be obtained by using a similar path decomposition. 
  Indeed, we obtains
  \[
    F_0(x) = 1 + 3x F_0(x) + 4x (F_0(x)-1)F_0(x),
  \]
  and hence
  \[
    F_0(x) = \frac{1}{1-x + 4x F_0(x)}.
  \]
  Together with~\eqref{eq:sTam_cong1}, the claim follows.
\end{proof}

\bibliographystyle{alpha}
\bibliography{sample.bib}

\end{document}